\documentclass[12pt]{article}
\usepackage{amsmath}
\usepackage{amssymb}
\usepackage{amsthm}
\usepackage{latexsym}
\usepackage{color,fancyhdr,lastpage,graphicx}
\usepackage{subfigure, epsfig, epsf}
\usepackage{xspace}
\usepackage{setspace}
\usepackage{bbm}
\usepackage{esint,stmaryrd}
\usepackage{mathrsfs} 

\theoremstyle{plain}

\newtheorem{lemma}{Lemma}[section]

\newtheorem{theorem}[lemma]{Theorem}
\newtheorem{proposition}[lemma]{Proposition}

\theoremstyle{remark}

\definecolor{blue}{rgb}{0.0,0.02,0.67}
\definecolor{brown}{rgb}{0.7,0.3,0.0}
\def\bb{\begin{color}{blue}}
\def\bw{\begin{color}{white}}
\def\bg{\begin{color}{green}}
\def\br{\begin{color}{red}}
\def\blb{\begin{color}{cyan}}
\def\bbr{\begin{color}{brown}}

\def\eg{\end{color}}
\def\ew{\end{color}}
\def\er{\end{color}}
\def\eb{\end{color}}
\def\elb{\end{color}}
\def\ebr{\end{color}}
\def\jj{\j}
\def\jj{}

\def\vfe{\vfill\eject}
\def\vfe{}
\def\PP{\mathbb{P}}

\def\H{\mathbb{H}}

\def\ve{{\varepsilon}}
\def\le{\leqslant}

\def\pd{\partial}
\def\tn{\interleave}

\def\ge{\geqslant}

\def\widebar{\overline}

\def\E{{\mathbb E}}
\def\O{{\Omega}}

\def\R{{\mathbb R}}
\def\C{{\mathbb C}}
\def\N{{\mathbb N}}

\def\Z{{\mathbb Z}}

\def\a{{\alpha}}
\def\b{{\beta}}
\def\d{{\delta}}

\def\t{{\tau}}
\def\k{{\kappa}}
\def\g{{\gamma}}
\def\G{{\Gamma}}
\def\s{{\sigma}}
\def\l{{\lambda}}
\def\L{{\Lambda}}
\def\th{{\theta}}
\def\Th{{\Theta}}
\def\o{{\omega}}
\def\z{{\zeta}}

\def\cS{{\cal S}}

\def\cT{{\cal T}}

\def\cH{{\cal H}}
\def\cV{{\cal V}}

\def\bw{{\text{\boldmath$w$}}}

\def\q{\quad}

\def\re{\operatorname{Re}}
\def\im{\operatorname{Im}}
\def\cp{{\operatorname{cap}}}

\def\supp{\operatorname{supp}}

\def\disc{{\operatorname{disc}}}
\def\slit{{\operatorname{slit}}}
\def\bump{{\operatorname{bump}}}
\def\disk{{\operatorname{disk}}}

\def\<{\langle}
\def\>{\rangle}
\def\ra{\rightarrow}

\def\sse{\subseteq}
\def\sm{\setminus}

\begin{document}
\bibliographystyle{plain}


\begin{center}
\LARGE \textbf{Stability of regularized Hastings-Levitov aggregation in the subcritical regime}

\vspace{0.2in}

\large {\bfseries 
James Norris\footnote{Statistical Laboratory, Centre for Mathematical Sciences, Wilberforce Road, Cambridge, CB3 0WB, UK. 
 Email:  j.r.norris@statslab.cam.ac.uk}, 
Vittoria Silvestri\footnote{University of Rome La Sapienza, Piazzale Aldo Moro 5, 00185, Rome, Italy. \\ 
Email: silvestri@mat.uniroma1.it}, 
Amanda Turner\footnote{Department of Mathematics and Statistics, Lancaster University, Lancaster LA1 4YF, UK. \\
Email: a.g.turner@lancaster.ac.uk.
Research supported by EPSRC grant EP/T027940/1.}
}


\end{center}
\begin{abstract} 
We prove bulk scaling limits and fluctuation scaling limits for a two-parameter class ALE$(\a,\eta)$ 
of continuum planar aggregation models.
The class includes regularized versions of the Hastings--Levitov family HL$(\a)$ 
and continuum versions of the family of dielectric-breakdown models, 
where the local attachment intensity for new particles is specified as a negative power $-\eta$ of the density 
of arc length with respect to harmonic measure.
The limit dynamics follow solutions of a certain Loewner--Kufarev equation, 
where the driving measure is made to depend on the solution and on the parameter $\z=\a+\eta$.
Our results are subject to a subcriticality condition $\z\le1$: this includes HL$(\a)$ for $\a\le1$
and also the case $\a=2,\eta=-1$ corresponding to a continuum Eden model.
Hastings and Levitov predicted a change in behaviour for HL$(\a)$ at $\a=1$,
consistent with our results.
In the regularized regime considered, the fluctuations around the scaling limit are shown to be Gaussian, with
independent Ornstein--Uhlenbeck processes driving each Fourier mode, which are seen to be stable if and only if $\z\le1$.
\end{abstract}

\setcounter{tocdepth}{2} 
\tableofcontents


\section{Introduction}
\subsection{Hastings--Levitov aggregation}
\label{INT}
In many physical contexts there appear clusters whose shape is complex, formed apparently by some mechanism of random growth.
It has long been a challenge to account for the observed variety of complex cluster shapes, starting from
plausible physical principles governing the aggregation of individual microscopic particles.
For clusters which are essentially two-dimensional, 
there is an approach introduced by Carleson and Makarov \cite{C+M} and Hastings and Levitov \cite{HL},
in which clusters are encoded as a composition of conformal maps, one for each particle. 
In this approach, a growing cluster is modelled by an increasing sequence of compact sets $K_n\sse\C$ which are assumed to be simply connected.
We will take the initial set $K_0$ to be the closed unit disk $\{|z|\le1\}$.
The increments $K_n\sm K_{n-1}$ are then thought of as a sequence of particles added to the cluster.
The idea is to study the clusters $K_n$ via the conformal isomorphisms 
$$
\Phi_n:D_0\to D_n
$$
where $D_n$ is the complementary domain $\C\sm K_n$ and $\Phi_n$ is normalized by $\Phi_n(\infty)=\infty$ and $\Phi_n'(\infty)>0$.
Then $\Phi_0(z)=z$ for all $z$ and $K_n$ has logarithmic capacity $\Phi_n'(\infty)>1$ for all $n\ge1$.
This formulation is convenient because the harmonic measure from $\infty$ on the boundary $\pd D_n$,
which provides a natural way to choose the location of the next particle,
is then simply the image under $\Phi_n$ of the uniform distribution on $\pd D_0=\{|z|=1\}$.
Having chosen a random angle $\Th_{n+1}$ to locate the next particle, 
and a model particle $P_{n+1}$ attached to $K_0$ at $e^{i\Th_{n+1}}$, for example a small disk tangent to $K_0$, 
the cluster map is updated to
$$
\Phi_{n+1}=\Phi_n\circ F_{n+1}
$$
where $F_{n+1}$ is the conformal isomorphism $D_0\to D_0\sm P_{n+1}$, normalized similarly to $\Phi_n$.
Then $\Phi_{n+1}$ encodes the cluster
$$
K_{n+1}=K_n\cup\Phi_n(P_{n+1}).
$$
Thus, once we specify distributions for the angles $\Th_n$ and model particles $P_n$,
we have specified a mechanism to grow a random cluster.

We will write
$$
\cp(K_n)=\log\Phi'(\infty),\q c_n=\log F_n'(\infty)
$$
and we will refer to $\cp(K_n)$ as the capacity\footnote{This is an abuse of terminology since it is then $e^{\cp(K_n)}$ which is
the logarithmic capacity.}
of $K_n$ and $c_n$ as the capacity of $P_n$.
Then
$$
\cp(K_n)=c_1+\dots+c_n.
$$
We will be looking for scaling limits where the particle capacities $c_n$ and the associated particles $P_n$ become small,
but where $n$ is chosen sufficiently large that the cluster capacities $\cp(K_n)$ grow macroscopically.

A simple case is to choose $\Th_{n+1}$ uniformly distributed on the unit circle and to take $P_{n+1}=e^{i\Th_{n+1}}P$,
where $P$ is a small disk tangent to the unit disk at $1$, of radius $r(c)$, chosen so that $P$ has capacity $c$.
Then in fact $r(c)/\sqrt c$ has a positive limit as $c\to0$.
The location of the new particle $\Phi_n(P_{n+1})$ is then distributed according to harmonic measure on $\pd K_n$.
However, if we assume that $\pd K_n$ is approximately linear on the scale of $P$, then we would have
\begin{equation}
\label{ADISK}
\Phi_n(P_{n+1})\approx\Phi_n(e^{i\Th_{n+1}})+\Phi_n'(e^{i\Th_{n+1}})P
\end{equation}
so we would add an approximate disk of diameter proportional to $\sqrt c|\Phi_n'(e^{i\Th_{n+1}})|$.

In order to compensate for this distortion, Hastings and Levitov proposed the HL$(\a)$ family of models where,
once $\Th_{n+1}$ is chosen, we choose $P_{n+1}$ to be a particle of capacity
$$
c_{n+1}=|\Phi_n'(e^{i\Th_{n+1}})|^{-\a}c.
$$
Then, in the case $\a=2$, the particles added to the cluster would be approximately of constant size.
The approximation \eqref{ADISK} is in fact misleading, at least on a microscopic level, 
because $\pd K_n$ develops inhomogeneities on the scale of the particles.
Nevertheless, HL$(2)$ has been considered as a variant of diffusion-limited aggregation (DLA) \cite{W+S},
with some justification, see \cite{HL}, derived from numerical experiments.

In general, the HL$(\a)$ model offers a convenient mechanism for such experiments, 
and moves away from the lattice formulation of \cite{W+S} which has been shown to lead to unphysical effects on large scales.
See for example \cite{MR3819427}.
Moreover, it might be hoped that an evolving family of conformal maps would present a more tractable framework for 
the analysis of scaling limits than other growth models, 
while potentially sharing the same bulk scaling limit and fluctuation universality class.
That is the direction explored in this paper.

Besides the mechanism of diffusive aggregation, based on harmonic measure, there is another one-parameter family of models,
conceived originally in the lattice case, called dielectric breakdown models \cite{NPW}, 
which interpolates between DLA and the Eden model \cite{Eden}.
In the Eden model, each boundary site is chosen with equal probability.  
In the continuum setting, for an Eden-type model we would choose an attachment point on the boundary according to normalized arc length, 
which has density proportional to $|\Phi_n'(e^{i\th})|$ with respect to harmonic measure.
We can widen our family of models to include a continuum analogue of dielectric breakdown models by choosing
$$
\PP(\Th_{n+1}\in d\th|\Phi_n)\propto|\Phi_n'(e^{i\th})|^{-\eta}d\th.
$$
The case $\eta=-1$ then provides a continuum variant of the Eden model.

In a law-of-large-numbers regime, 
it might be guessed that bulk characteristics of the cluster for the model incorporating both the $\a$ and $\eta$ modifications 
would depend only on their sum $\z=\a+\eta$ since, 
once this is fixed, up to a global time-scaling,
the growth rate of capacity due to particles attached near $e^{i\th}$ does not depend further on $\a$ or $\eta$. 
We will show, in the regime which we can address, that this is indeed true.

In this paper we investigate the two-parameter family of models just described,
but modified by the introduction of a regularization parameter $\s>0$, which controls the minimum length scale over which
feedback occurs through $c_{n+1}$ and $\Th_{n+1}$.
We will require throughout that $\s\gg\sqrt c$ (and sometimes $\s\gg c^{1/4}$, or more) 
and we will restrict attention to the subcritical regime $\z\le1$.
This includes the Eden case ($\a=2$, $\eta=-1$) but excludes continuum DLA ($\a=2$, $\eta=0$).
In the regularized models, we will show fluctuation behaviour which is universal over all choices of particle family.
Our first main result shows that, in this regime, in the limit $c\to0$, disks are stable, that is, 
an initial disk cluster remains close to a disk as particles are added and its capacity becomes large. 
Our second main result is to prove convergence of the normalized fluctuations of the cluster around its deterministic limit,
to an explicit Gaussian process.
The constraint $\z\le1$ appears sharp for this behaviour: 
we see an explicit dependence of the fluctuations on $\a$ and $\eta$ and, in particular,
an exponential instability of rate $(\z-1)k$ in the $k$th Fourier mode if we formally take $\z>1$.

\subsection{Statement of results}
\label{DESM}
In this section, we define the continuous-time ALE$(\a,\eta)$ model, which is our object of study, 
and we specify our standing assumptions for individual particles.
We then state our main results.

Our model is constructed as a composition of univalent functions on the exterior unit disk $D_0=\{|z|>1\}$.
Each of these functions corresponds to a choice of attachment angle $\th\in[0,2\pi)$ and a basic particle $P$.
Recall that $K_0=\{|z|\le1\}$.
By a basic particle $P$ we mean a non-empty subset of $D_0$ such that $K_0\cup P$ is compact and simply connected.
Set $D=D_0\sm P$.
By the Riemann mapping theorem, there is a $c\in(0,\infty)$ and a conformal isomorphism $F:D_0\to D$ with Laurent expansion of the form
\begin{equation}
\label{CAPACITY}
F(z)=e^c\left(z+\sum_{k=0}^\infty a_kz^{-k}\right).
\end{equation} 
Then $F$ is uniquely determined by $P$, and $P$ has capacity $c$.
Our model depends on three parameters $\a,\eta\in\R$ and $\s\in(0,\infty)$, 
along with the choice of a family of basic particles $(P^{(c)}:c\in(0,\infty))$ with $P^{(c)}$ of capacity $c$.
The associated maps $F_c:D_0\to D^{(c)}$ then have the form \eqref{CAPACITY} with $a_k=a_k(c)$ for all $k$.
We assume throughout that $F_c$ extends continuously to $\{|z|\ge1\}$.
We require that our particle family is nested
\begin{equation}
\label{NESTED}
P^{(c_1)}\sse P^{(c_2)}\q\text{for $c_1<c_2$}
\end{equation} 
and satisfies, for some $\L\in[1,\infty)$,
\begin{equation}
\label{CONCENTRATED}
\d(c)\le\L r_0(c)\q\text{for all $c$}
\end{equation} 
where
$$
r_0(c)=\sup\{|z|-1:z\in P^{(c)}\},\q
\d(c)=\sup\{|z-1|:z\in P^{(c)}\}.
$$
In our results, only small values of $c$ are of interest.
For such $c$, the last condition \eqref{CONCENTRATED} forces our particles $P^{(c)}$ to concentrate near the point $1$ 
while never becoming too flat against the unit circle.

The following are all examples of particle families satisfying both conditions \eqref{NESTED} and \eqref{CONCENTRATED}:
$$
P^{(c)}_\slit=(1,1+\d(c)],\q
P^{(c)}_\bump=\{z\in D_0:|z-1|\le\d(c)\}
$$
and
$$
P^{(c)}_\disk=\{z\in D_0:|z-1-r(c)|\le r(c)\},\q r(c)=\d(c)/2
$$
where in each case $\d$ is a suitable increasing homeomorphism of $(0,\infty)$.

It will be convenient to place our aggregation models from the outset in continuous time.
By a (continuous-time) aggregate Loewner evolution of parameters $\a,\eta\in\R$, or ALE$(\a,\eta)$, we mean 
a finite-rate, continuous-time Markov chain $(\Phi_t)_{t\ge0}$ 
taking values in the set of univalent functions $D_0\to D_0$, starting from $\Phi_0(z)=z$, 
which, when in state $\phi$, jumps to $\phi\circ F_{c(\th,\phi),\th}$ at rate $\l(\th,\phi)d\th/(2\pi)$, where
\begin{equation}
\label{MODEL}
F_{c,\th}(z)=e^{i\th}F_c(e^{-i\th}z),\q
c(\th,\phi)=c|\phi'(e^{\s+i\th})|^{-\a},\q
\l(\th,\phi)=c^{-1}|\phi'(e^{\s+i\th})|^{-\eta}.
\end{equation}
Since $\s>0$, the rate $\l(\th,\phi)$ is continuous in $\th$, so the total jump rate is finite.
The model may be thought of equivalently in term of the random process of compact sets $(K_t)_{t\ge0}$ given by
$$
K_0=\{|z|\le1\},\q K_t=K_0\cup\{z\in D_0:z\not\in\Phi_t(D_0)\}.
$$
The effect of the jump just described is then to add to the current cluster the set $\phi(e^{i\th}P^{(c(\th,\phi))})$
thereby increasing its capacity by $c(\th,\phi)$.

\begin{figure}[h!]
  \centering
    \includegraphics[width=0.7\textwidth]{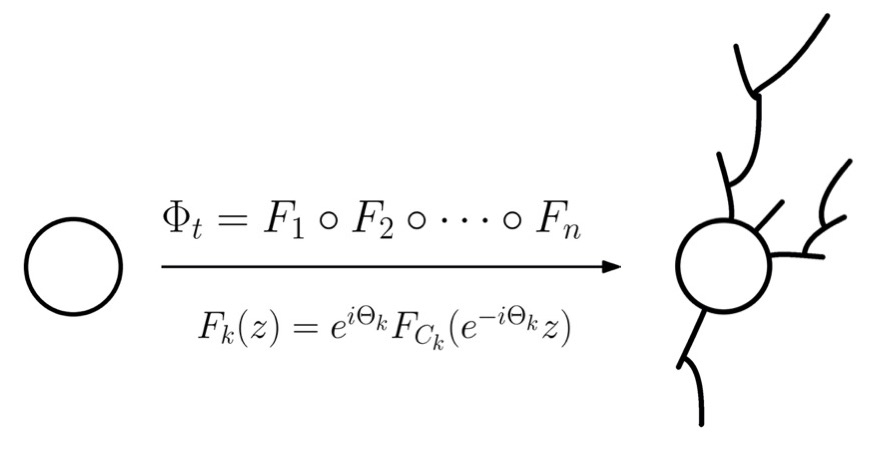}
    \caption{Cluster map with $n$ particles} 
    \label{FIGA}
\end{figure}
In the case where $(\Phi_t)_{t\ge0}$ takes exactly $n$ jumps by time $t$, we have
$$
\Phi_t=F_1\circ\dots\circ F_n,\q
F_n=F_{C_n,\Th_n}
$$
where $C_n$ is the capacity of the $n$th particle and $\Th_n$ is its attachment angle, as in Figure \ref{FIGA}.
Moreover, the capacity $\cT_t$ of the cluster $K_t$ is then given by
$$
\cT_t=\log\Phi_t'(\infty)=C_1+\dots+C_n.
$$
For certain parameter values, the process $(\Phi_t)_{t\ge0}$ may explode, 
that is, may take infinitely many jumps in a finite time interval.
In fact this can happen only if $\cT_t\to\infty$ at the same time, 
and this possibility is excluded (with high probability) over the relevant time interval in the conclusions of our main results.
So we make no attempt to define $\Phi_t$ beyond explosion.%
\footnote{%
The total jump rate $\l(\phi)$ at a state $\phi$ is given by
$$
\l(\phi)=c^{-1}\fint_0^{2\pi}|\phi'(e^{\s+i\th})|^{-\eta}d\th
$$
so, by distortion estimates, there is a constant $C(\eta,\s)<\infty$ such that 
$$
c^{-1}e^{-\eta\t}/C\le\l(\phi)\le Cc^{-1}e^{-\eta\t}
$$
where $\t=\phi'(\infty)$.
Similarly, there is a constant $C(\a,\s)<\infty$ such that the next jump in capacity $\Delta\t$ satisfies
$$
ce^{-\a\t}/C\le\Delta\t\le Cce^{-\a\t}.
$$
These estimates imply by standard arguments that, almost surely, explosion occurs 
if and only if both $\eta<0$ and $\z=\a+\eta<0$, and only if $\cT_t\to\infty$ at the same time.
}%

The discrete-time process $(\Phi_n)_{n\ge0}$ in the introductory discussion is given by the Markov chain
formed of the sequence of distinct values taken by $(\Phi_t)_{t\ge0}$.
We denote this process from now on by $(\Phi^\disc_n)_{n\ge0}$ for clarity.
Prior work on ALE models \cite{NST,STV} was framed in terms of this discrete-time process.
The continuous-time framework allows a more local specification of the dynamics, 
without the need to normalise the distribution of attachment angles. 
It further allows us to organise the computation of martingales in terms of a standard calculus for Poisson random measures.

We can now state our first main result.
Define
$$
t_\z=
\begin{cases}
\infty,&\text{if $\z\ge0$},\\
|\z|^{-1},&\text{if $\z<0$}.
\end{cases}
$$
and for $t<t_\z$ set
$$
\t_t=
\begin{cases}
t,&\text{if $\z=0$},\\
\z^{-1}\log(1+\z t),&\text{otherwise}.
\end{cases}
$$
Note that $\t_t\to\infty$ as $t\to t_\z$ for all $\z$.
\begin{figure}[h!]
  \centering
    \includegraphics[width=0.7\textwidth]{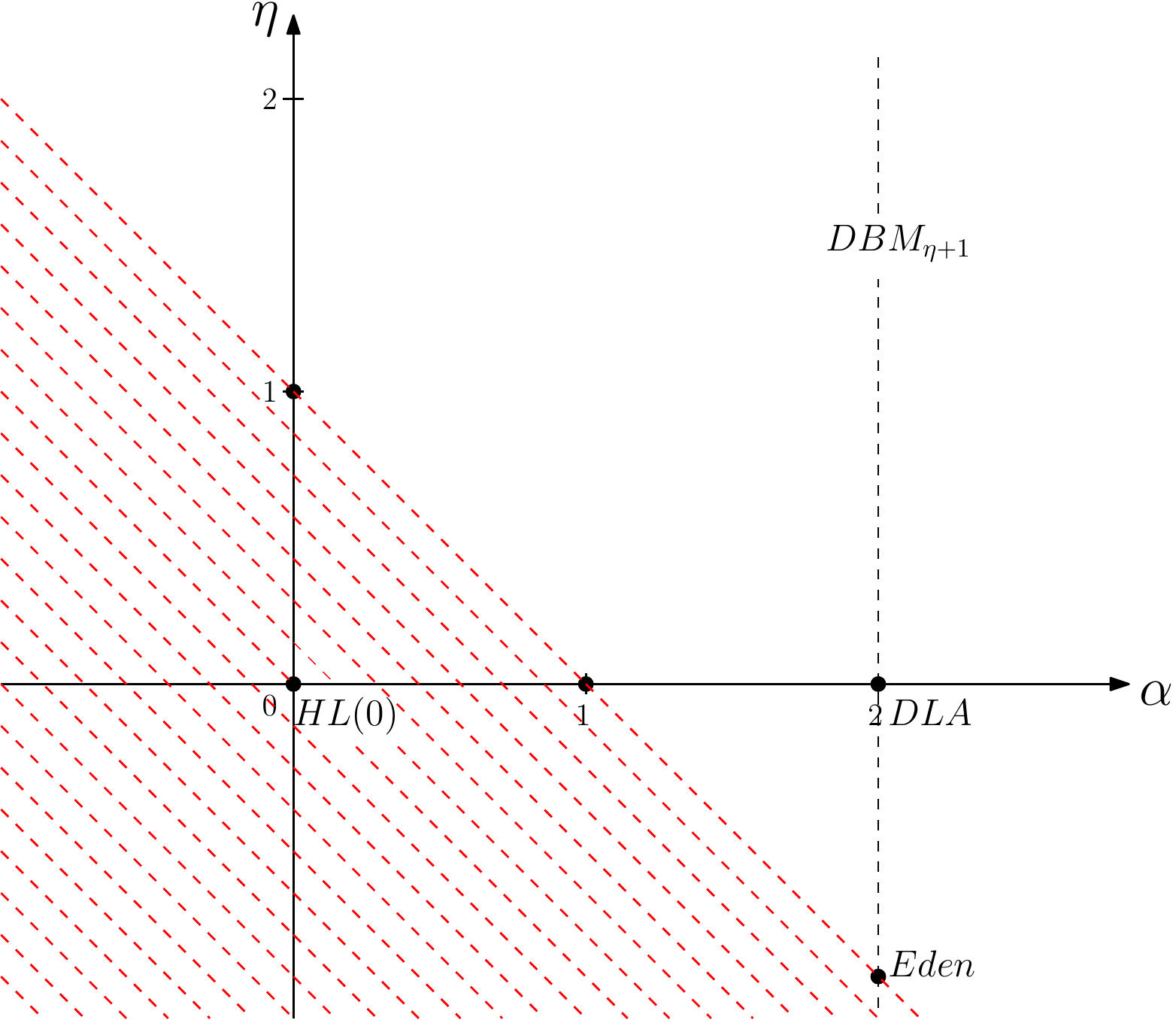}
    \caption{Domain of stability for ALE$(\a,\eta)$} 
    \label{FIG}
\end{figure}
The result identifies the small-particle scaling limit of $K_t$ in the case $\z\le1$
as a disk of radius $e^{\t_t}$, with quantified error estimates.
It is proved in Propositions \ref{UDFA} and \ref{UDFB}.
The range of parameter values to which the result applies is indicated by the region shaded red in Figure \ref{FIG},
with diagonal lines showing parameter pairs $(\a,\eta)$ sharing a common bulk scaling limit.
Recall that $\cT_t=\log\Phi_t'(\infty)$, which is the capacity of $K_t$, and set 
$$
\hat\Phi_t(z)=\Phi_t(z)/e^{\cT_t}.
$$

\begin{theorem}
\label{DISK}
For all $\a,\eta\in\R$ with $\z=\a+\eta\le1$, for all $\ve\in(0,1/2]$ and $\nu\in(0,\ve/4]$, 
for all $m\in\N$ and $T\in[0,t_\z)$, there is a constant $C=C(\a,\eta,\L,\ve,\nu,m,T)<\infty$ with the following property.
In the case $\z<1$, for all $c\le1/C$ and all $\s\ge c^{1/2-\ve}$, with probability exceeding $1-c^m$, for all $t\le T$, 
\begin{equation*}
|\cT_t-\t_t|
\le C\left(c^{1/2-\nu}+c^{1-4\nu}\left(\frac{e^\s}{e^\s-1}\right)^2\right)
\end{equation*}
and, for all $|z|=r\ge1+c^{1/2-\ve}$,
\begin{equation*}
|\hat\Phi_t(z)-z|
\le C\left(c^{1/2-\nu}+c^{1-4\nu}\left(\frac{e^\s}{e^\s-1}\right)^2\right).
\end{equation*}
Moreover, in the case $\z=1$ with $\ve\in(0,1/5]$,
for all $c\le1/C$ and all $\s\ge c^{1/5-\ve}$, with probability exceeding $1-c^m$, for all $t\le T$,
\begin{equation*}
|\cT_t-\t_t|
\le C\left(c^{1/2-\nu}+c^{1-4\nu}\left(\frac{e^\s}{e^\s-1}\right)^3\right)
\end{equation*}
and, for all $|z|=r\ge1+c^{1/5-\ve}$,
$$
|\hat\Phi_t(z)-z|
\le C\left(c^{1/2-\nu}\left(\frac r{r-1}\right)^{1/2}+c^{1-4\nu}\left(\frac{e^\s}{e^\s-1}\right)^3\right).
$$
\end{theorem}

We will show a similar result for the discrete-time process $(\Phi_n^\disc)_{n\ge0}$.
Set
$$
\cT^\disc_n=\log(\Phi^\disc_n)'(\infty),\q\hat\Phi^\disc_n(z)=\Phi^\disc_n(z)/e^{\cT^\disc_n}.
$$
Define
$$
n_\a=
\begin{cases}
\infty,&\text{if $\a\ge0$},\\
|\a|^{-1},&\text{if $\a<0$}
\end{cases}
$$
and for $n<n_\a/c$ set
\begin{equation}
\label{TDISC}
\t^\disc_n=
\begin{cases}
cn,&\text{if $\a=0$},\\
\a^{-1}\log(1+\a cn),&\text{otherwise}.
\end{cases}
\end{equation}
The following result is proved at the end of Section \ref{sec:highprob}.
The case $\a=0$ is Theorem 1.1 in \cite{NST}.

\begin{theorem}
\label{DISC}
For all $\a,\eta\in\R$ with $\z=\a+\eta\le1$, for all $\ve\in(0,1/2]$ and $\nu\in(0,\ve/4]$, 
for all $m\in\N$ and $N\in[0,n_\a)$, not necessarily an integer,
there is a constant $C=C(\a,\eta,\L,\ve,\nu,m,N)<\infty$ with the following property.
In the case $\z<1$, for all $c\le1/C$ and all $\s\ge c^{1/2-\ve}$, with probability exceeding $1-c^m$, for all $n\le N/c$, 
\begin{equation*}
|\cT^\disc_n-\t^\disc_n|
\le Cc^{1-4\nu}\left(\frac{e^\s}{e^\s-1}\right)^2
\end{equation*}
and, for all $|z|=r\ge1+c^{1/2-\ve}$,
\begin{equation*}
|\hat\Phi^\disc_n(z)-z|
\le C\left(c^{1/2-\nu}+c^{1-4\nu}\left(\frac{e^\s}{e^\s-1}\right)^2\right).
\end{equation*}
Moreover, in the case $\z=1$ with $\ve\in(0,1/5]$,
for all $c\le1/C$ and all $\s\ge c^{1/5-\ve}$, with probability exceeding $1-c^m$, for all $n\le N/c$,
\begin{equation*}
|\cT^\disc_n-\t^\disc_n|
\le Cc^{1-4\nu}\left(\frac{e^\s}{e^\s-1}\right)^3
\end{equation*}
and, for all $|z|=r\ge1+c^{1/5-\ve}$,
$$
|\hat\Phi^\disc_n(z)-z|
\le C\left(c^{1/2-\nu}\left(\frac r{r-1}\right)^{1/2}+c^{1-4\nu}\left(\frac{e^\s}{e^\s-1}\right)^3\right).
$$
\end{theorem}

We turn to our second main result, which describes the limiting fluctuations of ALE$(\a,\eta)$.
Denote by $\cH$ the set of all holomorphic functions on $D_0=\{|z|>1\}$ which are bounded at $\infty$.
We equip $\cH$ with the the topology of uniform convergence on $\{|z|\ge r\}$ for all $r>1$.
Define for $t<t_\z$
$$
\hat\Psi_t(z)=\hat\Phi_t(z)-z,\q\Psi^\cp_t=\cT_t-\t_t.
$$
Let $(B_t)_{t\ge0}$ be a (real) Brownian motion.
Let $(B_t(k))_{t\ge0}$ for $k\ge0$ be a sequence of independent complex Brownian motions, independent of $(B_t)_{t\ge0}$.
We can define continuous Gaussian processes $(\G_t(k))_{t<t_\z}$ and $(\G^\cp_t)_{t<t_\z}$ 
by the following Ornstein--Uhlenbeck-type stochastic differential equations
\begin{align*}
d\G_t(k)&=e^{-\a\t_t}\left(\sqrt2e^{-\eta\t_t/2}dB_t(k)-(1+(1-\z)k)\G_t(k)e^{-\eta\t_t}dt\right),\q\G_0(k)=0,\\
d\G^\cp_t&=e^{-\a\t_t}\left(e^{-\eta\t_t/2}dB_t-\z\G^\cp_te^{-\eta\t_t}dt\right),\q\G^\cp_0=0.
\end{align*}
We show in Section \ref{GLP} that the following series converges in $\cH$, uniformly on compacts in $[0,t_\z)$, almost surely
$$
\hat\G_t(z)=\sum_{k=0}^\infty\G_t(k)z^{-k}.
$$
In fact $(\hat\G_t)_{t<t_\z}$ satisfies the following stochastic differential equation in $\cH$
$$
d\hat\G_t=e^{-\a\t_t}\left(\sqrt 2e^{-\eta\t_t/2}d\hat B_t-(Q_0+1)\hat\G_te^{-\eta\t_t}dt\right),\q \hat\G_0=0
$$
where $Q_0f(z)=-(1-\z)zf'(z)$ and
$$
\hat B_t(z)=\sum_{k=0}^\infty B_t(k)z^{-k}.
$$
The following two results are proved in Section \ref{sec:fluctuations}.

\begin{theorem}
\label{FLUCT}
Assume that $\z=\a+\eta\in(-\infty,1]$.
Fix $T\in[0,t_\z)$ and $\ve>0$ and consider the limit $c\to0$ with $\s\to0$ subject to the constraint
$$
\s\ge
\begin{cases}
c^{1/4-\ve},&\text{if $\z<1$},\\
c^{1/6-\ve},&\text{if $\z=1$}.
\end{cases}
$$
Then
$$
c^{-1/2}(\hat\Psi_t,\Psi^\cp_t)_{t\le T}\to(\hat\G_t,\G^\cp_t)_{t\le T}
$$
weakly in the Skorokhod space $D([0,T],\cH\times\R)$.
\end{theorem}

As in the bulk scaling limit, we can deduce an analogous discrete-time fluctuation theorem.
The case $\a=0$ recovers Theorem 1.2 in \cite{NST}.
Define for $t\ge0$
$$
\hat\Psi^\disc_t(z)=\hat\Phi^\disc_{\lfloor t\rfloor}(z)-z.
$$
We have seen already in Theorem \ref{DISC}, for $N<n_\a$, 
that $(\cT^\disc_n-\t^\disc_n)_{n\le N/c}$ does not fluctuate at scale $\sqrt c$.
We can define a continuous Gaussian process $(\hat\G^\disc_t)_{t<n_\a}$ in $\cH$ by 
$$
d\hat\G^\disc_t=\frac{\sqrt 2d\hat B_t-(Q_0+1)\hat\G^\disc_tdt}{1+\a t},\q\hat\G^\disc_0=0.
$$

\begin{theorem}
\label{DISCFLUCT}
Assume that $\z=\a+\eta\in(-\infty,1]$.
Fix $N\in[0,n_\a)$, not necessarily an integer, and fix $\ve>0$.
In the limit $c\to0$ with $\s\to0$ considered in Theorem \ref{FLUCT}, we have
$$
c^{-1/2}(\hat\Psi^\disc_{t/c})_{t\le N}\to(\hat\G^\disc_t)_{t\le N}
$$
weakly in $D([0,N],\cH)$.
\end{theorem}

\subsection{Commentary and review of related work}
Hastings and Levitov \cite{HL} introduced the family of planar aggregation models HL($\a$), 
which are the cases $\eta=\s=0$ of our ALE$(\a,\eta)$ model.
They discovered by numerical experiments that, for small particles, the models underwent a transition at $\a=1$: 
for $\a\le1$ the cluster grows like a disk, while for $\a>1$ it exhibits fractal properties.
The HL$(0)$ model was subsequently investigated rigorously in a series of works 
\cite{RZ}, 
\cite{NT2} (bulk scaling limit),
\cite{JST12},
\cite{Sil} (fluctuation scaling limit).
The $\s$-regularized variant of HL$(\a)$ was proposed in \cite{JST15}, where it was shown for slit maps that,
if $\s\gg(\log(1/c))^{-1/2}$, there is disk-like behaviour for all $\a\ge0$: 
it appeared that the observed fractal properties of HL$(\a)$ for $\a>1$ were suppressed by strong regularization.
In contrast, for the weaker regularization used in the present paper, the phase transition at $\a=1$ (or $\z=1$) becomes visible.
The method of \cite{JST15} used a comparison with an HL$(0)$-type model which breaks down for smaller values of $\s$.
The regularized ALE$(\a,\eta)$ model appeared in \cite{STV}, where it was shown that, for slit maps and for $\eta>1$, a
$\s$-regularized ALE$(0,\eta)$ grows as a line for sufficiently small $\s\le c^\g$, for some $\g<\infty$ depending on $\eta$.

A new approach was begun in \cite{NST}, 
treating regularized ALE$(0,\eta)$ as a Markov chain in univalent functions by martingale arguments:
a bulk scaling limit and fluctuation scaling limit were shown, 
subject to the constraint $\eta\le1$ and to restrictions on $\s$ as a fractional power of $c$. 
These limits turn out not to depend on the details of individual particle shapes.
In this paper, we extend the analysis of \cite{NST} to ALE$(\a,\eta)$, subject now to the constraint $\z=\a+\eta\le1$.
Thus we now include regularized HL$(\a)$ for $\a\le1$.
Hastings and Levitov had argued that there should be a trade-off between $\a$ and $\eta$, 
with only $\z$ affecting the bulk scaling limit, and on this basis proposed HL$(1)$, that is ALE$(1,0)$,
as a continuum variant of the Eden model.
A more direct continuum analogue of the Eden model is ALE$(2,-1)$. 
Our results, in the regularized case, both justify the trade-off argument and show a disk scaling limit whenever $\z\le1$.
On the other hand, we show that ALE$(1,0)$ and ALE$(2,-1)$ have different fluctuation behaviour.
As in \cite{NST}, the behaviour of fluctuations as a function of $\z$ is consistent with the conjectured transition in 
behaviour at $\z=1$.

Hastings and Levitov \cite{HL} identify a Loewner--Kufarev-type equation, 
which they propose as governing the small-particle limit of HL$(\a)$,
citing a discussion of Shraiman and Bensimon \cite{MR763133} for the Hele--Shaw flow, where $\a$ is taken to be $2$.
This is the LK$(\a)$ equation, which is the subject of the next section.
As noted by Sola in a contribution to \cite{MR3817856}, there is a lack of mathematical theory for the LK$(\a)$ equation, 
except in the case $\a=2$ when some special techniques become available.
In this paper, since our focus is on clusters initiated as a disk, we are able to use an explicit solution of the equation, 
along with its linearization around that solution, so we do not rely on a general theory.
However, the particle interpretation established here offers some evidence that for $\a\le1$, 
the LK$(\a)$ equation may have a suitable existence, uniqueness and stability theory, 
and that it may be possible to derive the equation as a limit of particle models.
McEnteggart \cite{MCENT} has shown short-time existence and uniqueness for holomorphic initial data, 
by adapting a classical argument for the case $\a=2$.

Our results depend on constraints on the regularization parameter $\s$, 
though substantially weaker ones than those used in \cite{JST15}.
These constraints limit the interactions of individual particles and place us in the simplest case of Gaussian fluctuations.
At a technical level, for Theorem \ref{DISK}, 
these constraints come from the need to have $\bar\d(e^\s)\le c^\ve$ in Proposition \ref{II},
while for Theorem \ref{FLUCT} they are needed to show that the Poisson integral process $(\Pi_t)_{t\ge0}$ 
is a good approximation to the fluctuations in Proposition \ref{UDFA}.
In the case $\z=1$, the regularizing operator $Q$ obtained by linearization of the LK$(\z)$ equation collapses from
a fixed multiple of the Cauchy operator to $\s$ times the second derivative.
In general, for scaling regimes where $\s\to0$ faster than our fluctuation results allow, 
it remains possible that ALE$(\a,\eta)$ has different universal fluctuation behaviour, 
such as KPZ, as has been conjectured for the lattice Eden model.

\subsection{Structure of the paper}
In the next section, we discuss the Loewner--Kufarev equation for the limit dynamics. 
Then, in Section \ref{IFMC}, we derive an interpolation formula between ALE$(\a,\eta)$ and solutions of the limit equation.
The terms in this formula are estimated in Section \ref{EST}.
Equipped with these estimates, 
we show the bulk scaling limit in Section \ref{BSL} and the fluctuation scaling limit in Section \ref{sec:fluctuations}.
We collect in an Appendix \ref{APP} some further estimates needed in the course of the paper, 
including estimates on the conformal maps which encode single particles and particle families.

\section{Loewner--Kufarev equation}
\label{SBE}
Let $\cS$ denote the set of univalent holomorphic functions $\phi$ on $\{|z|>1\}$ with $\phi(\infty)=\infty$ and $\phi'(\infty)\in(0,\infty)$.
Then each $\phi\in\cS$ has the form
$$
\phi(z)=e^c\left(z+\sum_{k=0}^\infty a_kz^{-k}\right)
$$
for some $c\in\R$ and some sequence $(a_k:k\ge0)$ in $\C$.
Fix parameters $\z\in\R$ and $\s\ge0$.
Given $\phi_0\in\cS$, consider the following Cauchy problem for $(\phi_t)_{t\ge0}$ in $\cS$
\begin{equation}
\label{RSBE}
\dot\phi_t=a(\phi_t)
\end{equation}
where
$$
a(\phi)(z)=z\phi'(z)\fint_0^{2\pi}\frac{z+e^{i\th}}{z-e^{i\th}}\left|\phi'(e^{\s+i\th})\right|^{-\z}d\th.
$$
The case $\s=0$ of this equation is the equation proposed by Hastings and Levitov as scaling limit for HL$(\z)$, 
which we will call the LK$(\z)$ equation. 
When $\z=0$, the value of $\s$ is immaterial and there is a unique solution given by
$$
\phi_t(z)=\phi_0(e^tz).
$$
Where $\z=2$ and $\s=0$, \eqref{RSBE} is the Loewner--Kufarev equation associated to the Hele--Shaw flow.
For $\s>0$, we will refer to \eqref{RSBE} as the $\s$-regularized LK$(\z)$ equation.
We will be interested in the subcritical case $\z\in(-\infty,1]$.

The general form of the Loewner--Kufarev equation is given by
\begin{equation*}
\label{LKE}
\dot\phi_t(z)=z\phi'_t(z)\int_0^{2\pi}\frac{z+e^{i\th}}{z-e^{i\th}}\,\mu_t(d\th)
\end{equation*}
with $(\mu_t:t\ge0)$ a given family of measures on $[0,2\pi)$.
Thus the $\s$-regularized LK$(\z)$ equation is obtained by requiring that the driving measures are given by
$$
\mu_t(d\th)=\left|\phi'_t(e^{\s+i\th})\right|^{-\z}d\th/(2\pi).
$$
Note that the density of these driving measures is the product of the density of the local attachment rate and the local
particle capacity \eqref{MODEL} for ALE$(\a,\eta)$.
By the Loewner--Kufarev theory, for any solution $(\phi_t)_{t\ge0}$ of \eqref{RSBE}, the sets
$$
K_t=\C\sm\{\phi_t(z):|z|>1\}
$$
form an increasing family of simply-connected compacts, with capacities given by
$$
\t_t=\cp(K_t)=\log\phi_t'(\infty)=\log\phi_0'(\infty)+\int_0^t\mu_s([0,2\pi))ds.
$$

\vfe
\subsection{Linearization}
We compute the linearization of \eqref{RSBE} around a solution $(\phi_t)_{t\ge0}$.
For $\psi$ holomorphic in $\{|z|>1\}$, we have
$$
(\nabla a(\phi)\psi)(z)
=\left.\frac d{d\ve}\right|_{\ve=0}a(\phi+\ve\psi)(z)
=z\psi'(z)h(z)-\z z\phi'(z)g(z)
$$
where
$$
h(z)=\fint_0^{2\pi}\frac{z+e^{i\th}}{z-e^{i\th}}|\phi'(e^{\s+i\th})|^{-\z}d\th
$$
and, setting $\rho=\psi'/\phi'$,
\begin{equation}
\label{CTI}
g(z)=\fint_0^{2\pi}\frac{z+e^{i\th}}{z-e^{i\th}}|\phi'(e^{\s+i\th})|^{-\z}\re\rho(e^{\s+i\th})d\th.
\end{equation}
Note that first-order variations in $\cS$ have the form
$$
\psi(z)=\d z+\sum_{k=0}^\infty\psi_kz^{-k},\q\d\in\R,\q\psi_k\in\C.
$$
The process of first-order variations $(\psi_t)_{t\ge0}$ around a solution $(\phi_t)_{t\ge0}$ 
can be expected to satisfy the linearized equation
$$
\dot\psi_t=\nabla a(\phi_t)\psi_t.
$$

\vfe
\subsection{Linear stability of disk solutions in the subcritical case}
\label{DSV}
For all $\s\ge0$, a trial solution $\phi_t(z)=e^{\t_t}z$ for \eqref{RSBE} leads to the equation $\dot\t_t=e^{-\z\t_t}$.
Let us consider first the case where $\z\ge0$.
We can solve to obtain
$$
\t_t=
\begin{cases}
\t_0+t,&\text{if $\z=0$},\\
\z^{-1}\log(e^{\z\t_0}+\z t),&\text{otherwise}.
\end{cases}
$$
We thus find solutions $(\phi_t)_{t\ge0}$ to \eqref{RSBE} in which the sets $K_t$ form a growing family of disks.
For such disk solutions, we have $\phi_t'(z)=e^{\t_t}$ for all $z$, so we can evaluate the integral \eqref{CTI} to obtain
$$
(\nabla a(\phi_t)\psi)(z)
=-Q\psi(z)\dot\t_t
$$
where
\begin{equation}
\label{DEFQ}
Q\psi(z)=-z\psi'(z)+\z z\psi'(e^\s z)=-D\psi(z)+\z e^{-\s}D\psi(e^\s z).
\end{equation}
Here and below, we write $D\psi(z)$ for $z\psi'(z)$.
Note that $Q$ acts as a multiplier on the Laurent coefficients.
For $\psi(z)=\sum_{k=-1}^\infty\psi_kz^{-k}$, we have 
$$
Q\psi(z)=\sum_{k=-1}^\infty q(k)\psi_kz^{-k},\q
q(k)=k(1-\z e^{-\s(k+1)}).
$$
We split $Q$ as a sum of multiplier operators $Q_0+Q_1$ with multipliers given by
\begin{equation}
\label{NRQ}
q_0(k)=(1-\z)k,\q
q_1(k)=\z k(1-e^{-\s(k+1)}).
\end{equation}
Define for $\d\ge0$
\begin{equation}
\label{PDEF}
P(\d)=e^{-\d Q},\q
P_0(\d)=e^{-\d Q_0},\q
P_1(\d)=e^{-\d Q_1}.
\end{equation}
We alert the reader to the similarity of this notation with that used for the particle family $(P^{(c)}:c\in(0,\infty))$.
Then
$$
P(\d)\psi(z)
=P_0(\d)P_1(\d)\psi(z)
=P_1(\d)\psi(e^{(1-\z)\d}z)
$$
and, at least formally, the linearized equation $\dot\psi_t=\nabla a(\phi_t)\psi_t$ has solution given by
\begin{equation}
\label{PSI}
\psi_t(z)=P(\t_t-\t_0)\psi_0(z)=P_1(\t_t-\t_0)\psi_0(e^{(1-\z)(\t_t-\t_0)}z).
\end{equation}
In the case $\z>1$, at least when $\s=0$ so $Q=Q_0$, for example in the Hele--Shaw case,
we see that $\psi_t$ can be holomorphic in $\{|z|>1\}$ 
only if $\psi_0$ extends to a holomorphic function in the larger domain $\{|z|>e^{-(\z-1)(\t_t-\t_0)}\}$.
On the other hand, 
we will show, for all $\s\ge0$, that $P_1(\t_t-\t_0)$ preserves the set of holomorphic first-order variations,
so, when $\z\in[0,1]$, the variation $\psi_t$ as given by \eqref{PSI} remains holomorphic for all $t$.

We turn to the case where $\z<0$.
The differential equation $\dot\t_t=e^{-\z\t_t}$ has now only a local solution, given by
$$
\t_t=z^{-1}\log(e^{\z\t_0}+\z t),\q
t<t_\z=e^{\z\t_0}/|\z|
$$
with $\t_t\to\infty$ as $t\to t_\z$.
It is convenient in this case to split $Q$ differently, setting
\begin{equation}
\label{PSIN}
\tilde q_0(k)=k,\q
\tilde q_1(k)=|\z|ke^{-\s(k+1)}.
\end{equation}
Then, making similar definitions in all other respects, we have 
$$
\psi_t(z)=P(\t_t-\t_0)\psi_0(z)=\tilde P_1(\t_t-\t_0)\psi_0(e^{\t_t-\t_0}z)
$$
and $\psi_t$ remains holomorphic for all $t<t_\z$, as in the case $\z\in[0,1]$.

Define for $r>1$
$$
\|\psi\|_{p,r}=\left(\fint_0^{2\pi}|\psi(re^{i\th})|^pd\th\right)^{1/p}.
$$
The following inequality will be used in Section \ref{EST}.

\begin{lemma}
\label{PBDD}
For all $p\in(1,\infty)$, there is a constant $C(p)<\infty$ such that, 
for all $\z\in\R$ and all $\s\ge0$, for all holomorphic functions $\psi$ on $\{|z|>1\}$ bounded at $\infty$, 
all $t\ge0$ and all $r>1$, we have
\begin{align*}
\|P_1(t)\psi\|_{p,r}\le C(p)\|\psi\|_{p,r},&\q\text{if $\z\ge0$},\\
\|\tilde P_1(t)\psi\|_{p,r}\le C(p)\|\psi\|_{p,r},&\q\text{if $\z<0$}.
\end{align*}
\end{lemma}
\begin{proof}
Consider first the case $\z\ge0$.
The operator $P_1(t)$ acts as multiplication by $p_1(k,t)=e^{-tq_1(k)}$ on the $k$th Laurent coefficient.
We have $0\le q_1(k)\le q_1(k+1)$ so $0\le p_1(k+1,t)\le p_1(k,t)\le1$ for all $k$.
Hence the conditions of the Marcinkiewicz multiplier theorem (as recalled in Section \ref{sec:operators}) 
hold for $P_1(t)$ with $A=1$.
The desired estimate follows.

In the case $\z<0$, we modified the split so that $\tilde q_1(k)\ge0$, so $0\le\tilde p_1(k,t)\le1$ for all $k$.
Now $\tilde q_1(k)$ is no longer increasing but is unimodal in $k$, so $\tilde p_1(k,t)$ is also unimodal in $k$, and so
$$
\sum_{k=0}^\infty|\tilde p_1(k+1,t)-\tilde p_1(k,t)|\le2.
$$
Hence the Marcinkiewicz theorem applies with $A=2$ and we can conclude as before.
\end{proof}

\def\j{
We restrict our main account to the case $\z\ge0$ from now on, removing the need to refer to the time $t_\z$.
The case $\z<0$ is essentially similar, 
except that we obtain estimates uniform on compact subsets of $[0,t_\z)$ rather than compact subsets of $[0,\infty)$.
The details are left to the reader, except for occasional comments in the footnotes.
}

\vfe
\subsection{Transformation to (Schlicht function, capacity) coordinates}
\label{SFC}
Write $\cS_1$ for the set of `Schlicht functions at $\infty$' on $\{|z|>1\}$, given by
$$
\cS_1=\{\phi\in\cS:\phi'(\infty)=1\}.
$$
It will be convenient to use coordinates $(\hat\phi,\t)$ on $\cS$, given by
$$
\hat\phi(z)=e^{-\t}\phi(z),\q\t=\log\phi'(\infty).
$$
Then $\hat\phi\in\cS_1$ and $\t\in\R$.
It is straightforward to show that, for a solution $(\phi_t)_{t\ge0}$ to \eqref{RSBE}, 
the transformed variables $(\hat\phi_t,\t_t)_{t\ge0}$ satisfy
\begin{equation}
\label{SBES}
(\dot{\hat\phi}_t,\dot\t_t)
=b(\hat\phi_t,\t_t)=(\hat b,b^\cp)(\hat\phi_t,\t_t)
\end{equation}
where
\begin{align*}
\hat b(\hat\phi,\t)(z)
&=e^{-\z\t}z\hat\phi'(z)\fint_0^{2\pi}\frac{z+e^{i\th}}{z-e^{i\th}}|\hat\phi'(e^{\s+i\th})|^{-\z}d\th
-e^{-\z\t}\hat\phi(z)\fint_0^{2\pi}|\hat\phi'(e^{\s+i\th})|^{-\z}d\th,\\
b^\cp(\hat\phi,\t)
&=e^{-\z\t}\fint_0^{2\pi}|\hat\phi'(e^{\s+i\th})|^{-\z}d\th.
\end{align*}
On linearizing \eqref{SBES} around a solution $(\hat\phi_t,\t_t)_{t\ge0}$, 
we obtain equations for first-order variations $(\hat\psi_t,\psi^\cp_t)_{t\ge0}$ in the new coordinates,
where now $\hat\psi_t$ is bounded at $\infty$ for all $t$, reflecting the normalization of $\hat\phi_t$.
These are then related to the first-order variations $(\psi_t)_{t\ge0}$ in the old coordinates by
$$
\psi_t(z)=e^{\t_t}(\hat\psi_t(z)+\psi^\cp_t\hat\phi_t(z)).
$$

For a disk solution $(\phi_t)_{t<t_\z}$, we have $\hat\phi_t(z)=z$ and $b(\hat\phi_t,\t)=(0,e^{-\z\t})$.
The equations for first-order variations are then given by
$$
\dot{\hat\psi}_t(z)=-(Q+1)\hat\psi_t(z)\dot\t_t,\q
\dot\psi^\cp_t=-\z\psi^\cp_t\dot\t_t
$$
with solutions
$$
\hat\psi_t(z)
=e^{-(\t_t-\t_0)}P(\t_t-\t_0)\hat\psi_0(z),\q
\psi^\cp_t=e^{-\z(\t_t-\t_0)}\psi^\cp_0.
$$

\vfe
\section{Interpolation formula for Markov chain fluid limits}
\label{IFMC}
We use an interpolation formula between continuous-time Markov chains and differential equations,
which we first review briefly in a general setting.
This formula is then applied to an ALE($\a,\eta$) aggregation process $(\Phi_t)_{t\ge0}$ with capacity parameter $c$, 
regularization parameter $\s$ and particle family $(P^{(c)}:c\in(0,\infty))$,
taking as limit equation the $\s$-regularized LK$(\z)$ equation with $\z=\a+\eta$.
We use (Schlicht function, capacity) coordinates for both the process and the limit equation.

\vfe
\subsection{General form of the interpolation formula}
Let $(X_t)_{t\ge0}$ be a continuous-time Markov chain with state-space $E$ and transition rate kernel $q$, starting from $x_0$ say.
Suppose for this general discussion that $E=\R^d$.
Let $b$ be a vector field on $E$ with continuous bounded derivative $\nabla b$.
Write $(\xi_t(x):t\ge0,x\in E)$ for the flow of $b$.
The compensated jump measure of $(X_t)_{t\ge0}$ is the signed measure $\tilde\mu^X$ on $E\times(0,\infty)$ given by
$$
\tilde\mu^X(dy,dt)=\mu^X(dy,dt)-q(X_{t-},dy)dt,\q\mu^X=\sum_{t:X_t\not=X_{t-}}\d_{(X_t,t)}.
$$
Set $x_t=\xi_t(x_0)$ and define, for $s\in[0,t]$,
$$
Z_s=x_t+\nabla\xi_{t-s}(x_s)(X_s-x_s).
$$
Then $Z_0=x_t$ and $Z_t=X_t$ and, on computing the martingale decomposition of $(Z_s)_{s\le t}$, 
we obtain the interpolation formula
\begin{equation}
\label{HLISR}
X_t-x_t=M_t+A_t
\end{equation}
where 
$$
M_t=\int_{E\times(0,t]}\nabla\xi_{t-s}(x_s)(y-X_{s-})\tilde\mu^X(dy,ds)
$$
and
$$
A_t=\int_0^t\nabla\xi_{t-s}(x_s)(\b(X_s)-b(x_s)-\nabla b(x_s)(X_s-x_s))ds
$$
where $\b$ is the drift of $(X_t)_{t\ge0}$, given by
$$
\b(x)=\int_E(y-x)q(x,dy).
$$
We will use this formula in a case where the state-space $E$ is infinite-dimensional.
Rather than justify its validity generally in such a context, in the next section, 
we will prove directly the special case of the formula which we require.
Note that the integrands in $M_t$ and $A_t$ depend on $t$.
Nevertheless, we will call $M_t$ the martingale term and $A_t$ the drift term.

\vfe
\subsection{Proof of the formula for ALE($\a,\eta$)}
\label{PFH}
Let $(\Phi_t)_{t\ge0}$ be an ALE($\a,\eta$) aggregation process with capacity parameter $c$, 
regularization parameter $\s$ and particle family $(P^{(c)}:c\in(0,\infty))$.
See Section \ref{DESM} and \eqref{MODEL} for the specification of this process.
We use (Schlicht function, capacity) coordinates, as in Section \ref{SFC}, 
to obtain a continuous-time Markov chain $(X_t)_{t\ge0}=(\hat\Phi_t,\cT_t)_{t\ge0}$ in $\cS_1\times[0,\infty)$.
When in state $x=(\hat\phi,\t)$, for all $\th\in[0,2\pi)$,
this process makes a jump of size $(\Delta(\th,z,c(\th),\hat\phi),c(\th))$ at rate $\l(\th)d\th/(2\pi)$, where
\begin{equation*}
\label{HEQQ}
\Delta(\th,z,c,\hat\phi)=e^{-c}\hat\phi(F_c(\th,z))-\hat\phi(z)
\end{equation*}
and
$$
c(\th)=c(\th,\hat\phi,\t)=ce^{-\a\t}|\hat\phi'(e^{\s+i\th})|^{-\a},\q
\l(\th)=\l(\th,\hat\phi,\t)=c^{-1}e^{-\eta\t}|\hat\phi'(e^{\s+i\th})|^{-\eta}.
$$
We can and do assume that the process is constructed from a Poisson random measure 
$\mu$ on $[0,2\pi)\times[0,\infty)\times(0,\infty)$ of intensity $(2\pi)^{-1}d\th dvdt$ 
by the following stochastic differential equation:
$$
\hat\Phi_t(z)=\int_{E(t)}H_s(\th,z)1_{\{v\le\L_s(\th)\}}\mu(d\th,dv,ds),\q
\cT_t=\int_{E(t)}C_s(\th)1_{\{v\le\L_s(\th)\}}\mu(d\th,dv,ds)
$$
where
$$
E(t)=[0,2\pi)\times[0,\infty)\times(0,t]
$$
and
$$
H_s(\th,z)=\Delta(\th,z,C_s(\th),\hat\Phi_{s-}),\q
C_s(\th)=c(\th,\hat\Phi_{s-},\cT_{s-}),\q
\L_s(\th)=\l(\th,\hat\Phi_{s-},\cT_{s-}).
$$
We use the vector field $b=(\hat b,b^\cp)$ of the $\s$-regularized LK$(\z)$ equation \eqref{SBES}, 
written in (Schlicht function, capacity) coordinates.
We consider the disk solution $(x_t)_{t\ge0}=(\hat\phi_t,\t_t)_{t<t_\z}$ with initial capacity $\t_0=0$,
which is given by 
\begin{equation}
\label{DISCSOL}
\hat\phi_t(z)=z,\q 
\t_t
=
\begin{cases}
t,&\text{if $\z=0$},\\
\z^{-1}\log(1+\z t),&\text{if $\z\not=0$},
\end{cases}
\q
t_\z=
\begin{cases}
\infty,&\text{if $\z\ge0$},\\
|\z|^{-1},&\text{if $\z<0$}.
\end{cases}
\end{equation}
We will compute the form of the interpolation formula in this case and then prove directly that it holds.
Note that
$$
b(x_t)=(\hat b,b^\cp)(\hat\phi_t,\t_t)=(0,e^{-\z\t_t})
$$
and, for $y=(\hat\psi,\psi^\cp)$,
$$
\nabla b(x_t)y
=-e^{-\z\t_t}((Q+1)\hat\psi,\z\psi^\cp)
$$
and the first-order variation at time $t$ due to a variation $y$ at time $s\le t$ is given by
$$
\nabla\xi_{t-s}(x_s)y
=(e^{-(\t_t-\t_s)}P(\t_t-\t_s)\hat\psi,e^{-\z(\t_t-\t_s)}\psi^\cp).
$$
Write $\tilde\mu$ for the compensated Poisson random measure
$$
\tilde\mu(d\th,dv,ds)=\mu(d\th,dv,ds)-(d\th/2\pi)dvds.
$$
Fix $t\ge0$ and set $\d_s=\t_t-\t_s$.
We alert the reader to the concealed dependence of $\d_s$ on $t$.
The martingale term $M_t=(\hat M_t,M^\cp_t)$ in the interpolation formula may then be written 
\begin{align*}
\hat M_t(z)
&=\int_{E(t)}e^{-\d_s}P(\d_s)H_s(\th,z)1_{\{v\le\L_s(\th)\}}\tilde\mu(d\th,dv,ds),\\
M^\cp_t
&=\int_{E(t)}e^{-\z\d_s}C_s(\th)1_{\{v\le\L_s(\th)\}}\tilde\mu(d\th,dv,ds).
\end{align*}
The drift $\b=(\hat\b,\b^\cp)$ for $(\hat\Phi,\cT)$ is given by
\begin{align*}
\hat\b(\hat\phi,\t)(z)
&=\fint_0^{2\pi}\Delta(\th,z,c(\th,\hat\phi,\t),\hat\phi)\l(\th,\hat\phi,\t)d\th,\\
\b^\cp(\hat\phi,\t)
&=\fint_0^{2\pi}c(\th,\hat\phi,\t)\l(\th,\hat\phi,\t)d\th.
\end{align*}
Write $\hat\Psi_s(z)=\hat\Phi_s(z)-\hat\phi_s(z)=\hat\Phi_s(z)-z$ and $\Psi^\cp_s=\cT_s-\t_s$.
Then we have formally
$$
\nabla b(x_s)(X_s-x_s)
=-e^{-\z\t_s}((Q+1)\hat\Psi_s,\z\Psi^\cp_s)
$$
and so
$$
\nabla\xi_{t-s}(x_s)\nabla b(x_s)(X_s-x_s)=-e^{-\z\t_s}(e^{-\d_s}P(\d_s)(Q+1)\hat\Psi_s,e^{-\z\d_s}\z\Psi^\cp_s).
$$
The following interpolation identities may then be obtained formally by
splitting equation \eqref{HLISR} into its Schlicht function and capacity components.

\begin{proposition}
For all $t<t_\z$ and all $|z|>1$, we have
\begin{equation}
\label{KID}
\hat\Psi_t(z)=\hat M_t(z)+\hat A_t(z),\q
\Psi^\cp_t=M^\cp_t+A^\cp_t
\end{equation}
where
\begin{align*}
\hat A_t(z)
&=\int_0^te^{-\d_s}P(\d_s)\left(\hat\b(\hat\Phi_s,\cT_s)+e^{-\z\t_s}(Q+1)\hat\Psi_s\right)(z)ds,\\
A^\cp_t
&=\int_0^te^{-\z\d_s}\left(\b^\cp(\hat\Phi_s,\cT_s)-e^{-\z\t_s}+\z e^{-\z\t_s}\Psi^\cp_s\right)ds.
\end{align*}
\end{proposition}
\begin{proof}
Fix $t<t_\z$. 
For $x\in[0,t]$, recall that $\d_x=\t_t-\t_x$ and define for $|z|>1$
$$
\hat\Psi_{x,t}(z)=e^{-\d_x}P(\d_x)(\hat\Phi_x-\hat\phi_x)(z),\q
\Psi^\cp_{x,t}=e^{-\z\d_x}(\cT_x-\t_x).
$$
Set
\begin{align*}
\hat M_{x,t}(z)
&=\int_{E(x)}e^{-\d_s}P(\d_s)H_s(\th,z)1_{\{v\le\L_s(\th)\}}\tilde\mu(d\th,dv,ds),\\
M^\cp_{x,t}
&=\int_{E(x)}e^{-\z\d_s}C_s(\th)1_{\{v\le\L_s(\th)\}}\tilde\mu(d\th,dv,ds)
\end{align*}
and
\begin{align*}
\hat A_{x,t}(z)
&=\int_0^xe^{-\d_s}P(\d_s)\left(\hat\b(\hat\Phi_s,\cT_s)+e^{-\z\t_s}(Q+1)\hat\Psi_s\right)(z)ds,\\
A^\cp_{x,t}
&=\int_0^xe^{-\z\d_s}\left(\b^\cp(\hat\Phi_s,\cT_s)-e^{-\z\t_s}+\z e^{-\z\t_s}\Psi^\cp_s\right)ds.
\end{align*}
We will show that, for all $x\in[0,t]$ and all $|z|>1$,
$$
\hat\Psi_{x,t}(z)=\hat M_{x,t}(z)+\hat A_{x,t}(z),\q
\Psi^\cp_{x,t}=M^\cp_{x,t}+A^\cp_{x,t}.
$$
The case $x=t$ gives the claimed identities.
In the case $x=0$, all terms are $0$.
The left-hand and right-hand sides are piecewise continuously differentiable in $x$,
except for finitely many jumps, at the jump times of $(\Phi_x)_{0\le x\le t}$,
which occur when $\mu$ has an atom at $(\th,v,x)$ with $v\le\L_x(\th)$.
It will suffice to check that the jumps and derivatives agree.
Now $\hat A_{x,t}(z)$ and $A^\cp_{x,t}$ are continuous in $x$ and, at the jump times of $\Phi_x$, 
the jumps in $\hat\Psi_{x,t}(z)$ and $\Psi^\cp_{x,t}$ are given by
\begin{align*}
\Delta\hat\Psi_{x,t}(z)
&=e^{-\d_x}P(\d_x)\Delta\hat\Phi_x(z)\\
&=e^{-\d_x}P(\d_x)(e^{-C_x(\th)}\hat\Phi_{x-}\circ F_{C_x(\th)}(\th,.)-\hat\Phi_{x-})(z)\\
&=e^{-\d_x}P(\d_x)H_x(\th,z)
=\Delta\hat M_{x,t}(z)
\end{align*}
and
$$
\Delta\Psi^\cp_{x,t}
=e^{-\z\d_x}\Delta\cT_x
=e^{-\z\d_x}C_x(\th)
=\Delta M^\cp_{x,t}.
$$
So it remains to check the derivatives.
We will use a spectral calculation for the semigroup of multiplier operators $P(\t)=e^{-\t Q}$,
whose justification is straightforward.
Recall that $\dot\t_t=e^{-\z\t_t}$.
We have
$$
\frac d{dx}\d_x=-e^{-\z\t_x},\q
\frac d{dx}e^{-\d_x}=e^{-\d_x}e^{-\z\t_x},\q
\frac d{dx}e^{-\z\d_x}=\z 
e^{-\z\d_x}e^{-\z\t_x}
$$
and
$$
\frac d{dx}P(\d_x)=e^{-\z\t_x}QP(\d_x).
$$
So, between the jump times, we have
$$
\frac d{dx}\hat\Psi_{x,t}(z)
=e^{-\z\t_x}e^{-\d_x}P(\d_x)(Q+1)\hat\Psi_x(z),\q
\frac d{dx}\Psi^\cp_{x,t}(z)
=-e^{-\z\d_x}e^{-\z\t_x}(1-\z\Psi^\cp_x)
$$
and
\begin{align*}
\frac d{dx}\hat M_{x,t}(z)
&=-\fint_0^{2\pi}e^{-\d_x}P(\d_x)H_x(\th,z)\L_x(\th)d\th
=-e^{-\d_x}P(\d_x)\hat\b(\hat\Phi_x,\cT_x)(z),\\
\frac d{dx}M^\cp_{x,t}
&=-\fint_0^{2\pi}e^{-\z\d_x}C_x(\th)\L_x(\th)d\th
=-e^{-\z\d_x}\b^\cp(\hat\Phi_x,\cT_x)
\end{align*}
and
\begin{align*}
\frac d{dx}\hat A_{x,t}(z)
&=e^{-\d_x}P(\d_x)\left(\hat\b(\hat\Phi_x,\cT_x)+e^{-\z\t_x}(Q+1)\hat\Psi_x\right)(z),\\
\frac d{dx}A^\cp_{x,t}
&=e^{-\z\d_x}\left(\b^\cp(\hat\Phi_x,\cT_x)-e^{-\z\t_x}+\z e^{-\z\t_x}\Psi^\cp_x\right).
\end{align*}
Hence, between the jump times, 
$$
\frac d{dx}\hat\Psi_{x,t}(z)
=\frac d{dx}(\hat M_{x,t}(z)+\hat A_{x,t}(z)),\q
\frac d{dx}\Psi^\cp_{x,t}
=\frac d{dx}(M^\cp_{x,t}+A^\cp_{x,t})
$$
as required.
\end{proof}

\vfe
\section{Estimation of terms in the interpolation formula}
\label{EST}
We obtain some estimates on the terms in the interpolation formula \eqref{KID} for ALE($\a,\eta$)
when it is close to the disk solution \eqref{DISCSOL} of the LK$(\z)$ equation, with $\z=\a+\eta$.
For $\d_0\in(0,1/2]$, define
$$
T_0=T_0(\d_0)=\inf\big\{t\in[0,t_\z):\sup_{\th\in[0,2\pi)}|\hat\Psi_t'(e^{\s+i\th})|>\d_0\text{ or }|\Psi^\cp_t|>\d_0\big\}.
$$
We estimate first the martingale term and then the drift term.

\vfe
\subsection{Estimates for the martingale terms}
Recall that the martingale term $(\hat M_t,M^\cp_t)$ in the interpolation formula is given by
\begin{align*}
\hat M_t(z)
&=\int_{E(t)}e^{-(\t_t-\t_s)}P(\t_t-\t_s)H_s(\th,z)1_{\{v\le\L_s(\th)\}}\tilde\mu(d\th,dv,ds),\\
M^\cp_t
&=\int_{E(t)}e^{-\z(\t_t-\t_s)}C_s(\th)1_{\{v\le\L_s(\th)\}}\tilde\mu(d\th,dv,ds)
\end{align*}
where
\begin{equation}
\label{CDEF}
C_s(\th)=c(\th,\hat\Phi_{s-},\cT_{s-}),\q
\L_s(\th)=\l(\th,\hat\Phi_{s-},\cT_{s-})
\end{equation}
with
$$
c(\th,\hat\phi,\t)=ce^{-\a\t}|\hat\phi'(e^{\s+i\th})|^{-\a},\q
\l(\th,\hat\phi,\t)=c^{-1}e^{-\eta\t}|\hat\phi'(e^{\s+i\th})|^{-\eta}
$$
and
\begin{equation*}
\label{DEL}
H_s(\th,z)=\Delta(\th,z,C_s(\th),\hat\Phi_{s-}),\q
\Delta(\th,z,c,\hat\phi)=e^{-c}\hat\phi(F_c(\th,z))-\hat\phi(z).
\end{equation*}
Consider the following approximations to $\hat M_t(z)$ and $M_t^\cp$, which are obtained by replacing 
$\hat\Phi_{s-}$ by $\hat\phi_s$, $\cT_{s-}$ by $\t_s$ and $e^{-c}F_c(\th,z)-z$ by $2cz/(e^{-i\th}z-1)$.
(Under our assumptions on the particle family, the last approximation becomes good in the limit $c\to0$.
See Section \ref{sec:particle} and in particular equation \eqref{FEST7}.)
Define
\begin{align}
\label{MPS}
\hat\Pi_t(z)
&=\int_{E(t)}e^{-(\t_t-\t_s)}P(\t_t-\t_s)H(\th,z)2c_s1_{\{v\le\l_s\}}\tilde\mu(d\th,dv,ds),\\
\label{MPT}
\Pi^\cp_t
&=\int_{E(t)}e^{-\z(\t_t-\t_s)}c_s1_{\{v\le\l_s\}}\tilde\mu(d\th,dv,ds)
\end{align}
where 
$$
c_s=ce^{-\a\t_s},\q\l_s=c^{-1}e^{-\eta\t_s}
$$
and
\begin{equation}
\label{DEFH}
H(\th,z)=\frac z{e^{-i\th}z-1}=\sum_{k=0}^\infty e^{i(k+1)\th}z^{-k}.
\end{equation}

\begin{lemma}
\label{NEST}
For all $\a,\eta\in\R$, all $p\ge2$ and all $T<t_\z$, 
there is a constant $C(\a,\eta,p,T)<\infty$, such that, for all $c\in(0,1]$, all $\s\ge0$ and all $\d_0\in(0,1/2]$,
$$
\big\|\sup_{t\le T_0(\d_0)\wedge T}|M^\cp_t|\big\|_p\le C\sqrt c
$$
and
$$
\big\|\sup_{t\le T_0(\d_0)\wedge T}|M^\cp_t-\Pi^\cp_t|\big\|_p\le C(c+\sqrt{c\d_0}).
$$
\end{lemma}
\begin{proof}
Write $T_0$ for $T_0(\d_0)$ throughout the proofs.
Consider the martingale $(M_t)_{t<t_\z}$ given by
$$
M_t=\int_{E(t)}e^{\z\t_s}C_s(\th)1_{\{v\le\L_s(\th),\,s\le T_0\}}\tilde\mu(d\th,dv,ds).
$$
By an inequality of Burkholder,
for all $p\ge2$, there is a constant $C(p)<\infty$ such that, for all $t\ge0$,
\begin{equation}
\label{BUR}
\|M_t^*\|_p\le C(p)\left(\|\<M\>_t\|_{p/2}^{1/2}+\|(\Delta M)^*_t\|_p\right).
\end{equation}
We write here $M^*_t$ for $\sup_{s\le t}|M_s|$ and similarly for other processes.
See \cite[Theorem 21.1]{MR365692} for the discrete-time case. 
The continuous-time case follows by a standard limit argument.
Now
$$
\<M\>_t=\int_0^{T_0\wedge t}\fint_0^{2\pi}e^{2\z\t_s}C_s(\th)^2\L_s(\th)d\th ds
$$
and
$$
\Delta M_t
=|M_t-M_{t-}|\le e^{\z\t_t}\sup_{\th\in[0,2\pi)}C_t(\th).
$$
For all $t\le T_0\wedge T$ and all $\th\in[0,2\pi)$, we have
\begin{equation}
\label{JKJ}
e^{\z\t_t}\le C,\q
C_t(\th)\le Cc,\q
\L_t(\th)\le C/c
\end{equation}
so $\<M\>_t\le Cc$ and $(\Delta M)^*_t\le Cc$.
Here and below, we write $C$ for a finite constant of the dependence allowed in the statement.
The value of $C$ may vary from one instance to the next.
We remind the reader that $C_t(\th)$ and $\L_t(\th)$ are defined at \eqref{CDEF}.
Hence
$$
\|M^*_t\|_p\le C\sqrt c.
$$
Since $M^\cp_t=e^{-\z\t_t}M_t$ for all $t\le T_0$, the first claimed estimate follows.

For the second estimate, we use instead the martingale $(M_t)_{t\ge0}$ given by
$$
M_t=\int_{E(t)}e^{\z\t_s}\left(C_s(\th)1_{\{v\le\L_s(\th)\}}-c_s1_{\{v\le\l_s\}}\right)1_{\{s\le T_0\}}\tilde\mu(d\th,dv,ds).
$$
Then
$$
\<M\>_t=\int_0^{T_0\wedge t}\int_0^\infty\fint_0^{2\pi}
e^{2\z\t_s}\left(C_s(\th)1_{\{v\le\L_s(\th)\}}-c_s1_{\{v\le\l_s\}}\right)^2d\th dvds.
$$
For $t\le T_0\wedge T$ and $\th\in[0,2\pi)$, we have
\begin{equation}
\label{JKK}
|C_t(\th)-c_t|\le Cc\d_0,\q
|\L_t(\th)-\l_t|\le C\d_0/c
\end{equation}
so
\begin{equation}
\label{JKL}
\int_0^\infty
\left(C_t(\th)1_{\{v\le\L_t(\th)\}}-c_t1_{\{v\le\l_t\}}\right)^2dv\le Cc\d_0.
\end{equation}
Then $\<M\>_t\le Cc\d_0$ and $(\Delta M)_t\le Cc$.
Hence, by Burkholder's inequality, 
$$
\|M_t^*\|_p\le C(c+\sqrt{c\d_0}).
$$
Since $M^\cp_t-\Pi^\cp_t=e^{-\z\t_t}M_t$ for all $t\le T_0$, the second claimed estimate follows.
\end{proof}

Note that, since $\hat\Phi_t$ takes values in $\cS_1$, 
$\hat\Psi_t(z)=\hat\Phi_t(z)-z$ is bounded at $\infty$ and hence has a limiting value $\hat\Psi_t(\infty)$.
The same is true for the terms $\hat M_t$ and $\hat A_t$ in the interpolation formula.
Instead of estimating these terms directly, 
we estimate first their values at $\infty$ and then their derivatives $D\hat M_t$ and $D\hat A_t$,
since this gives the best control of the derivative of $\hat\Phi_t$ near the unit circle, which drives the dynamics of the process.

\begin{lemma}
\label{MESTI}
For all $\a,\eta\in\R$ with $\z=\a+\eta\le1$, all $p\ge2$ and all $T<t_\z$, 
there is a constant $C(\a,\eta,\L,p,T)<\infty$, 
such that, for all $c\in(0,1]$, all $\s\ge0$, all $\d_0\in(0,1/2]$ and all $t\le T$,
$$
\big\|\sup_{s\le T_0(\d_0)\wedge t}|\hat M_s(\infty)|\big\|_p^p
\le Cc^{p/2}\bigg(1+\int_0^t\|\hat\Psi_{s-}(\infty)1_{\{s\le T_0(\d_0)\}}\|_p^pds\bigg)
$$
and
$$
\big\|\sup_{s\le T_0(\d_0)\wedge t}|\hat M_s(\infty)-\hat\Pi_s(\infty)|\big\|_p^p
\le C\bigg(\left(c+\sqrt{c\d_0}\right)^p+c^{p/2}\int_0^t\|\hat\Psi_{s-}(\infty)1_{\{s\le T_0(\d_0)\}}\|_p^pds\bigg).
$$
\end{lemma}
\begin{proof}
By considering the Laurent expansions of $F_c$ and $\hat\phi$, we have
\begin{equation}
\label{DINF}
\Delta(\th,\infty,c,\hat\phi)=a_0(c)e^{i\th}+(e^{-c}-1)\hat\psi(\infty),\q\hat\psi(z)=\hat\phi(z)-z.
\end{equation}
Consider the martingale $(M_t)_{t<t_\z}$ given by
$$
M_t=\int_{E(t)}
e^{\t_s}\left(a_0(C_s(\th))e^{i\th}+(e^{-C_s(\th)}-1)\hat\Psi_{s-}(\infty)\right)
1_{\{v\le\L_s(\th),\,s\le T_0\}}\tilde\mu(d\th,dv,ds).
$$
Then $\hat M_t(\infty)=e^{-\t_t}M_t$ for all $t\le T_0$.
By Proposition \ref{CNC}, $|a_0(c)|\le Cc$ for all $c$.
Hence 
\begin{align*}
\<M\>_t
&=\int_0^{T_0\wedge t}\fint_0^{2\pi}e^{2\t_s}\left|a_0(C_s(\th))e^{i\th}+(e^{-C_s(\th)}-1)\hat\Psi_{s-}(\infty)\right|^2\L_s(\th)d\th ds\\
&\le Cce^{2\t_t}\int_0^{T_0\wedge t}(1+|\hat\Psi_{s-}(\infty)|^2)ds
\end{align*}
and, for $p\ge2$, since
$$
|(\Delta M)^*_t|^p
\le\int_{E(t)}
e^{p\t_s}\left|a_0(C_s(\th))e^{i\th}+(e^{-C_s(\th)}-1)\hat\Psi_{s-}(\infty)\right|^p
1_{\{v\le\L_s(\th),s\le T_0\}}\mu(d\th,dv,ds)
$$
we have
\begin{align*}
\|(\Delta M)^*_t\|^p_p
&\le\E\int_0^{T_0\wedge t}\fint_0^{2\pi}e^{p\t_s}\left|a_0(C_s(\th))e^{i\th}+(e^{-C_s(\th)}-1)\hat\Psi_{s-}(\infty)\right|^p\L_s(\th)d\th ds\\
&\le Cc^{p-1}e^{p\t_t}\E\int_0^{T_0\wedge t}(1+|\hat\Psi_{s-}(\infty)|^p)ds.
\end{align*}
The first claimed estimate then follows from Burkholder's inequality \eqref{BUR}.

For the second estimate, we consider instead the martingale $(M_t)_{t<t_\z}$ given by
\begin{align*}
M_t&=\int_{E(t)}e^{\t_s}\Big(\left(a_0(C_s(\th))1_{\{v\le\L_s(\th)\}}-2c_s1_{\{v\le\l_s\}}\right)e^{i\th}\\
&\q\q\q\q+(e^{-C_s(\th)}-1)\hat\Psi_{s-}(\infty)1_{\{v\le\L_s(\th)\}}\Big)1_{\{s\le T_0\}}\tilde\mu(d\th,dv,ds).
\end{align*}
Then $\hat M_t(\infty)-\hat\Pi_t(\infty)=e^{-\t_t}M_t$ for all $t\le T_0$.
By Proposition \ref{CNC}, we have $|a_0(c)-2c|\le Cc^{3/2}$. 
We combine this with \eqref{JKJ} and \eqref{JKK} to see that
$$
\int_0^\infty\left|a_0(C_t(\th))1_{\{v\le\L_t(\th)\}}-2c_t1_{\{v\le\l_t\}}\right|^pdv
\le C(c^{3p/2-1}+c^{p-1}\d_0)
\le C(c^p+c^{p-1}\d_0).
$$
The second estimate then follows by Burkholder's inequality as above.
\end{proof}

Recall that, for $p\in[1,\infty)$ and $r>1$, we set 
$$
\|\psi\|_{p,r}=\left(\fint_0^{2\pi}|\psi(re^{i\th})|^pd\th\right)^{1/p}.
$$
For a measurable function $\Psi$ on $\O\times\{|z|>1\}$, we set 
$$
\tn\Psi\tn_{p,r}=\left(\E\fint_0^{2\pi}|\Psi(re^{i\th})|^pd\th\right)^{1/p}.
$$

\begin{lemma}
\label{MESTL}
For all $\a,\eta\in\R$ with $\z=\a+\eta\le1$, all $\ve\in(0,1/2)$, 
all $p\ge2$ and all $T<t_\z$, 
there is a constant $C(\a,\eta,\ve,\L,p,T)<\infty$ such that, 
for all $c\in(0,1]$, all $\s\ge0$, all $\d_0\in(0,1/2]$ and all $t\le T$,
for all $r\ge1+c^{1/2-\ve}$, for $\rho=(1+r)/2$, we have, 
in the case $\z<1$,
\begin{equation}
\label{RZW}
\tn D\hat M_t1_{\{t\le T_0(\d_0)\}}\tn_{p,r}
\le\frac{C\sqrt c}r\left(1+r\sup_{s\le t}\tn D\hat\Psi_{s-}1_{\{s\le T_0(\d_0)\}}\tn_{p,\rho}\right)\left(\frac r{r-1}\right)
\end{equation}
and
\begin{align}
\notag
&\tn D(\hat M_t-\hat\Pi_t)1_{\{t\le T_0(\d_0)\}}\tn_{p,r}\\
\label{RTY}
&\q\q\q\q\le\frac{C\sqrt c}r\left(\sqrt{\d_0}+r\sup_{s\le t}\tn D\hat\Psi_{s-}1_{\{s\le T_0(\d_0)\}}\tn_{p,\rho}\right)\left(\frac r{r-1}\right)
+\frac{Cc}r\left(\frac r{r-1}\right)^2
\end{align}
while in the case $\z=1$ the same bounds hold with an additional factor $\left(\frac r{r-1}\right)^{1/2}$ on the right-hand side.
\end{lemma}
\begin{proof}
We restrict our account to the case $\z\in[0,1]$, omitting the minor modifications needed when $\z<0$.
The case $\z<0$ proceeds just as for $\z\in[0,1)$ but only for $T<t_\z$ and using the alternative split $Q=\tilde Q_0+\tilde Q_1$ and taking $r_s=e^{\d_s}r$.

Fix $t\le T<\infty$ and consider for $|z|>1$, the martingale $(M_x(z))_{0\le x\le t}$ given by
$$
M_x(z)=\int_{E(x)}D\tilde H_s(\th,z,\d_s)1_{\{v\le\L_s(\th),\,s\le T_0\}}\tilde\mu(d\th,dv,ds)
$$
where
$$
\tilde H_s(\th,z,\d)=e^{-\d}P(\d)H_s(\th,z),\q \d_s=\t_t-\t_s.
$$
Note that, for $p\ge2$ and $r>1$, 
$$
\|D\tilde H_s(\th,.,\d_s)\|_{p,r}
=\|e^{-\d_s}P_0(\d_s)P_1(\d_s)DH_s(\th,.)\|_{p,r}
\le Ce^{-\d_s}\left(\frac r{r-1}\right)\|H_s(\th,.)\|_{p,\rho_s}
$$
where $\rho_s=(r_s+1)/2$ and $r_s=e^{(1-\z)\d_s}r$.
Here we used Proposition \ref{PBDD} and the inequality \eqref{DEST}.
By Burkholder's inequality, for $p\ge2$ and all $|z|>1$,
\begin{equation}
\label{BUR2}
\|M_t(z)\|_p
\le C(p)\left(\|\<M(z)\>_t\|_{p/2}^{1/2}+\|(\Delta M(z))^*_t\|_p\right).
\end{equation}
For $t\le T_0$, we have $D\hat M_t(z)=M_t(z)$ so, on taking the $\|.\|_{p,r}$-norm in \eqref{BUR2}, we obtain
\begin{equation}
\label{DME}
\tn D\hat M_t1_{\{t\le T_0\}}\tn_{p,r}
\le\tn M_t\tn_{p,r}
\le C(p)\left(\tn\<M(.)\>_t\tn_{p/2,r}^{1/2}+\tn(\Delta M(.))^*_t\tn_{p,r}\right).
\end{equation}
Now
$$
\<M(z)\>_t=\int_0^{T_0\wedge t}\fint_0^{2\pi}|D\tilde H_s(\th,z,\t_t-\t_s)|^2\L_s(\th)d\th ds
$$
and 
\begin{equation}
\label{EB}
(\Delta M(z))^*_t
\le\sup_{s\le T_0\wedge t,\th\in[0,2\pi)}|D\tilde H_s(\th,z,\t_t-\t_s)|.
\end{equation}
Also
$$
|(\Delta M(z))^*_t|^p
\le\int_{E(t)}|D\tilde H_s(\th,z,\t_t-\t_s)|^p1_{\{v\le\L_s(\th),s\le T_0\}}\mu(d\th,dv,ds)
$$
so
$$
\|(\Delta M(z))^*_t\|^p_p
\le\E\int_0^{T_0\wedge t}\fint_0^{2\pi}|D\tilde H_s(\th,z,\t_t-\t_s)|^p\L_s(\th)d\th ds.
$$
We have $\L_s(\th)\le C/c$ for all $s\le T_0$ and $\th\in[0,2\pi)$.
Hence
\begin{equation}
\label{EAT}
\<M(z)\>_t\le\frac Cc\int_0^{T_0\wedge t}\fint_0^{2\pi}|D\tilde H_s(\th,z,\t_t-\t_s)|^2d\th ds
\end{equation}
and
\begin{align}
\label{QVD}
\|\<M(.)\>_t\|_{p/2,r}
&\le\frac Cc\int_0^{T_0\wedge t}\fint_0^{2\pi}\|D\tilde H_s(\th,.,\t_t-\t_s)\|^2_{p,r}d\th ds\\
\label{QVE}
&\le\frac Cc
\int_0^{T_0\wedge t}e^{-2(\t_t-\t_s)}
\left(\frac{r_s}{r_s-1}\right)^2
\fint_0^{2\pi}\|H_s(\th,.)\|^2_{p,\rho_s}d\th ds.
\end{align}
Similarly, 
\begin{align}
\notag
\tn(\Delta M(.))^*_t\tn_{p,r}^p
&\le\frac Cc\E\int_0^{T_0\wedge t}\fint_0^{2\pi}\|D\tilde H_s(\th,.,\t_t-\t_s)\|^p_{p,r}d\th ds\\
\label{MAXP}
&\le\frac Cc
\E\int_0^{T_0\wedge t}e^{-p(\t_t-\t_s)}
\left(\frac{r_s}{r_s-1}\right)^p
\fint_0^{2\pi}\|H_s(\th,.)\|^p_{p,\rho_s}d\th ds.
\end{align}
We will split the jump $\Delta(\th,z,c,\hat\phi)$ as the sum of several terms,
and thereby split $H_s(\th,z)$ and hence $M_t$ also as a sum of terms.
For each of these terms, we will use one of the inequalities \eqref{EAT}, \eqref{QVD}, \eqref{QVE}
and one of \eqref{EB}, \eqref{MAXP} to obtain a suitable upper bound for the right-side of \eqref{DME}.
These bounds will combine to prove the first claimed estimate.

Recall that $\hat\phi(z)=z+\hat\psi(z)$, so
\begin{equation}
\label{DELTA}
\Delta(\th,z,c,\hat\phi)
=\Delta_0(\th,z,c)+\left(e^{-c}\hat\psi(F_c(\th,z))-\hat\psi(z)\right)
\end{equation}
where
$$
\Delta_0(\th,z,c)
=e^{-c}F_c(\th,z)-z.
$$
We further split the second term by expanding in Taylor series, using an interpolation from $z$ to $F_c(\th,z)$.
For $u\in[0,1]$, define
$$
F_{c,u}(\th,z)=e^{uf_c(\th,z)}z,\q
f_c(\th,z)=\log(F_c(\th,z)/z).
$$
Then $F_{c,0}(\th,z)=z$ and $F_{c,1}(\th,z)=F_c(\th,z)$.
Fix $c$, $\th$ and $z$ and set
$$
f(u)=e^{-cu}\hat\psi(F_{c,u}(\th,z))
$$
then
$$
f^{(k)}(u)=e^{-cu}\sum_{j=0}^k\binom kj(-c)^{k-j}f_c(\th,z)^jD^j\hat\psi(F_{c,u}(\th,z)).
$$
Set $m=\lceil1/(8\ve)\rceil$ and recall that our constants $C$ are allowed to depend on $\ve$.
Then
\begin{align}
\notag
e^{-c}\hat\psi(F_c(\th,z))-\hat\psi(z)
&=f(1)-f(0)\\
\label{TAY}
&=\sum_{k=1}^m\frac{f^{(k)}(0)}{k!}+\int_0^1\frac{(1-u)^m}{m!}f^{(m+1)}(u)du
=\sum_{k=1}^{m+1}\Delta_k(\th,z,c,\hat\psi)
\end{align}
where, for $k=1,\dots,m$,
$$
\Delta_k(\th,z,c,\hat\psi)
=\frac1{k!}\sum_{j=0}^k\binom kj(-c)^{k-j}f_c(\th,z)^jD^j\hat\psi(z)
$$
and
$$
\Delta_{m+1}(\th,z,c,\hat\psi)
=\frac1{m!}\int_0^1(1-u)^me^{-cu}\sum_{j=0}^{m+1}\binom{m+1}j(-c)^{m+1-j}f_c(\th,z)^jD^j\hat\psi(F_{c,u}(\th,z))du.
$$
Let us write
$$
H^0_s(\th,z)=\Delta_0(\th,z,C_s(\th)),\q
H^k_s(\th,z)=\Delta_k(\th,z,C_s(\th),\hat\Phi_{s-}),\q k=1,\dots,m+1
$$
and
$$
\tilde H_s^k(\th,z,\d)=e^{-\d}P(\d)H_s^k(\th,z)
$$
and
$$
M^k_x(z)=\int_{E(x)}D\tilde H^k_s(\th,z,\t_t-\t_s)1_{\{v\le\L_s(\th),s\le T_0\}}\tilde\mu(d\th,dv,ds).
$$

We consider first the contribution of 
$$
\Delta_0(\th,z,c)
=e^{-c}F_c(\th,z)-z.
$$
We make the further split $\Delta_0=\Delta_{0,0}+\Delta_{0,1}$, where
$$
\Delta_{0,0}(\th,z,c)
=\frac{a_0(c)z}{e^{-i\th}z-1}
=a_0(c)\sum_{k=0}^\infty e^{i(k+1)\th}z^{-k}
$$
and
$$
\Delta_{0,1}(\th,z,c)
=e^{-c}F_c(\th,z)-z-\frac{a_0(c)z}{e^{-i\th}z-1}.
$$
We will exploit the more explicit form of $\Delta_{0,0}$, 
which is the main term as $c\to0$ under our particle assumptions \eqref{CAPACITY}, \eqref{NESTED} and \eqref{CONCENTRATED}, 
to obtain better estimates.
We have, with obvious notation,
$$
H_s^{0,0}(\th,z)
=a_0(C_s(\th))\sum_{k=0}^\infty e^{i(k+1)\th}z^{-k}
$$
so, for $\d\ge0$, in the case $\z\in[0,1]$,
$$
P_1(\d)DH_s^{0,0}(\th,z)
=a_0(C_s(\th))\sum_{k=1}^\infty e^{i(k+1)\th}(-k)e^{-\d q_1(k)}z^{-k}.
$$
By Proposition \ref{CNC}, $|a_0(c)|\le Cc$ for all $c$.
So, for $|z|=r$ and $\d\ge0$,
\begin{equation}
\label{FBN}
|P_1(\d)DH_s^{0,0}(\th,z)|
\le Cc\sum_{k=1}^\infty kr^{-k}
=\frac{Ccr}{(r-1)^2}
\end{equation}
and
\begin{align*}
\fint_0^{2\pi}|P_1(\d)DH_s^{0,0}(\th,z)|^2d\th
&\le Cc^2\fint_0^{2\pi}\left|\sum_{k=1}^\infty e^{i(k+1)\th}(-k)e^{-\d q_1(k)}z^{-k}\right|^2d\th\\
&\le Cc^2\sum_{k=1}^\infty k^2r^{-2k}
\le\frac{Cc^2r}{(r-1)^3}.
\end{align*}
Write $z_s=e^{(1-\z)(\t_t-\t_s)}z$.
We use \eqref{EAT} to see that
\begin{align*}
\<M^{0,0}(z)\>_t
&\le\frac Cc\int_0^{T_0\wedge t}\fint_0^{2\pi}|D\tilde H^{0,0}_s(\th,z,\t_t-\t_s)|^2d\th ds\\
&=\frac Cc\int_0^{T_0\wedge t}\fint_0^{2\pi}e^{-2(\t_t-\t_s)}|P_1(\t_t-\t_s)DH_s^{0,0}(\th,z_s)|^2d\th ds\\
&\le Cc\int_0^t\frac{e^{-2(\t_t-\t_s)}r_s}{(r_s-1)^3}ds\\
&\le\frac{Cc}{r^2}\int_0^t\left(\frac{r_s}{r_s-1}\right)^3ds
\le\frac{Cct}{r^2}\left(\frac r{r-1}\right)^3.
\end{align*}
For $\z<1$, using Lemma \ref{PUSHOUT}, we have the better bound
$$
\<M^{0,0}(z)\>_t
\le\frac{Cc(1+t)}{r^2}
\left(\frac r{r-1}\right)^2
$$
where we have absorbed the factor $(1-\z)^{-1}$ into the constant $C$.
We use \eqref{EB} and \eqref{FBN} to obtain, for $|z|=r>1$,
$$
|(\Delta M^{0,0}(z))^*_t|
\le\sup_{s\le T_0\wedge t,\th\in[0,2\pi)}e^{-(\t_t-\t_s)}|P_1(\t_t-\t_s)DH_s^{0,0}(\th,z_s)|
\le\frac{Cc}r\left(\frac r{r-1}\right)^2.
$$
On substituting the estimates for $\<M^{0,0}(z)\>_t$ and $(\Delta M^{0,0}(z))^*_t$ into \eqref{DME}, 
we obtain for $r\ge1+\sqrt c$ and $p\ge2$, for $\z<1$,
\begin{equation}
\label{M00N}
\tn M_t^{0,0}\tn_{p,r}
\le\frac{C\sqrt{c(1+t)}}r\left(\frac r{r-1}\right)
\end{equation}
while, for $\z=1$,
\begin{equation}
\label{M001N}
\tn M_t^{0,0}\tn_{p,r}
\le\frac{C\sqrt{c(1+t)}}r\left(\frac r{r-1}\right)^{3/2}.
\end{equation}

We turn to the contribution of $\Delta_{0,1}$.
Then
\begin{equation}
\label{MON}
H^{0,1}_s(\th,z)
=\Delta_{0,1}(\th,z,C_s(\th))
=e^{-C_s(\th)}\int_0^{C_s(\th)}Q_t(\th,z)dt
\end{equation}
where $Q_t(\th,z)=e^{i\th}Q_t(e^{-i\th}z)$ and $Q_t$ is given by \eqref{QTH}.
Hence, for $\d\ge0$ and $s\le T_0$,
$$
|P_1(\d)DH^{0,1}_s(\th,z)|
=e^{-C_s(\th)}\left|\int_0^{C_s(\th)}P_1(\d)DQ_t(\th,z)dt\right|
\le\int_0^{Cc}|P_1(\d)DQ_t(\th,z)|dt
$$
By Proposition \ref{LQP}, for $|z|>1$ and $t$ in the range of the last integral, 
\begin{equation}
\label{MOON}
|Q_t(\th,z)|\le\frac{C\sqrt t|z|}{|e^{-i\th}z-1|^2}.
\end{equation}
Then, for $|z|=r$, 
$$
\left(\fint_0^{2\pi}|Q_t(\th,z)|^2d\th\right)^{1/2}
\le\frac{C\sqrt t}r\left(\frac r{r-1}\right)^{3/2}
$$
so, for $\rho=(r+1)/2$,
\begin{align*}
&\left(\fint_0^{2\pi}|P_1(\d)DQ_t(\th,z)|^2d\th\right)^{1/2}
=\|P_1(\d)DQ_t(0,.)\|_{2,r}\\
&\q\q\le C\left(\frac r{r-1}\right)\|Q_t(0,.)\|_{2,\rho}
=C\left(\frac r{r-1}\right)\left(\fint_0^{2\pi}|Q_t(\th,\rho)|^2d\th\right)^{1/2}
\le\frac{C\sqrt t}r\left(\frac r{r-1}\right)^{5/2}.
\end{align*}
Hence
\begin{align*}
\notag
\left(\fint_0^{2\pi}|P_1(\d)DH_s^{0,1}(\th,z)|^2d\th\right)^{1/2}
&\le\int_0^{Cc}\left(\fint_0^{2\pi}|P_1(\d)DQ_t(\th,z)|^2d\th\right)^{1/2}dt\\
\label{THIN}
&\le\frac Cr\left(\frac r{r-1}\right)^{5/2}\int_0^{Cc}\sqrt tdt
=\frac Cr\left(\frac r{r-1}\right)^{5/2}c^{3/2}.
\end{align*}
Hence, we obtain, for $\z=1$,
\begin{align*}
\<M^{0,1}(z)\>_t
&\le\frac Cc\int_0^{T_0\wedge t}e^{-2(\t_t-\t_s)}\fint_0^{2\pi}|P_1(\t_t-\t_s)DH_s^{0,1}(\th,z_s)|^2d\th ds\\
&\le Cc^2\int_0^{T_0\wedge t}\frac{e^{-2(\t_t-\t_s)}}{r_s^2}\left(\frac{r_s}{r_s-1}\right)^5ds
\le\frac{Cc^2t}{r^2}\left(\frac r{r-1}\right)^5
\end{align*}
while, for $\z<1$, by Lemma \ref{PUSHOUT}, we have the better bound
$$
\<M^{0,1}(z)\>_t
\le\frac{Cc^2(1+t)}{r^2}\left(\frac r{r-1}\right)^4.
$$
From \eqref{MON} and \eqref{MOON}, and for $s\le T_0$, we have
$$
|H_s^{0,1}(\th,z)|\le\frac{Cc^{3/2}|z|}{|e^{-i\th}z-1|^2}
$$
so, for $p\ge2$ and $r>1$, we have, for $\z=1$,
\begin{align}
\notag
\tn(\Delta M^{0,1}(.))^*_t\tn_{p,r}^p
&\le\frac Cc\E\int_0^{T_0\wedge t}e^{-p(\t_t-\t_s)}\left(\frac{r_s}{r_s-1}\right)^p\fint_0^{2\pi}\|H_s^{0,1}(\th,.)\|^p_{p,\rho_s}d\th ds\\
\tn(\Delta M^{0,1}(.))^*_t\tn_{p,r}^p
\notag
&\le Cc^{3p/2-1}\int_0^t\frac{e^{-p(\t_t-\t_s)}}{r_s^p}\left(\frac{r_s}{r_s-1}\right)^{3p-1}ds
\le\frac{Cc^{3p/2-1}t}{r^p}\left(\frac r{r-1}\right)^{3p-1}
\end{align}
while, for $\z<1$,
\begin{equation*}
\label{DMB}
\tn(\Delta M^{0,1}(.))^*_t\tn_{p,r}^p
\le\frac{Cc^{3p/2-1}(1+t)}{r^p}\left(\frac r{r-1}\right)^{3p-2}.
\end{equation*}
On substituting the estimates for $\<M^{0,1}(z)\>_t$ and $(\Delta M^{0,1}(.))^*_t$ into \eqref{DME}, 
we obtain for $r\ge1+\sqrt c$ and $p\ge2$, for $\z<1$,
\begin{equation}
\label{M01}
\tn M_t^{0,1}\tn_{p,r}
\le\frac{Cc\sqrt{1+t}}r\left(\frac r{r-1}\right)^2
\le\frac{C\sqrt{c(1+t)}}r\left(\frac r{r-1}\right)
\end{equation}
while for $\z=1$
\begin{equation}
\label{M011}
\tn M_t^{0,1}\tn_{p,r}
\le\frac{Cc\sqrt{1+t}}r\left(\frac r{r-1}\right)^{5/2}
\le\frac{C\sqrt{c(1+t}}r\left(\frac r{r-1}\right)^{3/2}.
\end{equation}

We consider next, for $k=1,\dots,m$, the contribution of 
$$
\Delta_k(\th,z,c,\hat\psi)
=\frac1{k!}\sum_{j=0}^k\binom kj(-c)^{k-j}f_c(\th,z)^jD^j\hat\psi(z).
$$
We have 
$$
f_c(\th,z)=\int_0^cL_t(\th,z)dt
$$
where $L_t(\th,z)=e^{i\th}L_t(e^{-i\th}z)$ and $L_t$ is given by \eqref{QTH}.
Then
\begin{equation}
\label{LLL}
H^k_s(\th,z)
=\Delta_k(\th,z,C_s(\th),\hat\Psi_{s-})
=\frac1{k!}\sum_{j=0}^k\binom kj(-C_s(\th))^{k-j}\left(\int_0^{C_s(\th)}L_t(\th,z)dt\right)^jD^j\hat\Psi_{s-}(z)
\end{equation}
so
\begin{align*}
&P_1(\d)DH_s^k(\th,z)\\
&\q\q=\frac1{k!}(-C_s(\th))^kP_1(\d)D\hat\Psi_{s-}(z)\\
&\q\q\q\q+\frac1{k!}\sum_{j=1}^k\binom kj(-C_s(\th))^{k-j}\int_0^{C_s(\th)}\dots\int_0^{C_s(\th)}P_1(\d)D(L_{t_1,\dots,t_j}(\th,.)D^j\hat\Psi_{s-})(z)dt_1\dots dt_j
\end{align*}
where
$$
L_{t_1,\dots,t_j}(\th,z)=\prod_{i=1}^jL_{t_i}(\th,z).
$$
Hence, for $s\le T_0$,
\begin{align*}
&|P_1(\d)DH_s^k(\th,z)|\\
&\q\le Cc^k|P_1(\d)D\hat\Psi_{s-}(z)|
+C\sum_{j=1}^kc^{k-j}\int_0^{Cc}\dots\int_0^{Cc}|P_1(\d)D(L_{t_1,\dots,t_j}(\th,.)D^j\hat\Psi_{s-})(z)|dt_1\dots dt_j
\end{align*}
so
\begin{align*}
\left(\fint_0^{2\pi}|P_1(\t_t-\t_s)DH_s^k(\th,z)|^2d\th\right)^{1/2}
\le Cc^kh_s(z)+C\sum_{j=1}^kc^{k-j}\int_0^{Cc}\dots\int_0^{Cc}h_{s,t_1,\dots,t_j}(z)dt_1\dots dt_j
\end{align*}
where
$$
h_s(z)=|P_1(\t_t-\t_s)D\hat\Psi_{s-}(z)|,\q
h_{s,t_1,\dots,t_j}(z)=
\left(\fint_0^{2\pi}|P_1(\t_t-\t_s)D(L_{t_1,\dots,t_j}(\th,.)D^j\hat\Psi_{s-})(z)|^2d\th\right)^{1/2}.
$$
By Proposition \ref{LQP}, for $|z|=r\ge1+\sqrt c$ and $t\le Cc$,
\begin{equation*}
\label{FCE}
|L_t(z)|
\le\frac{C|z|}{|z-1|}
\end{equation*}
so
$$
\|L_{t_1,\dots,t_j}\|_{2,r}
\le C\left(\frac r{r-1}\right)^{j-1/2}
$$
and so by Proposition \ref{LFG},
for $j=1,\dots,k$ and $\rho=(r+1)/2$ and $\rho'=(3r+1)/4$,
\begin{align*}
&\|h_{s,t_1,\dots,t_j}\|_{p,r}
\le\|P_1(\t_t-\t_s)D\|_{p,\rho'\ra r}\|L_{t_1,\dots,t_j}\|_{2,\rho'}\|D^j\hat\Psi_{s-}\|_{p,\rho'}\\
&\q\q\le C\left(\frac r{r-1}\right)\left(\frac r{r-1}\right)^{j-1/2}\left(\frac r{r-1}\right)^{j-1}\|D\hat\Psi_{s-}\|_{p,\rho}
\le C\left(\frac r{r-1}\right)^{2k-1/2}\|D\hat\Psi_{s-}\|_{p,\rho}.
\end{align*}
Now
\begin{align*}
\<M^k(z)\>_t
&\le\frac Cc\int_0^{T_0\wedge t}e^{-2(\t_t-\t_s)}\fint_0^{2\pi}|P_1(\t_t-\t_s)DH_s^k(\th,z_s)|^2d\th ds\\
&\le\frac Cc\int_0^{T_0\wedge t}e^{-2(\t_t-\t_s)}
\left(c^kh_s(z_s)+\sum_{j=1}^kc^{k-j}\int_0^{Cc}\dots\int_0^{Cc}h_{s,t_1,\dots,t_j}(z_s)dt_1\dots dt_j\right)^2ds
\end{align*}
so, for $r\ge1+\sqrt c$,
\begin{align}
\notag
\|\<M^k(.)\>_t\|_{p/2,r}
&\le\frac Cc\int_0^{T_0\wedge t}e^{-2(\t_t-\t_s)}
\left(c^k\|h_s\|_{p,r_s}+\sum_{j=1}^kc^{k-j}\int_0^{Cc}\dots\int_0^{Cc}\|h_{s,t_1,\dots,t_j}\|_{p,r_s}dt_1\dots dt_j\right)^2ds\\
\notag
&\le Cc^{2k-1}\int_0^{T_0\wedge t}e^{-2(\t_t-\t_s)}\left(\frac{r_s}{r_s-1}\right)^{4k-1}\|D\hat\Psi_{s-}\|_{p,\rho_s}^2ds\\
\notag
&\le Cc\int_0^{T_0\wedge t}e^{-2(\t_t-\t_s)}\left(\frac{r_s}{r_s-1}\right)^3\|D\hat\Psi_{s-}\|_{p,\rho_s}^2ds.
\end{align}
For $p\ge2$ and $r>1$, we have
\begin{align*}
\tn(\Delta M^k(.))^*_t\tn_{p,r}^p
&\le\frac Cc\E\int_0^{T_0\wedge t}e^{-p(\t_t-\t_s)}
\fint_0^{2\pi}\|DH_s^k(\th,.)\|^p_{p,r_s}d\th ds\\
\end{align*}
and, from \eqref{LLL}, for $r\ge1+\sqrt c$, estimating as above, we get
$$
\|DH_s^k(\th,.)\|_{p,r}
\le Cc^k\left(\frac r{r-1}\right)^{2k-1/p}\|D\hat\Psi_{s-}\|_{p,\rho}
\le Cc\left(\frac r{r-1}\right)^{2-1/p}\|D\hat\Psi_{s-}\|_{p,\rho}
$$
so, for $r\ge1+\sqrt c$,
\begin{align}
\notag
\tn(\Delta M^k(.))^*_t\tn_{p,r}^p
&\le Cc^{p-1}\E\int_0^{T_0\wedge t}e^{-p(\t_t-\t_s)}\left(\frac{r_s}{r_s-1}\right)^{2p-1}\|D\hat\Psi_s\|^p_{p,\rho_s}ds\\
\label{DMKE}
&\le Cc^{p-1}\sup_{s\le t}\tn D\hat\Psi_s1_{\{s\le T_0\}}\tn^p_{p,\rho}\int_0^te^{-p(\t_t-\t_s)}\left(\frac{r_s}{r_s-1}\right)^{2p-1}ds
\end{align}
On substituting the estimates for $\<M^k(z)\>_t$ and $(\Delta M^k(.))^*_t$ into \eqref{DME}, and using Lemma \ref{PUSHOUT},
we obtain for $r\ge1+\sqrt c$ and $p\ge2$, for $\z<1$,
\begin{equation}
\label{Mk1}
\tn M_t^k\tn_{p,r}
\le C\sqrt{c(1+t)}\left(\frac r{r-1}\right)\sup_{s\le t}\tn D\hat\Psi_{s-}1_{\{s\le T_0\}}\tn_{p,\rho}
\end{equation}
while for $\z=1$
\begin{equation}
\label{Mk11}
\tn M_t^k\tn_{p,r}
\le C\sqrt{c(1+t)}\left(\frac r{r-1}\right)^{3/2}\sup_{s\le t}\tn D\hat\Psi_{s-}1_{\{s\le T_0\}}\tn_{p,\rho}.
\end{equation}

We consider finally the contribution of 
$$
\Delta_{m+1}(\th,z,c,\hat\psi)
=\frac1{m!}\int_0^1(1-u)^me^{-cu}\sum_{j=0}^{m+1}\binom{m+1}j(-c)^{m+1-j}f_c(\th,z)^jD^j\hat\psi(F_{c,u}(\th,z))du.
$$
Then
\begin{align*}
&H_s^{m+1}(\th,z)=\Delta_{m+1}(\th,z,C_s(\th),\hat\Psi_{s-})\\
&\q\q=\frac1{m!}\int_0^1(1-u)^me^{-C_s(\th)u}\sum_{j=0}^{m+1}\binom{m+1}j(-C_s(\th))^{m+1-j}f_{C_s(\th)}(\th,z)^jD^j\hat\Psi_{s-}(F_{C_s(\th),u}(\th,z))du.
\end{align*}
By Proposition \ref{CNC}, we have
$$
|f_c(\th,z)|\le\frac{Cc|z|}{|z-1|}.
$$
Hence, for $s\le T_0$,
\begin{align*}
&\|DH_s^{m+1}(\th,.)\|_{p,r}\\
&\q\q\le Cc^{m+1}\|D\hat\Psi_{s-}\|_{p,r}
+C\left(\frac r{r-1}\right)\sum_{j=1}^{m+1}c^{m+1-j}\|f_{C_s(\th)}(\th,z)^j
D^j\hat\Psi_{s-}(F_{C_s(\th),u}(\th,.))\|_{p,\rho'}\\
&\q\q\le Cc^{m+1}\left(\frac r{r-1}\right)^{2(m+1)}\|D\hat\Psi_{s-}\|_{p,\rho}
\end{align*}
where we have used the fact that $|F_{c,u}(\th,z)|\ge|z|$ to see that
$$
\|D^j\hat\Psi_{s-}(F_{C_s(\th),u}(\th,.))\|_{p,\rho'}
\le\|D^j\hat\Psi_{s-}\|_{p,\rho'}
\le C\left(\frac r{r-1}\right)^{j-1}\|D\hat\Psi_{s-}\|_{p,\rho}.
$$
Then, using \eqref{QVD}, we obtain
$$
\|\<M^{m+1}(.)\>_t\|_{p/2,r}
\le Cc^{2m+1}\int_0^{T_0\wedge t}e^{-2(\t_t-\t_s)}
\left(\frac{r_s}{r_s-1}\right)^{4(m+1)}\|D\hat\Psi_s\|_{p,\rho}^2ds.
$$
Hence we obtain for $\z=1$
\begin{align*}
\tn\<M^{m+1}(.)\>_t\tn_{p/2,r}
&\le Cc^{2m+1}(1+t)\left(\frac r{r-1}\right)^{4(m+1)}\sup_{s\le t}\tn D\hat\Psi_s1_{\{s\le T_0\}}\tn_{p,\rho}^2\\
&\le Cc(1+t)\left(\frac r{r-1}\right)^3\sup_{s\le t}\tn D\hat\Psi_s1_{\{s\le T_0\}}\tn_{p,\rho}^2
\end{align*}
while, using Lemma \ref{PUSHOUT}, for $\z<1$ we have
\begin{align*}
\tn\<M^{m+1}(.)\>_t\tn_{p/2,r}
&\le Cc^{2m+1}(1+t)\left(\frac r{r-1}\right)^{4m+3}\sup_{s\le t}\tn D\hat\Psi_s1_{\{s\le T_0\}}\tn_{p,\rho}^2\\
&\le Cc(1+t)\left(\frac r{r-1}\right)^2\sup_{s\le t}\tn D\hat\Psi_s1_{\{s\le T_0\}}\tn_{p,\rho}^2.
\end{align*}
Here we have used our choice of $m\ge1/(8\ve)$ and the assumption $r\ge1+c^{1/2-\ve}$ to see that
$$
c^{2m}\left(\frac r{r-1}\right)^{4m+1}\le C.
$$
The bound \eqref{DMKE} remains valid with $M^{m+1}$ in place of $M^k$.
Hence for $\z<1$
\begin{equation}
\label{Mm1}
\tn M^{m+1}_t\tn_{p,r}\le C\sqrt{c(1+t)}\left(\frac r{r-1}\right)
\sup_{s\le t}\tn D\hat\Psi_s1_{\{s\le T_0\}}\tn_{p,\rho}
\end{equation}
and for $\z=1$
\begin{equation}
\label{Mm11}
\tn M^{m+1}_t\tn_{p,r}\le C\sqrt{c(1+t)}\left(\frac r{r-1}\right)^{3/2}
\sup_{s\le t}\tn D\hat\Psi_s1_{\{s\le T_0\}}\tn_{p,\rho}.
\end{equation}
Now
$$
M_t=M^{0,0}_t+M^{0,1}_t+\sum_{k=1}^{m+1}M^k_t
$$
and we have shown that all terms on the right-hand side can be bounded by the right-hand side in \eqref{RZW},
so this first estimate is now proved.

It remains to show the second estimate.
Fix $t\ge0$ and consider, for $|z|>1$, the martingale $(\Pi_x(z))_{x\ge0}$ given by
$$
\Pi_x(z)
=\int_{E(x)}e^{-(\t_t-\t_s)}P(\t_t-\t_s)DH(\th,z)2c_s1_{\{v\le\l_s,\,s\le T_0\}}\tilde\mu(d\th,dv,ds).
$$
Set $\tilde M_x(z)=M^{0,0}_x(z)-\Pi_x(z)$.
Then 
\begin{align*}
&\tilde M_x(z)\\
&\q=\int_{E(x)}e^{-(\t_t-\t_s)}\left(a_0(C_s(\th))1_{\{v\le\L_s(\th)\}}-2c_s1_{\{v\le\l_s\}}\right)P(\t_t-\t_s)DH(\th,z)
1_{\{s\le T_0\}}\tilde\mu(d\th,dv,ds)
\end{align*}
and
$$
D(\hat M_t-\hat\Pi_t)=M_t-\Pi_t=\tilde M_t+M^{0,1}_t+\sum_{k=1}^{m+1}M^k_t.
$$
For all but the first term on the right, the bounds 
\eqref{M01},
\eqref{M011},
\eqref{Mk1},
\eqref{Mk11},
\eqref{Mm1},
\eqref{Mm11},
are sufficient for \eqref{RTY}.
It remains to show a suitable bound on $\tilde M_t$.
We use the estimate \eqref{JKL} to see that
\begin{align*}
&\<\tilde M(z)\>_t\\
&\q=\int_0^{T_0\wedge t}\int_0^\infty\fint_0^{2\pi}
e^{-2(\t_t-\t_s)}
\left|\left(a_0(C_s(\th))1_{\{v\le\L_s(\th)\}}-2c_s1_{\{v\le\l_s\}}\right)P(\t_t-\t_s)DH(\th,z)\right|^2d\th dvds\\
&\q\le\frac{Cc\d_0}{r^2}\int_0^t\left(\frac{r_s}{r_s-1}\right)^3ds.
\end{align*}
Otherwise we can proceed as for $M^{0,0}$ to arrive as the following estimates, which suffice for \eqref{RTY}.
For $\z<1$, we have
$$
\tn\tilde M_t\tn_{p,r}
\le\frac{C\sqrt{c\d_0}}r\left(\frac r{r-1}\right)
+\frac{Cc}r\left(\frac r{r-1}\right)^2
$$
while for $\z=1$
$$
\tn\tilde M_t\tn_{p,r}
\le\frac{C\sqrt{c\d_0}}r\left(\frac r{r-1}\right)^{3/2}
+\frac{Cc}r\left(\frac r{r-1}\right)^2.
$$
\end{proof}

\vfe
\subsection{Estimates for the drift terms}
We turn to the drift terms, beginning with estimates for the drift $(\hat\b,\b^\cp)$ of the ALE$(\a,\eta)$ process.
Recall that $(\cT_t)_{t\ge0}$ has drift given by
$$
\b^\cp(\hat\phi,\t)=\fint_0^{2\pi}c(\th,\hat\phi,\t)\l(\th,\hat\phi,\t)d\th
$$
where
$$
c(\th,\hat\phi,\t)=ce^{-\a\t}|\hat\phi'(e^{\s+i\th)}|^{-\a},\q
\l(\th,\hat\phi,\t)=c^{-1}e^{-\eta\t}|\hat\phi'(e^{\s+i\th)}|^{-\eta}.
$$

\begin{lemma}
\label{BEST}
For all $\z\in\R$ and all $T<t_\z$, there is a constant $C(\z,T)<\infty$ such that, 
for all $\d_0\in(0,1/2]$, all $t\le T$, all $\hat\phi\in\cS_1$ and all $\t\ge0$,
we have
$$
|\b^\cp(\hat\phi,\t)-e^{-\z\t_t}+\z e^{-\z\t_t}\psi^\cp_t|\le C\d_0^2
$$
whenever $|\psi^\cp_t|\le\d_0$ and $|\hat\psi'(e^{\s+i\th})|\le\d_0$ for all $\th$,
where $\psi^\cp_t=\t-\t_t$ and $\hat\psi(z)=\hat\phi(z)-z$.
\end{lemma}
\begin{proof}
We have
$$
c(\th,\hat\phi,\t)\l(\th,\hat\phi,\t)
=e^{-\z\t}|\hat\phi'(e^{\s+i\th)}|^{-\z}
=e^{-\z\t_t}e^{-\z\psi^\cp_t}|1+\hat\psi'(e^{\s+i\th)}|^{-\z}
$$
and, for $|w|\le1/2$,
$$
|1+w|^{-\z}=1-\z\re w+\ve(w),\q|\ve(w)|\le C|w|^2
$$
so
\begin{equation}
\label{CL}
c(\th,\hat\phi,\t)\l(\th,\hat\phi,\t)
=e^{-\z\t_t}\left(1-\z\psi^\cp_t-\z\re\hat\psi'(e^{\s+i\th})+\g_t(\th,\hat\phi,\t)\right)
\end{equation}
where
\begin{equation}
\label{CLL}
|\g_t(\th,\hat\phi,\t)|\le C\d_0^2
\end{equation}
whenever $|\psi^\cp_t|\le\d_0$ and $|\hat\psi'(e^{\s+i\th})|\le\d_0$ for all $\th$.
For $\hat\phi\in\cS_1$, $\hat\psi$ is holomorphic in $\{|z|>1\}$ and bounded at $\infty$, so
\begin{equation}
\label{PSINTZ}
\fint_0^{2\pi}\re\hat\psi'(e^{\s+i\th})d\th=0.
\end{equation}
The claimed estimate follows on integrating \eqref{CL} in $\th$.
\end{proof}

Recall that the drift of $(\hat\Phi_t)_{t\ge0}$ is given by
$$
\hat\b(\hat\phi,\t)(z)=\fint_0^{2\pi}\Delta(\th,z,c(\th,\hat\phi,\t),\hat\phi)\l(\th,\hat\phi,\t)d\th
$$
where
$$
\Delta(\th,z,c,\hat\phi)
=e^{-c}\hat\phi(F_c(\th,z))-\hat\phi(z),\q
F_c(\th,z)=e^{i\th}F_c(e^{-i\th}z).
$$

\begin{lemma}
\label{AESTL}
For all $\a,\eta\in\R$ and all $T<t_\z$, 
there is a constant $C(\a,\eta,\L,T)<\infty$ with the following property.
For all $c\in(0,1/C]$, all $\s>0$, all $\d_0\in(0,1/2]$, all $t\le T$, all $\hat\phi\in\cS_1$ and all $\t\ge0$, we have
\begin{equation}
\label{JJN}
|\hat\b(\hat\phi,\t)(\infty)+e^{-\z\t_t}(Q+1)\hat\psi(\infty)|
\le C(\d_0\sqrt c+\d_0^2)+C(c+\d_0)|\hat\psi(\infty)|
\end{equation}
whenever $|\psi^\cp_t|\le\d_0$ and $|\hat\psi'(e^{\s+i\th})|\le\d_0$ for all $\th$,
where $\psi^\cp_t=\t-\t_t$ and $\hat\psi(z)=\hat\phi(z)-z$.

Moreover, for all $\a,\eta\in\R$, all $\ve\in(0,1/2]$, all $p\ge2$ and all $T<\t_\z$, 
there is a constant $C(\a,\eta,\L,\ve,p,T)<\infty$ with the following property. 
For all $c\in(0,1/C]$, all $\s>0$, all $\d_0\in(0,1/2]$, all $t\le T$, all $\hat\phi\in\cS_1$ and all $\t\ge0$,
for all $r\ge1+c^{1/2-\ve}$ and $\rho=(1+r)/2$, we have 
\begin{align}
\notag
&\|D\hat\b(\hat\phi,\t)+e^{-\z\t_t}(Q+1)D\hat\psi\|_{p,r}\\
\notag
&\q\q\le\frac{C\d_0^2}r\left(\frac r{r-1}\right)
+\frac{C\d_0}r\left(\frac r{r-1}\right)\left(1+\log\left(\frac r{r-1}\right)\right)r\|D\hat\psi\|_{p,\rho}\\
&\q\q\q\q+\frac{C\d_0\sqrt c}r\left(\frac r{r-1}\right)^2+\frac{Cc}r\left(\frac r{r-1}\right)^2r\|D\hat\psi\|_{p,\rho}
\label{KKN}
\end{align}
whenever $|\psi^\cp_t|\le\d_0$ and $|\hat\psi'(e^{\s+i\th})|\le\d_0$ for all $\th$.
\end{lemma}
\begin{proof}
We use the split \eqref{DELTA} and the Taylor expansion \eqref{TAY} to write
$$
\Delta(\th,z,c,\hat\phi)
=\Delta_0(\th,z,c)+\sum_{k=1}^{m+1}\Delta_k(\th,z,c,\hat\psi)
$$
where $m=\lceil1/(8\ve)\rceil$.
We further split\footnote{\label{ADD}%
It is convenient to split $\Delta_0$ slightly differently to the split $\Delta_0=\Delta_{0,0}+\Delta_{0,1}$ used for the martingale term: 
where before we had $a_0(c)$ we now approximate by $2c$, putting an additional error into the remainder term $\tilde\Delta_0$.
}%
$$
\Delta_0(\th,z,c)
=\frac{2cz}{e^{-i\th}z-1}+\tilde\Delta_0(\th,z,c)
$$
and
$$
\Delta_1(\th,z,c,\hat\psi)
=c(D\hat\psi(z)-\hat\psi(z))+\tilde\Delta_1(\th,z,c,\hat\psi).
$$
Set
$$
\tilde\Delta(\th,z,c,\hat\phi)=\tilde\Delta_0(\th,z,c)+\tilde\Delta_1(\th,z,c,\hat\psi)+\sum_{k=2}^{m+1}\Delta_k(\th,z,c,\hat\psi)
$$
and note that
\begin{equation}
\label{EECO}
e^{-c}\hat\phi(F_c(\th,z))-\hat\phi(z)
=c\left(\frac{2z}{e^{-i\th}z-1}+D\hat\psi(z)-\hat\psi(z)\right)+\tilde\Delta(\th,z,c,\hat\psi).
\end{equation}
We use equation \eqref{CL} to write
\begin{align*}
\hat\b(\hat\phi,\t)(z)
&=\fint_0^{2\pi}\left(\frac{2z}{e^{-i\th}z-1}+D\hat\psi(z)-\hat\psi(z)\right)
c(\th,\hat\phi,\t)\l(\th,\hat\phi,\t)d\th\\
&\q\q+\fint_0^{2\pi}\tilde\Delta(\th,z,c(\th,\hat\phi,\t),\hat\psi)\l(\th,\hat\phi,\t)d\th\\
&=e^{-\z\t_t}\fint_0^{2\pi}\left(D\hat\psi(z)-\hat\psi(z)+\frac{2z}{e^{-i\th}z-1}\left(1-\z\psi^\cp_t-\z\re\hat\psi'(e^{\s+i\th})\right)\right)d\th\\
&\q\q+e^{-\z\t_t}\fint_0^{2\pi}\frac{2z}{e^{-i\th}z-1}\g_t(\th,\hat\phi,\t)d\th\\
&\q\q\q\q+(D\hat\psi(z)-\hat\psi(z))\fint_0^{2\pi}(c(\th,\hat\phi,\t)\l(\th,\hat\phi,\t)-e^{-\z\t_t})d\th\\
&\q\q\q\q\q\q+\fint_0^{2\pi}\tilde\Delta(\th,z,c(\th,\hat\phi,\t),\hat\psi)\l(\th,\hat\phi,\t)d\th.
\end{align*}
Now $\hat\psi(z)\to0$ as $z\to\infty$, so
\begin{align*}
&e^{-\z\t_t}\fint_0^{2\pi}\left(D\hat\psi(z)-\hat\psi(z)+\frac{2z}{e^{-i\th}z-1}\left(1-\z\psi^\cp_t-\z\re\hat\psi'(e^{\s+i\th})\right)\right)d\th\\
&\q\q\q\q
=e^{-\z\t_t}\left(D\hat\psi(z)-\hat\psi(z)-\z e^{-\s}D\hat\psi(e^\s z)\right)
=-e^{-\z\t_t}(Q+1)\hat\psi(z).
\end{align*}
Hence
\begin{align}
\notag
&\hat\b(\hat\phi,\t)(\infty)+e^{-\z\t_t}(Q+1)\hat\psi(\infty)\\
\notag
&\q\q=2e^{-\z\t_t}\fint_0^{2\pi}e^{i\th}\g_t(\th,\hat\phi,\t)d\th
-\hat\psi(\infty)\fint_0^{2\pi}(c(\th,\hat\phi,\t)\l(\th,\hat\phi,\t)-e^{-\z\t_t})d\th\\
\label{INTT}
&\q\q\q\q+\fint_0^{2\pi}\tilde\Delta(\th,\infty,c(\th,\hat\phi,\t),\hat\psi)\l(\th,\hat\phi,\t)d\th
\end{align}
and 
\begin{align}
\notag
&D\hat\b(\hat\phi,\t)(z)+e^{-\z\t_t}(Q+1)D\hat\psi(z)\\
\notag
&\q\q=-e^{-\z\t_t}\fint_0^{2\pi}\frac{2z}{(e^{-i\th}z-1)^2}\g_t(\th,\hat\phi,\t)d\th\\
\notag
&\q\q\q\q+(D^2\hat\psi(z)-D\hat\psi(z))\fint_0^{2\pi}(c(\th,\hat\phi,\t)\l(\th,\hat\phi,\t)-e^{-\z\t_t})d\th\\
\label{BEQ}
&\q\q\q\q\q\q+\fint_0^{2\pi}D\tilde\Delta(\th,z,c(\th,\hat\phi,\t),\hat\psi)\l(\th,\hat\phi,\t)d\th.
\end{align}
We will estimate the terms on the right-hand sides of \eqref{INTT} and \eqref{BEQ},
assuming from now on that $t$, $\hat\phi$ and $\t$ are chosen so that
$|\psi^\cp_t|\le\d_0$ and $|\hat\psi'(e^{\s+i\th})|\le\d_0$ for all $\th$.

From \eqref{CL} and \eqref{CLL}, we have $|\g_t(\th,\hat\phi,\t)|\le C\d_0^2$ and
$$
\left|\fint_0^{2\pi}(c(\th,\hat\phi,\t)\l(\th,\hat\phi,\t)-e^{-\z\t_t})d\th\right|\le C\d_0.
$$
We use \eqref{DINF} to see that
$$
\tilde\Delta(\th,\infty,c,\hat\psi)
=\Delta(\th,\infty,c,\hat\psi)-2ce^{i\th}+c\hat\psi(\infty)
=(a_0(c)-2c)e^{i\th}+(e^{-c}-1+c)\hat\psi(\infty).
$$
Write $c(\th)$ for $c(\th,\hat\phi,\t)$ and $\l(\th)$ for $\l(\th,\hat\phi,\t)$.
Then
$$
|c(\th)-c_t|\le C\d_0c,\q|\l(\th)-\l_t|\le C\d_0c^{-1}
$$
and, by Proposition \ref{CNC}, we have 
$$
|a_0(c(\th))-2c(\th)|\le Cc^{3/2},\q
|(a_0(c(\th))-2c(\th))-(a_0(c_t)-2c_t)|\le Cc^{3/2}\d_0.
$$
We can now estimate in \eqref{INTT} to obtain \eqref{JJN}.

It remains to prove \eqref{KKN}.
For $|z|=r>1$, we have
\begin{equation}
\label{GAME}
\left|\fint_0^{2\pi}\frac{2z}{(e^{-i\th}z-1)^2}\g_t(\th,\hat\phi,\t)d\th\right|
\le C\d_0^2\fint_0^{2\pi}\frac{|z|}{|e^{-i\th}z-1|^2}d\th
\le\frac{C\d_0^2}{r-1}.
\end{equation}
For $r>1$ and $\rho=(r+1)/2$, we have
$$
\|D^2\hat\psi\|_{p,r}
\le C\left(\frac r{r-1}\right)\|D\hat\psi\|_{p,\rho}
$$
so
\begin{equation}
\label{FAME}
\left\|(D^2\hat\psi(z)-D\hat\psi(z))\fint_0^{2\pi}(c(\th,\hat\phi,\t)\l(\th,\hat\phi,\t)-e^{-\z\t_t})d\th\right\|_{p,r}
\le C\d_0\left(\frac r{r-1}\right)\|D\hat\psi\|_{p,\rho}.
\end{equation}
It remains to deal with the final term in \eqref{BEQ}.
We first estimate the function obtained on replacing $c(\th,\hat\phi,\t)$ and $\l(\th,\hat\phi,\t)$ in that term by $c_t=ce^{-\a\t_t}$ and $\l_t=c^{-1}e^{-\eta\t_t}$.
Note that, in the case $F_c(z)=e^cz$ and $m=1$, the Taylor expansion \eqref{TAY} has the form
$$
e^{-c}\hat\phi(e^cz)-\hat\phi(z)
=c(D\hat\psi(z)-\hat\psi(z))+c^2\int_0^1(1-u)e^{-cu}(D^2\hat\psi(e^{cu}z)-2D\hat\psi(e^{cu}z)+\hat\psi(e^{cu}z))du.
$$
On the other hand, by Cauchy's theorem, 
$$
\fint_0^{2\pi}\hat\phi(F_c(\th,z))d\th=\hat\phi(e^cz).
$$
Hence, on integrating in $\th$ in \eqref{EECO}, we see that
$$
\fint_0^{2\pi}\tilde\Delta(\th,z,c,\hat\psi)d\th
=c^2\int_0^1(1-u)e^{-cu}(D^2\hat\psi(e^{cu}z)-2D\hat\psi(e^{cu}z)+\hat\psi(e^{cu}z))du
$$
so
$$
\fint_0^{2\pi}D\tilde\Delta(\th,z,c,\hat\psi)d\th
=c^2\int_0^1(1-u)e^{-cu}(D^3\hat\psi(e^{cu}z)-2D^2\hat\psi(e^{cu}z)+D\hat\psi(e^{cu}z))du
$$
and so, for $r>1$ and $\rho=(r+1)/2$,
\begin{equation}
\label{CCL}
\left\|\fint_0^{2\pi}D\tilde\Delta(\th,.,c_t,\hat\psi)\l_td\th\right\|_{p,r}
\le Cc\left(\frac r{r-1}\right)^2\|D\hat\psi\|_{p,\rho}.
\end{equation}

It remains to deal with the error made in replacing $c(\th,\hat\phi,\t)$ and $\l(\th,\hat\phi,\t)$ by $c_t$ and $\l_t$.
We make a further split\footnote{%
Thus $\tilde\Delta_{0,0}=\Delta_{0,1}$, as considered in estimating the martingale terms, 
and $\tilde\Delta_{0,1}$ is the additional error referred to in footnote \ref{ADD}.
}%
$$
\tilde\Delta_0(\th,z,c)=\tilde\Delta_{0,0}(\th,z,c)+\tilde\Delta_{0,1}(\th,z,c),\q
\tilde\Delta(\th,z,c)=\bar\Delta(\th,z,c)+\tilde\Delta_{0,1}(\th,z,c)
$$
where
$$
\tilde\Delta_{0,0}(\th,z,c)=e^{-c}F_c(\th,z)-z-\frac{a_0(c)z}{e^{-i\th}z-1},\q
\tilde\Delta_{0,1}(\th,z,c)=\frac{(a_0(c)-2c)z}{e^{-i\th}z-1}.
$$
We first estimate the $\tilde\Delta_{0,1}$ term.
Since $|\psi^\cp_t|\le\d_0$ and $|\hat\psi'(e^{\s+i\th})|\le\d_0$ for all $\th$, we have
$$
|c(\th,\hat\phi,\t)-c_t|\le Cc\d_0,\q |\l(\th,\hat\phi,\t)-\l_t|\le C\d_0/c.
$$
Hence, by Proposition \ref{CNC}, for $c\le1/C$,
$$
|(a_0(c(\th,\hat\phi,\t))-2c(\th,\hat\phi,\t))\l(\th,\hat\phi,\t)-(a_0(c_t)-2c_t)\l_t|\le C\d_0\sqrt c
$$
so
$$
|D\tilde\Delta_{0,1}(\th,z,c(\th,\hat\phi,\t))\l(\th,\hat\phi,\t)-D\tilde\Delta_{0,1}(\th,z,c_t)\l_t|
\le\frac{C\d_0\sqrt c|z|}{|e^{-i\th}z-1|^2}
$$
so, for $|z|=r>1$,
\begin{equation}
\label{GA}
\fint_0^{2\pi}|D\tilde\Delta_{0,1}(\th,z,c(\th,\hat\phi,\t))\l(\th,\hat\phi,\t)-D\tilde\Delta_{0,1}(\th,z,c_t)\l_t|d\th
\le\frac{C\d_0\sqrt c}{r-1}.
\end{equation}
By Proposition \ref{CNC}, 
for $c\le1/C$,
$$
|\tilde\Delta_{0,0}(\th,z,c)|\le\frac{Cc^{3/2}|z|}{|e^{-i\th}z-1|^2}
$$
and, for $c_1,c_2\in(0,c]$ and $|z|\ge1+\sqrt c$,
$$
|\tilde\Delta_{0,0}(\th,z,c_1)-\tilde\Delta_{0,0}(\th,z,c_2)|
\le\frac{C\sqrt c|c_1-c_2||z|}{|e^{-i\th}z-1|^2}
$$
so
$$
|\tilde\Delta_{0,0}(\th,z,c(\th,\hat\phi,\t))\l(\th,\hat\phi,\t)-\tilde\Delta_{0,0}(\th,z,c_t)\l_t|
\le\frac{C\d_0\sqrt c|z|}{|e^{-i\th}z-1|^2}
$$
so, for $|z|=r\ge1+\sqrt c$,
$$
\fint_0^{2\pi}|\tilde\Delta_{0,0}(\th,z,c(\th,\hat\phi,\t))\l(\th,\hat\phi,\t)-\tilde\Delta_{0,0}(\th,z,c_t)\l_t|d\th
\le\frac{C\d_0\sqrt c}{r-1}
$$
and so, for $r\ge1+2\sqrt c$,
\begin{equation}
\label{FA}
\left\|\fint_0^{2\pi}(D\tilde\Delta_{0,0}(\th,.,c(\th,\hat\phi,\t))\l(\th,\hat\phi,\t)-D\tilde\Delta_{0,0}(\th,.,c_t)\l_t)d\th\right\|_{p,r}
\le\frac{C\d_0\sqrt c}r\left(\frac r{r-1}\right)^2.
\end{equation}
We have
$$
\tilde\Delta_1(\th,z,c,\hat\psi)
=\left(\log\left(\frac{F_c(\th,z)}z\right)-c\right)D\hat\psi(z)
$$
so, by Proposition \ref{CNC},
for $c\le1/C$,
$$
|\tilde\Delta_1(\th,z,c,\hat\psi)|
\le\frac{Cc}{|e^{-i\th}z-1|}|D\hat\psi(z)|
$$
and, for $c_1,c_2\in(0,c]$ and $|z|\ge1+\sqrt c$,
$$
|\tilde\Delta_1(\th,z,c_1,\hat\psi)-\tilde\Delta_1(\th,z,c_2,\hat\psi)|
\le\frac{C|c_1-c_2|}{|e^{-i\th}z-1|}|D\hat\psi(z)|
$$
so
$$
|\tilde\Delta_1(\th,z,c(\th,\hat\phi,\t))\l(\th,\hat\phi,\t)-\tilde\Delta_1(\th,z,c_t)\l_t|
\le\frac{C\d_0}{|e^{-i\th}z-1|}|D\hat\psi(z)|
$$
so, for $|z|=r\ge1+\sqrt c$,
$$
\fint_0^{2\pi}|\tilde\Delta_1(\th,z,c(\th,\hat\phi,\t))\l(\th,\hat\phi,\t)-\tilde\Delta_1(\th,z,c_t)\l_t|d\th
\le C\d_0\left(1+\log\left(\frac r{r-1}\right)\right)|D\hat\psi(z)|
$$
and so, for $r\ge1+2\sqrt c$,
\begin{align}
\notag
&\left\|\fint_0^{2\pi}(D\tilde\Delta_1(\th,.,c(\th,\hat\phi,\t))\l(\th,\hat\phi,\t)-D\tilde\Delta_1(\th,.,c_t)\l_t)d\th\right\|_{p,r}\\
\label{EA}
&\q\q\q\q\q\q\q\q\q\q\q\le C\d_0\left(\frac r{r-1}\right)\left(1+\log\left(\frac r{r-1}\right)\right)\|D\hat\psi\|_{p,\rho}.
\end{align}
For $k=2,\dots,m$, we have
$$
\Delta_k(\th,z,c,\hat\psi)
=\frac1{k!}\sum_{j=0}^k\binom kj(-c)^{k-j}f_c(\th,z)^jD^j\hat\psi(z)
$$
where $f_c(\th,z)=\log(F_c(\th,z)/z)$.
By Proposition \ref{CNC}, for $c\le1/C$ and $|z|=r>1$,
$$
|f_c(\th,z)|\le\frac{Ccr}{|e^{-i\th}z-1|}
$$
and, for $c_1,c_2\in(0,c]$ and $|z|=r\ge1+\sqrt c$,
$$
|f_{c_1}(\th,z)-f_{c_2}(\th,z)|\le\frac{C|c_1-c_2|r}{|e^{-i\th}z-1|}
$$
so, for $j=0,1,\dots,k$,
$$
|c_1^{k-j}f_{c_1}(\th,z)^j-c_2^{k-j}f_{c_2}(\th,z)^j|\le\frac{Cc^{k-1}|c_1-c_2|r^j}{|e^{-i\th}z-1|^j}
$$
so
$$
|\Delta_k(\th,z,c,\hat\psi)|
\le Cc^k\sum_{j=0}^k\frac{r^j}{|e^{-i\th}z-1|^j}|D^j\hat\psi(z)|
$$
and
$$
|\Delta_k(\th,z,c_1,\hat\psi)-\Delta_k(\th,z,c_2,\hat\psi) |
\le Cc^{k-1}|c_1-c_2|\sum_{j=0}^k\frac{r^j}{|e^{-i\th}z-1|^j}|D^j\hat\psi(z)|
$$
so
$$
|\Delta_k(\th,z,c(\th,\hat\phi,\t))\l(\th,\hat\phi,\t)-\Delta_k(\th,z,c_t)\l_t|
\le Cc^{k-1}\d_0\sum_{j=0}^k\frac{r^j}{|e^{-i\th}z-1|^j}|D^j\hat\psi(z)|
$$
so
$$
\fint_0^{2\pi}|\Delta_k(\th,z,c(\th,\hat\phi,\t))\l(\th,\hat\phi,\t)-\Delta_k(\th,z,c_t)\l_t|d\th
\le Cc^{k-1}\d_0\left(\frac r{r-1}\right)^{k-1}\sum_{j=0}^k|D^j\hat\psi(z)|
$$
and so, for $r\ge1+2\sqrt c$,
\begin{align}
\notag
&\left\|\fint_0^{2\pi}(D\Delta_k(\th,.,c(\th,\hat\phi,\t))\l(\th,\hat\phi,\t)-D\Delta_k(\th,.,c_t)\l_t)d\th\right\|_{p,r}\\
\label{DA}
&\q\q\q\q\q\q\q\q
\le Cc^{k-1}\d_0\left(\frac r{r-1}\right)^{2k-1}\|D\hat\psi\|_{p,\rho}
\le Cc\d_0\left(\frac r{r-1}\right)^3\|D\hat\psi\|_{p,\rho}
\end{align}
where we used the inequality $\|\hat\psi\|_{p,r}\le C\|D\hat\psi\|_{p,r}$ in the $j=0$ term.

In the final step, we use our assumption that $r\ge1+c^{1/2-\ve}$ and our choice of $m=\lceil1/(8\ve)\rceil$ to see that
$$
c^m\left(\frac r{r-1}\right)^{2m+1+1/p}
\le Cc\left(\frac r{r-1}\right)^2.
$$
Recall that
$$
\Delta_{m+1}(\th,z,c,\hat\psi)
=\frac1{m!}\int_0^1(1-u)^me^{-cu}\sum_{j=0}^{m+1}\binom{m+1}j(-c)^{m+1-j}f_c(\th,z)^jD^j\hat\psi(F_{c,u}(\th,z))du
$$
and, for $|z|=r>1$, since $|F_{c,u}(\th,z)|\ge r$, by \eqref{PINF}, we find, for $\rho'=(3r+1)/4$,
$$
|D^j\hat\psi(F_{c,u}(\th,z))|\le C
\left(\frac r{r-1}\right)^{1/p}
\|D^j\hat\psi\|_{p,\rho'}.
$$
So, for $|z|=r>1$,
$$
|\Delta_{m+1}(\th,z,c,\hat\psi)|
\le Cc^{m+1}
\left(\frac r{r-1}\right)^{1/p}
\left(\frac r{|e^{-i\th}z-1|}\right)^{m+1}
\|D^{m+1}\hat\psi\|_{p,\rho'}
$$
so
\begin{align*}
&\fint_0^{2\pi}|\Delta_{m+1}(\th,z,c(\th,\hat\phi,\t))\l(\th,\hat\phi,\t)-\Delta_{m+1}(\th,z,c_t)\l_t|d\th\\
&\q\q\q\q\q\q\q\q
\q\q\q\q\q\q\q\q\q
\le Cc^m
\left(\frac r{r-1}\right)^{m+1/p}
\|D^{m+1}\hat\psi\|_{p,\rho'}
\end{align*}
and so
\begin{align}
\notag
&\left\|\fint_0^{2\pi}(D\Delta_{m+1}(\th,.,c(\th,\hat\phi,\t))\l(\th,\hat\phi,\t)-D\Delta_{m+1}(\th,.,c_t)\l_t)d\th\right\|_{p,r}\\
\label{CA}
&\q\q\q\q\q\q\q\q\q\q\q
\le Cc^m\left(\frac r{r-1}\right)^{2m+1+1/p}\|D\hat\psi\|_{p,\rho}
\le Cc\left(\frac r{r-1}\right)^2\|D\hat\psi\|_{p,\rho}.
\end{align}
The claimed estimate is obtained by combining \eqref{GAME}, \eqref{FAME}, \eqref{CCL}, \eqref{GA}, \eqref{FA}, \eqref{EA}, \eqref{DA} and \eqref{CA}.
\end{proof}

Recall that the drift term $(\hat A_t,A^\cp_t)$ in the interpolation formula \eqref{KID} is given by
\begin{align*}
\hat A_t(z)
&=\int_0^te^{-(\t_t-\t_s)}P(\t_t-\t_s)\left(\hat\b(\hat\Phi_s,\cT_s)+e^{-\z\t_s}(Q+1)\hat\Psi_s\right)(z)ds,\\
A^\cp_t
&=\int_0^te^{-\z(\t_t-\t_s)}\left(\b^\cp(\hat\Phi_s,\cT_s)-e^{-\z\t_s}+\z e^{-\z\t_s}\Psi^\cp_s\right)ds
\end{align*}
where $\Psi^\cp_s=\cT_s-\t_s$ and $\hat\Psi_s(z)=\hat\Phi_s(z)-z$.
Recall also that
$$
T_0(\d_0)=\inf\big\{t\in[0,t_\z):\sup_{\th\in[0,2\pi)}|\hat\Psi_t'(e^{\s+i\th})|>\d_0\text{ or }|\Psi^\cp_t|>\d_0\big\}.
$$

\begin{lemma}
\label{DRIFTBEST}
For all $\z\in\R$ and all $T<t_\z$, there is a constant $C(\z,T)<\infty$ such that, 
for all $\s>0$, all $\d_0\in(0,1/2]$ and all $t\le T_0(\d_0)\wedge T$, we have
$$
|A^\cp_t|\le C\d_0^2.
$$
\end{lemma}
\begin{proof}
For all $t\le T_0(\d_0)\wedge t_\z$ and all $\th$, 
we have $|\Psi^\cp_t|\le\d_0$ and $|\hat\Psi'_t(e^{\s+i\th})|\le\d_0$ for all $\th$.
Hence, by Lemma \ref{BEST}, for $t\le T_0(\d_0)\wedge T$,
$$
|A^\cp_t|\le e^{-\z\t_t}\int_0^te^{\z\t_s}|\b^\cp(\hat\Phi_s,\cT_s)-e^{-\z\t_s}+\z e^{-\z\t_s}\Psi^\cp_s|ds\le C\d_0^2.
$$
\end{proof}

\begin{lemma}
\label{DRIFTAEST}
For all $\a,\eta\in\R$ with $\z=\a+\eta\le1$ and all $T<t_\z$,
there is a constant $C(\a,\eta,\L,T)<\infty$ with the following property.
For all $c\in(0,1/C]$, all $\s>0$, all $\d_0\in(0,1/2]$ and all $t\le T$,
\begin{equation}
\label{AINF}
\sup_{s\le t\wedge T_0(\d_0)}|\hat A_s(\infty)|
\le C(\d_0\sqrt c+\d_0^2)+C(c+\d_0)\int_0^{t\wedge T_0(\d_0)}|\hat\Psi_s(\infty)|ds.
\end{equation}
Moreover, for all such $\a,\eta$ and $T$, for all $\ve\in(0,1/2]$ and all $p\ge2$,
there is a constant $C(\a,\eta,\L,\ve,p,T)<\infty$ with the following property.
For all $c\in(0,1/C]$, all $\s>0$, all $\d_0\in(0,1/2]$ and all $t\le T$,
for all $r\ge1+c^{1/2-\ve}$, for $\rho=(1+r)/2$, we have in the case $\z<1$ 
\begin{align}
\notag
&\tn D\hat A_t1_{\{t\le T_0(\d_0)\}}\tn_{p,r}\\
\notag
&\q\q\le\frac Cr\bigg(\d_0^2\left(1+\log\left(\frac r{r-1}\right)\right)+\d_0\sqrt c\left(\frac r{r-1}\right)\bigg)\\
\label{RZM}
&\q\q\q\q+C\bigg(\d_0\left(1+\log\left(\frac r{r-1}\right)\right)^2+c\left(\frac r{r-1}\right)\bigg)
\sup_{s\le t}\tn D\hat\Psi_s1_{\{s\le T_0(\d_0)\}}\tn_{p,\rho}
\end{align}
while, in the case $\z=1$,
\begin{align}
\notag
&\tn D\hat A_t1_{\{t\le T_0(\d_0)\}}\tn_{p,r}\\
\notag
&\q\q\le\frac Cr\bigg(\d_0^2\left(\frac r{r-1}\right)+\d_0\sqrt c\left(\frac r{r-1}\right)^2\bigg)\\
\label{RTM}
&\q\q\q\q+C\bigg(\d_0\left(\frac r{r-1}\right)\left(1+\log\left(\frac r{r-1}\right)\right)+c\left(\frac r{r-1}\right)^2\bigg)
\sup_{s\le t}\tn D\hat\Psi_s1_{\{s\le T_0(\d_0)\}}\tn_{p,\rho}.
\end{align}
\end{lemma}
\begin{proof}
The estimate \eqref{AINF} follows immediately from \eqref{JJN}.
We write the argument for the case $\z\in[0,1]$, leaving to the reader the modifications needed for $\z<0$.
By Lemma \ref{PBDD}, for any holomorphic function $\psi$ on $\{|z|>1\}$ bounded at $\infty$, for all $r>1$, 
$$
\|P(\d_s)\psi\|_{p,r}=\|P_1(\d_s)\psi\|_{p,r_s}\le C\|\psi\|_{p,r_s}
$$
where $r_s=re^{(1-\z)\d_s}$.
Hence
\begin{align*}
\|D\hat A_t\|_{p,r}
&=\left\|\int_0^te^{-\d_s}P(\d_s)\left(D\hat\b(\hat\Phi_s,\cT_s)+e^{-\z\t_s}(Q+1)D\hat\Psi_s\right)ds\right\|_{p,r}\\
&\le C\int_0^t\left\|D\hat\b(\hat\Phi_s,\cT_s)+e^{-\z\t_s}(Q+1)D\hat\Psi_s\right\|_{p,r_s}ds.
\end{align*}
On the event $\{t\le T_0\}$, we have $|\Psi^\cp_s|\le\d_0$ and $|\hat\Psi'_s(e^{\s+i\th})|\le\d_0$ for all $\th$ and all $s<t$.
Hence, by Lemma \ref{AESTL},
\begin{align*}
&\tn D\hat A_t1_{\{t\le T_0\}}\tn_{p,r}\\
&\q\q\le\frac{C\d_0^2}r\int_0^t\left(\frac{r_s}{r_s-1}\right)ds
+C\d_0\sup_{s<t}\tn D\hat\Psi_s1_{\{s\le T_0\}}\tn_{p,\rho}
\int_0^t\left(\frac{r_s}{r_s-1}\right)\left(1+\log\left(\frac{r_s}{r_s-1}\right)\right)ds\\
&\q\q\q\q+\frac{C\d_0\sqrt c}r\int_0^t\left(\frac{r_s}{r_s-1}\right)^2ds
+Cc\sup_{s<t}\tn D\hat\Psi_s1_{\{s\le T_0\}}\tn_{p,\rho}\int_0^t\left(\frac{r_s}{r_s-1}\right)^2ds.
\end{align*}
In the case $\z=1$, we have $r_s=r$ for all $s$, so this is the claimed estimate.
For $\z<1$, the claimed estimates follow by Lemma \ref{PUSHOUT}.
\end{proof}

\vfe
\section{Bulk scaling limit for ALE$(\a,\eta)$}
\label{BSL}
Recall that we write our ALE($\a,\eta$) process $(\Phi_t)_{t\ge0}$ in (Schlicht function, capacity) coordinates 
$(\hat\Phi_t,\cT_t)$, and that we set
$$
\hat\Psi_t(z)=\hat\Phi_t(z)-\hat\phi_t(z),\q \Psi^\cp_t=\cT_t-\t_t
$$
where $(\hat\phi_t,\t_t)_{t<t_\z}$ is the disk solution to the LK$(\z)$ equation with initial capacity $\t_0=0$.
We obtained the following interpolation formula \eqref{KID} 
$$
\hat\Psi_t(z)=\hat M_t(z)+\hat A_t(z),\q
\Psi^\cp_t=M^\cp_t+A^\cp_t
$$
and have estimated the terms on the right-hand sides in the preceding section.
We now put these estimates together to obtain first $L^p$-estimates and then pointwise high-probability estimates
which allow us to prove Theorems \ref{DISK} and \ref{DISC}.

\subsection{$L^p$-estimates}
Recall that
$$
T_0(\d_0)=\inf\big\{t\in[0,t_\z):\sup_{\th\in[0,2\pi)}|\hat\Psi_t'(e^{\s+i\th})|>\d_0\text{ or }|\Psi^\cp_t|>\d_0\big\}.
$$

\begin{proposition}
\label{CAPE}
For all $\a,\eta\in\R$, all $p\ge2$ and all $T<t_\z$, 
there is a constant $C(\a,\eta,p,T)<\infty$ such that, for all $c\in(0,1]$ and all $\d_0\in(0,1/2]$,
$$
\big\|\sup_{t\le T\wedge T_0(\d_0)}|\Psi^\cp_t|\big\|_p\le C(\sqrt c+\d_0^2)
$$
and
$$
\big\|\sup_{t\le T\wedge T_0(\d_0)}|\Psi^\cp_t-\Pi^\cp_t|\big\|_p
\le C(c+\sqrt{c\d_0}+\d_0^2)
$$
and
$$
\big\|\sup_{t\le T\wedge T_0(\d_0)}|\hat\Psi_t(\infty)|\big\|_p\le C(\sqrt c+\d_0^2)
$$
and
$$
\big\|\sup_{t\le T\wedge T_0(\d_0)}|\hat\Psi_t(\infty)-\hat\Pi_t(\infty)|\big\|_p
\le C(c+\sqrt{c\d_0}+\d_0^2).
$$
\end{proposition}
\begin{proof}
The first two estimates follow immediately from Lemmas \ref{NEST} and \ref{DRIFTBEST}.
From Lemmas \ref{MESTI} and \ref{DRIFTAEST}, we obtain, for all $t\le T$,
$$
\big\|\sup_{s\le t\wedge T_0(\d_0)}|\hat\Psi_t(\infty)|\big\|_p^p
\le C(\sqrt c+\d_0^2)^p+C\int_0^t\|\hat\Psi_s(\infty)1_{\{s\le T_0(\d_0)\}}\|_p^pds
$$
from which the third estimate follows by Gronwall's lemma.
The fourth estimate follows from the third, together with Lemmas \ref{MESTI} and \ref{DRIFTAEST}.
\end{proof}

The following estimates follow immediately from Lemmas \ref{MESTL} and \ref{DRIFTAEST}.

\begin{proposition}
\label{DFA}
For all $\a,\eta\in\R$ with $\z=\a+\eta\le1$, all $\ve\in(0,1/2]$, all $p\ge2$ and all $T<t_\z$,
there is a constant $C(\a,\eta,\L,\ve,p,T)<\infty$ with the following property.
For all $c\in(0,1]$, all $\d_0\in(0,1/2]$, all $r\ge1+c^{1/2-\ve}$ and all $t\le T$, setting $\rho=(1+r)/2$, we have
\begin{equation}
\label{PDF}
\tn D\hat\Psi_t1_{\{t\le T_0(\d_0)\}}\tn_{p,r}
\le\d(r)+\bar\d(r)\sup_{s\le t}\tn D\hat\Psi_{s-}1_{\{s\le T_0(\d_0)\}}\tn_{p,\rho}
\end{equation}
where, in the case $\z<1$,
\begin{align}
\notag
\d(r)
&=\frac Cr\bigg(\sqrt c\left(\frac r{r-1}\right)+\d_0^2\left(1+\log\left(\frac r{r-1}\right)\right)\bigg),\\
\label{DRE}
\bar\d(r)
&=C\bigg(\sqrt c\left(\frac r{r-1}\right)+\d_0\left(1+\log\left(\frac r{r-1}\right)\right)^2 \bigg)
\end{align}
while, in the case $\z=1$,
\begin{align}
\notag
\d(r)
&=\frac Cr\bigg(\sqrt c\left(\frac r{r-1}\right)^{3/2}+\sqrt c\d_0\left(\frac r{r-1}\right)^2+\d_0^2\left(\frac r{r-1}\right)\bigg),\\
\label{DRF}
\bar\d(r)
&=C\bigg(\sqrt c\left(\frac r{r-1}\right)^{3/2}+\d_0\left(\frac r{r-1}\right)\left(1+\log\left(\frac r{r-1}\right)\right)\bigg).
\end{align}
\end{proposition}

The preceding estimate may be improved by an iterative argument to obtain the following result.

\begin{proposition}
\label{II}
For all $\a,\eta\in\R$ with $\z=\a+\eta\le1$, all $\ve\in(0,1/2]$, all $p\ge2$ and all $T<t_\z$,
there is a constant $C(\a,\eta,\L,\ve,p,T)<\infty$ with the following property.
In the case $\z<1$, for all $c\in(0,1]$, all $r,e^\s\ge1+c^{1/2-\ve}$ and all $t\le T$, 
for all $\nu\in(0,\ve/2]$, setting $\d_0=c^{1/2-\nu}e^\s/(e^\s-1)$, we have
\begin{equation*}
\tn D\hat\Psi_t1_{\{t\le T_0(\d_0)\}}\tn_{p,r}
\le\frac{C\sqrt c}r\left(\frac r{r-1}\right)+\frac{Cc^{1-2\nu}}r\left(\frac{e^\s}{e^\s-1}\right)^2\left(1+\log\left(\frac r{r-1}\right)\right).
\end{equation*}
Moreover, in the case $\z=1$ and $\ve\le1/5$, for all $c\in(0,1]$, all $r,e^\s\ge1+c^{1/5-\ve}$ and all $t\le T$, 
for $\nu\in(0,\ve]$, setting $\d_0=c^{1/2-\nu}(e^\s/(e^\s-1))^{3/2}$, we have
\begin{equation}
\label{Z1}
\tn D\hat\Psi_t1_{\{t\le T_0(\d_0)\}}\tn_{p,r}
\le\frac{C\sqrt c}r\left(\frac r{r-1}\right)^{3/2}+\frac{Cc^{1-2\nu}}r\left(\frac{e^\s}{e^\s-1}\right)^3\left(\frac r{r-1}\right).
\end{equation}
\end{proposition}
\begin{proof}
We begin with a crude estimate which allows us to restrict further consideration to small values of $c$.
The function $\hat\Phi_t(z)$ is univalent on $\{|z|>1\}$, with $\hat\Phi_t(z)\sim z$ as $z\to\infty$. 
So, by a standard distortion estimate, for all $|z|=r>1$,
$$
|\hat\Phi_t'(z)-1|
\le\frac1{r^2-1}
$$
and so
\begin{equation}\label{DIST}
\|D\hat\Psi_t\|_{p,r}
=r\|\hat\Phi_t'-1\|_{p,r}
\le\frac 1{r-1}.
\end{equation}
It is straightforward to check that this implies the claimed estimates in the case where $c>1/C$, for any given constant
$C$ of the allowed dependence.
Hence it will suffice to consider the case where $c\le1/C$.

Consider first the case $\z<1$.
On substituting the chosen value of $\d_0$ in \eqref{DRE}, we obtain
\begin{align*}
\d(r)
&=\frac Cr\bigg(\sqrt c\left(\frac r{r-1}\right)+c^{1-2\nu}
\left(\frac{e^\s}{e^\s-1}\right)^2
\left(1+\log\left(\frac r{r-1}\right)\right)\bigg),\\
\bar\d(r)
&=C\bigg(\sqrt c\left(\frac r{r-1}\right)+c^{1/2-\nu}
\left(\frac{e^\s}{e^\s-1}\right)
\left(1+\log\left(\frac r{r-1}\right)\right)^2\bigg).
\end{align*}
Note that, for $\rho=(1+r)/2$, we have $\d(\rho)\le2\d(r)$ and $\bar\d(\rho)\le2\bar\d(r)$.
Note also that, for $r\ge1+c^{1/2-\ve/2}$ and $e^\s\ge1+c^{1/2-\ve}$, for all sufficiently small $c$, 
$$
\bar\d(r)\le Cc^{\ve/2}+Cc^{\ve/2}(1+\log(1/c))^2\le c^{\ve/3}\le1 .
$$
We restrict to such $c$.
Set $C_0=1$ and for $k\ge0$ define recursively $C_{k+1}=2^{k+1}C_k+1$.
We will show that, for all $k\ge0$, 
all $r\ge1+2^kc^{1/2-\ve/2}$ and all $t\le T$,
\begin{equation}
\label{induchop}
\tn D\hat\Psi_t1_{\{t\le T_0(\d_0)\}}\tn_{p,r}
\le C_k\left(\frac{\bar\d(r)^k}{r-1}+\d(r)\right).
\end{equation}
The case $k=0$ is implied by \eqref{DIST}.
Suppose inductively that \eqref{induchop} holds for $k$, for all $r\ge1+2^kc^{1/2-\ve/2}$ and all $t\le T$.
Take $r\ge1+2^{k+1}c^{1/2-\ve/2}$ and $t\le T$. 
Then $\rho=(r+1)/2\ge1+2^k c^{1/2-\ve/2}$ so, for all $s\le t$,
$$
\tn D\hat\Psi_s1_{\{s\le T_0\}}\tn_{p,\rho}
\le C_k\left(\frac{\bar\d(\rho)^k}{\rho-1}+\d(\rho)\right)\le2^{k+1}C_k\left( \frac{\bar\d(r)^k}{r-1}+\d(r)\right).
$$
Since $r\ge1+c^{1/2-\ve/2}$, we can use Proposition \ref{DFA} with $\ve$ replaced by $\ve/2$ 
and substitute the last inequality into \eqref{PDF} to obtain
$$
\tn D\hat\Psi_t1_{\{t\le T_0(\d_0)\}}\tn_{p,r} 
\le2^{k+1}C_k\left(\frac{\bar\d(r)^{k+1}}{r-1}+\bar\d(r)\d(r)\right)+\d(r)
\le C_{k+1}\left(\frac{\bar\d(r)^{k+1}}{r-1}+\d(r)\right).
$$
Hence \eqref{induchop} holds for $k+1$ and the induction proceeds.
Choose now $k=\lceil3/\ve\rceil$. 
Then
$$
\frac{\bar\d(r)^k}{r-1}\le\frac{c^{\ve k/3}}{r-1}\le\frac c{r-1}\le\d(r).
$$
For $c$ sufficiently small, we have $c^{\ve/2}\le 2^{-k/2}$.
Then, for all $r\ge1+c^{1/2-\ve}$, we have $r\ge1+2^kc^{1/2-\ve/2}$, so we obtain the claimed bound
$$
\tn D\hat\Psi_t1_{\{t\le T_0(\d_0)\}}\tn_{p,r}
\le C_k\left(\frac{\bar\d(r)^k}{r-1}+\d(r)\right)
\le2C_k\d(r).
$$

We turn to the case $\z=1$. 
We substitute the chosen value of $\d_0$ into \eqref{DRF} to obtain
\begin{align*}
\d(r)
&=\frac Cr\bigg(\sqrt c\left(\frac r{r-1}\right)^{3/2}+c^{1-2\nu}
\left(\frac{e^\s}{e^\s-1}\right)^3
\left(\frac r{r-1}\right)\bigg),\\
\bar\d(r)
&=C\bigg(\sqrt c\left(\frac r{r-1}\right)^{3/2}+c^{1/2-\nu}
\left(\frac{e^\s}{e^\s-1}\right)^{3/2}
\left(\frac r{r-1}\right)\left(1+\log\left(\frac r{r-1}\right)\right)\bigg).
\end{align*}
Here we have simplified the expression for $\d(r)$ using the inequality
$$
\sqrt c\d_0\left(\frac r{r-1}\right)^2
=\sqrt c\left(\frac r{r-1}\right)^{3/2}\d_0\left(\frac r{r-1}\right)^{1/2}
\le c\left(\frac r{r-1}\right)^3+\d_0^2\left(\frac r{r-1}\right)
$$
where the first term on the right can be dropped because $r\ge1+c^{1/3}$.
For $\rho=(r+1)/2$, we now have modified inequalities $\d(\rho)\le2^{3/2}\d(r)$ and $\bar\d(\rho)\le2^{3/2}\bar\d(r)$.
Also, for $r\ge1+c^{1/5-\ve/2}$ and $e^\s\ge1+c^{1/5-\ve}$, and all sufficiently small $c$, 
$$
\bar\d(r)\le Cc^{1/5}+Cc^{2\ve-\nu}(1+\log(1/c))\le c^\ve\le1 .
$$
We restrict to such $c$.
Set $C_0=1$ and for $k\ge0$ define now recursively $C_{k+1}=2^{3k/2+1}C_k+1$.
Then, by an analogous inductive argument, we obtain, for all $k\ge0$, all $t\le T$ and all $r\ge1+2^kc^{1/5}$,
\begin{equation*}
\label{induchop1}
\tn D\hat\Psi_t1_{\{t\le T_0(\d_0)\}}\tn_{p,r}
\le C_k\left(\frac{\bar\d(r)^k}{r-1}+\d(r)\right).
\end{equation*}
Choose now $k=\lceil 1/\ve\rceil$ and assume that $r\ge1+c^{1/5-\ve}$.
Then
$$
\frac{\bar\d(r)^k}{r-1}\le\frac{c^{\ve k}}{r-1}\le\frac c{r-1}\le\d(r). 
$$
and, for $c$ sufficiently small, we have $c^\ve\le2^{-k}$, so $r\ge1+2^kc^{1/5}$ and so
$$
\tn D\hat\Psi_t1_{\{t\le T_0(\d_0)\}}\tn_{p,r}\le2C_k\d(r).
$$
\end{proof}

We note also the following estimates, which are deduced from \eqref{RTY}, \eqref{RZM} and \eqref{RTM} using the estimates of Proposition \ref{II}

\begin{proposition}
\label{III}
For all $\a,\eta\in\R$ with $\z=\a+\eta\le1$, all $\ve\in(0,1/2]$, all $p\ge2$ and all $T<t_\z$,
there is a constant $C(\a,\eta,\L,\ve,p,T)<\infty$ with the following property.
In the case $\z<1$, for all $c\in(0,1]$, all $r,e^\s\ge1+c^{1/2-\ve}$ and all $t\le T$, 
for all $\nu\in(0,\ve/2]$, setting $\d_0=c^{1/2-\nu}e^\s/(e^\s-1)$, we have
\begin{equation*}
\tn D(\hat\Psi_t-\hat\Pi_t)1_{\{t\le T_0(\d_0)\}}\tn_{p,r}
\le\frac Cr\left(c\left(\frac r{r-1}\right)^2+\sqrt{c\d_0}\left(\frac r{r-1}\right)+\d_0^2\left(1+\log\left(\frac r{r-1}\right)\right)\right).
\end{equation*}
Moreover, in the case $\z=1$ and $\ve\le1/5$, for all $c\in(0,1]$, all $r,e^\s\ge1+c^{1/5-\ve}$ and all $t\le T$, 
for $\nu\in(0,\ve]$, setting $\d_0=c^{1/2-\nu}(e^\s/(e^\s-1))^{3/2}$, we have
\begin{align*}
\tn D(\hat\Psi_t-\hat\Pi_t)1_{\{t\le T_0(\d_0)\}}\tn_{p,r}
&\le\frac Cr\left(
\d_0^2\left(\frac r{r-1}\right)
+\sqrt{c\d_0}\left(\frac r{r-1}\right)^{3/2}
+c\left(\frac r{r-1}\right)^3\right.\\
&\q+\left.\bigg(
\d_0^3\left(\frac r{r-1}\right)^2
+\d_0\sqrt c\left(\frac r{r-1}\right)^{5/2}
\bigg)
\left(1+\log\left(\frac r{r-1}\right)\right)
\right).
\end{align*}
\end{proposition}

We turn now to some estimates needed for the discrete-time results Theorems \ref{DISC} and \ref{DISCFLUCT}.
Write $\cV_t$ for the number of particles added by time $t$ and define for $t<t_\z$
$$
\nu_t=\a^{-1}((1+\z t)^{\a/\z}-1).
$$
It is straightforward to see that, for all $\a,\eta\in\R$, we have $\nu_t\to n_\a$ as $t\to t_\z$.
Also
$$
\cV_t=\int_{E(t)}1_{\{v\le\L_s(\th)\}}\mu(d\th,dv,ds),\q \nu_t=\int_0^te^{-\eta\t_s}ds.
$$
The following may be shown, 
either by a variation of the argument leading to Proposition \ref{CAPE}, or directly by martingale estimates.
The details are left to the reader.

\begin{proposition}
\label{DTSA}
For all $\a,\eta\in\R$, all $p\ge2$ and all $T<t_\z$, 
there is a constant $C(\a,\eta,p,T)<\infty$ such that, for all $c\in(0,1]$ and all $\d_0\in(0,1/2]$,
$$
\big\|\sup_{t\le T\wedge T_0(\d_0)}|c\cV_t-\nu_t|\big\|_p\le C(\sqrt c+\d_0^2).
$$
\end{proposition}

We can also improve on the estimate of $\cT_t$ by $\t_t$ in Proposition \ref{CAPE}.
Define, for $c\cV_t<n_\a$,
$$
\tilde\cT_t=\t^\disc_{\cV_t}
$$
where $\t_n^\disc=\a^{-1}\log(1+\a cn)$ as at \eqref{TDISC}.
We leave any modifications needed for the case $\a=0$ to the reader.
By allowing $\tilde\cT_t$ to depend on the random time-scale of particle arrivals, 
we remove the main source of error when estimating $\cT_t$ by $\t_t$.

\begin{proposition}
\label{DTSB}
For all $\a,\eta\in\R$, all $p\ge2$, all $T<t_\z$ and all $N<n_\a$, 
there is a constant $C(\a,\eta,p,T,N)<\infty$
such that, for all $c\le1/C$ and all $\d_0\in(0,1/2]$,
$$
\big\|\sup_{t\le T\wedge T(\d_0),\,c\cV_t\le N}|\cT_t-\tilde\cT_t|\big\|_p\le C(c+\d_0^2).
$$
\end{proposition}
\begin{proof}
Set 
$$
\tilde C_t=\t^\disc_{\cV_{t-}+1}-\t^\disc_{\cV_{t-}}=\a^{-1}\log\left(1+\frac{\a c}{1+\a c\cV_{t-}}\right).
$$
Then
$$
\tilde\cT_t=\int_{E(t)}\tilde C_s1_{\{v\le\L_s(\th)\}}\mu(d\th,dv,ds)
$$
so
\begin{align*}
\cT_t-\tilde\cT_t
&=\int_{E(t)}(C_s(\th)-\tilde C_s)1_{\{v\le\L_s(\th)\}}\mu(d\th,dv,ds)\\
&=\int_{E(t)}(C_s(\th)-\tilde C_s)1_{\{v\le\L_s(\th)\}}\tilde\mu(d\th,dv,ds)
+\int_0^t\fint_0^{2\pi}(C_s(\th)-\tilde C_s)\L_s(\th)d\th ds.
\end{align*}
We have, for $c\cV_t\le N$,
$$
|\tilde C_t-ce^{-\a\tilde\cT_{t-}}|\le Cc^2
$$
and, for $t\le T_0(\d_0)$, as in the proof of Lemma \ref{BEST},
\begin{align*}
|C_t(\th)-ce^{-\a\cT_{t-}}(1+\a\re\hat\Psi_{t-}'(e^{\s+i\th}))|
&\le Cc\d_0^2,\\
|c\L_t(\th)-e^{-\eta\cT_{t-}}(1+\eta\re\hat\Psi_{t-}'(e^{\s+i\th}))|
&\le C\d_0^2,\\
|C_t(\th)\L_t(\th)-e^{-\z\cT_{t-}}(1+\z\re\hat\Psi_{t-}'(e^{\s+i\th}))|
&\le C\d_0^2
\end{align*}
so
$$
|C_t(\th)-\tilde C_t|\le Cc|\cT_{t-}-\tilde\cT_{t-}|+Cc\d_0
$$
and, using \eqref{PSINTZ},
$$
\left|\fint_0^{2\pi}C_t(\th)\L_t(\th)d\th-e^{-\z\cT_{t-}}\right|\le C\d_0^2
$$
and
$$
\left|\fint_0^{2\pi}\tilde C_t\L_t(\th)d\th-e^{-\a\tilde\cT_{t-}}e^{-\eta\cT_{t-}}\right|\le C(c+\d_0^2)
$$
and so
$$
\left|\fint_0^{2\pi}(C_t(\th)-\tilde C_t)\L_t(\th)d\th\right|\le C|\cT_{t-}-\tilde\cT_{t-}|+C(c+\d_0^2).
$$
Set
$$
f(t)=\E\left(\sup_{s\le t\wedge T_0(\d_0),\,c\cV_s\le N}|\cT_s-\tilde\cT_s|^p\right)
$$
Then, by Burkholder's and Jensen's inequalities, for $p\ge2$, and all $t\le T$,
$$
f(t)\le C(c^p+\d_0^{2p})+C\int_0^tf(s)ds
$$
and the claimed estimate follows by Gronwall's lemma.
\end{proof}
\vfe

\subsection{Spatially-uniform high-probability estimates}
\label{sec:highprob}
We now pass from the $L^p$-estimates of the preceding section to pointwise estimates which hold with high probability
on the function $\hat\Psi_t(z)=\hat\Phi_t(z)-z$, uniformly in $t\in[0,T]$ and $|z|\ge r(c)$ as $c\to0$,
for a suitable function $r(c)$, which is specified in the next result, and tends to $1$ as $c\to0$.
In order to show the desired uniformity, 
we combine the usual $L^p$-tail estimate with suitable dissections of $[0,T]$ and $\{|z|\ge r(c)\}$,
choosing $p$ large to deal with an increasing number of terms as $c\to0$.
We see at the same time that the event $\{T_0(\d_0)>T\}$, to which our previous estimates were restricted,
is in fact an event of high probability as $c\to0$, thus closing the argument for convergence to a disk.

\begin{proposition}
\label{UDFA}
For all $\a,\eta\in\R$ with $\z=\a+\eta\le1$, 
all $\ve\in(0,1/2]$ and all $\nu\in(0,\ve/4]$, 
all $m\in\N$ and all $T<t_\z$, 
there is a constant $C(\a,\eta,\L,\ve,\nu,m,T)<\infty$ with the following property.
In the case $\z<1$, for all $c\le1/C$, for $e^\s\ge1+c^{1/2-\ve}$ and $\d_0=c^{1/2-\nu}e^\s/(e^\s-1)$, 
there is an event ${\O_0\sse\{T_0(\d_0)>T\}}$ of probability exceeding $1-c^m$ on which,
for all $t\le T$ and all $|z|=r\ge1+c^{1/2-\ve}$,
\begin{equation}
\label{SUPD}
|\hat\Psi_t(z)|
\le C\bigg(c^{1/2-\nu}+c^{1-4\nu}\left(\frac{e^\s}{e^\s-1}\right)^2\bigg)
\end{equation}
and
\begin{equation}
\label{SUPE}
|D\hat\Psi_t(z)|
\le\frac Cr\bigg(c^{1/2-\nu}\left(\frac r{r-1}\right)+c^{1-4\nu}\left(\frac{e^\s}{e^\s-1}\right)^2\bigg)
\end{equation}
and
\begin{equation}
\label{CPPO}
|\hat\Psi_t(z)-\hat\Pi_t(z)|
\le Cc^{3/4-2\nu}\left(\frac{e^\s}{e^\s-1}\right)^{1/2}
+Cc^{1-4\nu}\bigg(\left(\frac r{r-1}\right)+\left(\frac{e^\s}{e^\s-1}\right)^2\bigg).
\end{equation}
Moreover, in the case $\z=1$ with $\ve\in(0,1/5]$,
for all $c\le1/C$, for $e^\s\ge1+c^{1/5-\ve}$ and ${\d_0=c^{1/2-\nu}(e^\s/(e^\s-1))^{3/2}}$
there is an event ${\O_0\sse\{T_0(\d_0)>T\}}$ of probability exceeding $1-c^m$ on which,
for all $t\le T$ and all $|z|=r\ge1+c^{1/5-\ve}$,
$$
|\hat\Psi_t(z)|
\le C\bigg(c^{1/2-\nu}\left(\frac r{r-1}\right)^{1/2}+c^{1-4\nu}\left(\frac{e^\s}{e^\s-1}\right)^3\bigg)
$$
and
$$
|D\hat\Psi_t(z)|
\le\frac Cr\bigg(c^{1/2-\nu}\left(\frac r{r-1}\right)^{3/2}+c^{1-4\nu}\left(\frac{e^\s}{e^\s-1}\right)^3\left(\frac r{r-1}\right)\bigg)
$$
and
\begin{align}
\notag
|\hat\Psi_t(z)-\hat\Pi_t(z)|
&\le Cc^{3/4-\nu}\left(\frac{e^\s}{e^\s-1}\right)^{3/4}\left(\frac r{r-1}\right)^{1/2}\\
\label{CPPP}
&\q\q+Cc^{1-4\nu}\bigg(\left(\frac r{r-1}\right)^2
+\left(\frac{e^\s}{e^\s-1}\right)^{3/2}\left(\frac r{r-1}\right)^{3/2}+\left(\frac{e^\s}{e^\s-1}\right)^3\bigg).
\end{align}
\end{proposition}
\begin{proof}
We will give details for the case $\z\in[0,1)$. 
Some minor modifications are needed for the case $\z=1$ because of the weaker $L^p$-estimate \eqref{Z1} which applies in that case,
and also for the case $\z<0$.
These are left to the reader.

Fix $\a,\eta,\ve,\nu,m$ and $T$ as in the statement.
By adjusting the value of $\ve$, it will suffice to consider the case where  $e^\s\ge1+2c^{1/2-\ve}$, 
and to find an event ${\O_0\sse\{T_0(\d_0)>T\}}$, of probability exceeding $1-c^m$, 
on which \eqref{SUPD}, \eqref{SUPE} and \eqref{CPPO} hold whenever ${r\ge1+2c^{1/2-\ve}}$ and $t\le T$.
There is a constant $C<\infty$ of the desired dependence, such that $\d_0\le1/2$ whenever $c\le1/C$.
We restrict to such $c$.
Set 
$$
\d=c^{m+3},\q 
t(n)=\d n,\q
N=\lfloor T/\d\rfloor,\q
N_0=\lfloor(T_0(\d_0)\wedge T)/\d\rfloor.
$$
Recall that $\cV_t$ denotes the number of particles added to the cluster by time $t$.
Consider the event
$$
\O_1=\{\cV_{t(n)}-\cV_{t(n-1)}\le1\text{ for all $n\le N_0$ and }\cV_{t(N_0)}=\cV_{T_0(\d_0)\wedge T}\}.
$$
Note that, on $\O_1$, for all $t\le T_0(\d_0)\wedge T$, there exists $n\in\{1,\dots,N_0\}$ such that $\hat\Psi_t=\hat\Psi_{t(n)}$.
Since $\d_0\le1/2$, there is a constant $C<\infty$ of the desired dependence such that
the process $(\cV_t)_{t\le T_0(\d_0)}$ is a thinning of a Poisson process of rate $C/c$.
Hence
$$
\PP(\O_1^c)\le N(C/c)^2\d^2+(C/c)/\d\le C\d/c^2=Cc^{m+1}
$$
and hence $\PP(\O_1^c)\le c^m/3$ for all $c\le1/(3C)$.
We restrict to such $c$.

Fix an integer $p\ge2$, to be chosen later, depending on $m$ and $\nu$.
By Proposition \ref{CAPE}, there is a constant $C<\infty$ of the desired dependence such that, 
for $\mu_0=C\left(\sqrt c+\d_0^2\right)$, we have
$$
\big\|\sup_{t\le T_0(\d_0)\wedge T}|\Psi^\cp_t|\big\|_p\le\mu_0,\q
\big\|\sup_{t\le T_0(\d_0)\wedge T}|\hat\Psi_t(\infty)|\big\|_p\le\mu_0.
$$
Set $\l_0=(6c^{-m})^{1/p}$ and consider the event
$$
\O_2=\{|\Psi^\cp_t|\le\l_0\mu_0\text{ and }|\hat\Psi_t(\infty)|\le\l_0\mu_0\text{ for all }t\le T_0(\d_0)\wedge T\}.
$$
Then $\PP(\O_2^c)\le2\l_0^{-p}=c^m/3$.
We choose $p\ge m/\nu$.
Then, since $e^\s\ge1+2c^{1/2-\ve}$ and $\nu\le\ve$, 
there is a constant $C<\infty$ of the desired dependence such that, for $c\le1/C$,
on the event $\O_2$, for all $t\le T_0(\d_0)\wedge T$,
\begin{equation}
\label{RIT}
|\Psi^\cp_t|
\le\l_0\mu_0
\le Cc^{-\nu}(\sqrt c+\d_0^2)
=C\left(c^{1/2-\nu}+c^{1-3\nu}\left(\frac{e^\s}{e^\s-1}\right)^2\right)
\le\d_0.
\end{equation}
We restrict to such $c$.
Set
$$
K=\min\{k\ge1:2^{k}c^{1/2-\ve}\ge1\}.
$$ 
Then $K\le\lfloor\log(1/c)\rfloor+1$. 
For $k=1,\dots,K$, set 
$$
r(k)=1+2^{k}c^{1/2-\ve},\q\rho(k)=\frac{r(k)+1}2.
$$
Then $\rho(k)\ge\rho(1)=1+c^{1/2-\ve}$ for all $k$ and $r(K)\in[2,4]$. 
By Proposition \ref{II}, there is a constant $C<\infty$ of the desired dependence such that, for $k=1,\dots,K$ and all $t\le T$,
$$
\tn D\hat\Psi_t1_{\{t\le T_0(\d_0)\}}\tn_{p,\rho(k)}\le\mu(r(k))
$$
where 
$$
\mu(r)=\frac Cr\left(\sqrt c\left(\frac r{r-1}\right)+c^{1-3\nu}\left(\frac{e^\s}{e^\s-1}\right)^2\right).
$$
Set $\l=\left(3KTc^{-2m-3}\right)^{1/p}$ and consider the event
$$
\O_3=\bigcap_{n=1}^N\bigcap_{k=1}^K\{\|D\hat\Psi_{t(n)}\|_{p,\rho(k)}1_{\{t(n)\le T_0(\d_0)\}}\le\l\mu(r(k))\}.
$$
Then
$$
\PP(\|D\hat\Psi_{t(n)}\|_{p,\rho(k)}1_{\{t(n)\le T_0(\d_0)\}}>\l\mu(r(k)))\le\l^{-p}
$$
so
$$
\PP(\O_3^c)\le KN\l^{-p}\le KT\l^{-p}/\d=c^m/3.
$$ 
Fix $r\ge1+2c^{1/2-\ve}$. 
Then $r(k)\le r<r(k+1)$ for some $k\in\{1,\dots,K\}$, where we set $r(K+1)=\infty$.
Note that $zD\hat\Psi_t(z)$ is a bounded holomorphic function on $\{|z|>\rho(1)\}$.
We use the inequality \eqref{PINF} to see that, on the event $\O_3$, for $n\le N_0$,
\begin{align*}
r\|D\hat\Psi_{t(n)}\|_{\infty,r}
&\le r(k)\|D\hat\Psi_{t(n)}\|_{\infty,r(k)}\\
&\le\left(\frac{r(k)+1}{r(k)-1}\right)^{1/p}r(k)\|D\hat\Psi_{t(n)}\|_{p,\rho(k)}
\le(2c^{-1/2})^{1/p}r(k)\l\mu(k)
\end{align*}
so
$$
\|D\hat\Psi_{t(n)}\|_{\infty,r}
\le(2c^{-1/2})^{1/p}\l\mu(r(k))
\le2(2c^{-1/2})^{1/p}\l\mu(r).
$$
We choose $p\ge(2m+4)/\nu$.
Then there is a constant $C<\infty$ of the desired dependence such that, for $c\le1/C$, on $\O_3$, 
for $n=1,\dots,N_0$ and all $r\ge1+2c^{1/2-\ve}$, we have
$$
\|D\hat\Psi_{t(n)}\|_{\infty,r}
\le\frac Cr\left(c^{1/2-\nu}\left(\frac r{r-1}\right)+c^{1-4\nu}\left(\frac{e^\s}{e^\s-1}\right)^2\right)
$$
and
$$
\|\hat\Psi'_{t(n)}\|_{\infty,e^\s}\le c^{1/2-\nu}\left(\frac{e^\s}{e^\s-1}\right)=\d_0.
$$
We restrict to such $c$.
Set
$$
\O_0=\O_1\cap\O_2\cap\O_3.
$$
Then $\PP(\O_0^c)\le c^m$ and, on the event $\O_0$, for all $t\le T_0(\d_0)\wedge T$ and all $r\ge1+2c^{1/2-\ve}$, 
$$
\|D\hat\Psi_t\|_{\infty,r}
\le\frac Cr\left(c^{1/2-\nu}\left(\frac r{r-1}\right)+c^{1-4\nu}\left(\frac{e^\s}{e^\s-1}\right)^2\right)
$$
and
$$
\|\hat\Psi'_t\|_{\infty,e^\s}\le\d_0.
$$
In conjunction with \eqref{RIT}, this forces $T_0(\d_0)>T$ on $\O_0$ and so concludes the proof of \eqref{SUPE}.

We deduce \eqref{SUPD} using the identity
$$
\psi(z)=\psi(\infty)-\int_1^\infty D\psi(sz)s^{-1}ds.
$$
On the event $\O_2$, for all $t\le T_0(\d_0)\wedge T$,
$$
|\hat\Psi_t(\infty)|\le C\left(c^{1/2-\nu}+c^{1-3\nu}\left(\frac{e^\s}{e^\s-1}\right)^2\right).
$$
On the other hand, $\O_0\sse\O_2$ 
and on $\O_0$ we have $T_0(\d_0)>T$ and, using \eqref{SUPE}, 
for $t\le T$ and $|z|=r\ge1+c^{1/2-\ve}$,
\begin{align*}
\int_1^\infty|D\hat\Psi_t(sz)|s^{-1}ds
&\le\int_1^\infty
\frac C{rs}\left(c^{1/2-\nu}\left(\frac{sr}{sr-1}\right)+c^{1-4\nu}\left(\frac{e^\s}{e^\s-1}\right)^2\right)s^{-1}ds\\
&\le\frac Cr\left(c^{1/2-\nu}\left(1+\log\left(\frac r{r-1}\right)\right)+c^{1-4\nu}\left(\frac{e^\s}{e^\s-1}\right)^2\right).
\end{align*}
Since $r\ge1+c^{1/2}$, the $\log$ factor can be absorbed in $c^{1/2-\nu}$ by adjustment of $\nu$.
Then, on combining the last two estimates, we obtain \eqref{SUPD}.

For the estimate \eqref{CPPO}, define
$$
\tilde\mu_0=C(c+\sqrt{c\d_0}+\d_0^2)
$$
where $C$ is the constant in Proposition \ref{CAPE}, and define
$$
\tilde\mu(r)
=\frac Cr\left(c\left(\frac r{r-1}\right)^2+\sqrt{c\d_0}\left(\frac r{r-1}\right)+\d_0^2\left(1+\log\left(\frac r{r-1}\right)\right)\right).
$$
where $C$ is the constant of Proposition \ref{III}. Set $\tilde\O_0=\O_1\cap\tilde\O_2\cap\tilde\O_3$, where
$$
\tilde\O_2=\O_2\cap\{|\hat\Psi_t(\infty)-\hat\Pi_t(\infty)|\le\l_0\tilde\mu_0\text{ for all }t\le T_0(\d_0)\wedge T\}
$$
and
$$
\tilde\O_3=\O_3\cap\bigcap_{n=1}^N\bigcap_{k=1}^K\{\|D(\hat\Psi_{t(n)}-\hat\Pi_{t(n)})\|_{p,\rho(k)}1_{\{t(n)\le T_0(\d_0)\}}\le\l\tilde\mu(r(k))\}.
$$
We follow a similar argument to above to see that $\PP(\tilde\O_0^c)\le2c^m$ and on $\tilde\O_0$ we have $T_0(\d_0)>T$ and
for $t\le T$ 
$$
|\hat\Psi_t(\infty)-\hat\Pi_t(\infty)|\le C\left(
c^{3/4-3\nu/2}\left(\frac{e^\s}{e^\s-1}\right)^{1/2}
+c^{1-3\nu}\left(\frac{e^\s}{e^\s-1}\right)^2\right)
$$
and for $|z|=r\ge1+2c^{1/2-\ve}$,
$$
\|D(\hat\Psi_t-\hat\Pi_t)\|_{\infty,r}
\le\frac Cr\left(
c^{1-\nu}\left(\frac r{r-1}\right)^2
+c^{3/4-3\nu/2}\left(\frac{e^\s}{e^\s-1}\right)^{1/2}\left(\frac r{r-1}\right)
+c^{1-4\nu}\left(\frac{e^\s}{e^\s-1}\right)^2\right).
$$
Finally we can integrate as above to deduce \eqref{CPPO}.

\def\j{
\bb MODIFICATION FOR $\z=1$.
Fix $\a,\eta,\ve,\nu,m$ and $T$ as in the statement.
By adjusting the value of $\ve$, it will suffice to consider the case where  $e^\s\ge1+2c^{1/5-\ve}$, 
and to find an event ${\O_0\sse\{T_0(\d_0)>T\}}$, of probability exceeding $1-c^m$, 
on which the claimed estimates hold whenever ${r\ge1+2c^{1/5-\ve}}$ and $t\le T$.
There is a constant $C<\infty$ of the desired dependence, such that $\d_0\le1/2$ whenever $c\le1/C$.
We restrict to such $c$.
Set 
$$
\d=c^{m+3},\q 
t(n)=\d n,\q
N=\lfloor T/\d\rfloor,\q
N_0=\lfloor(T_0(\d_0)\wedge T)/\d\rfloor.
$$
Write $\cV_t$ for the number of particles added to the cluster by time $t$.
Consider the event
$$
\O_1=\{\cV_{t(n)}-\cV_{t(n-1)}\le1\text{ for all $n\le N_0$ and }\cV_{t(N_0)}=\cV_{T_0(\d_0)\wedge T}\}.
$$
Note that, on $\O_1$, for all $t\le T_0(\d_0)\wedge T$, there exists $n\in\{1,\dots,N_0\}$ such that $\hat\Psi_t=\hat\Psi_{t(n)}$.
Since $\d_0\le1/2$, there is a constant $C<\infty$ of the desired dependence such that
the process $(\cV_t)_{t\le T_0(\d_0)}$ is a thinning of a Poisson process of rate $C/c$.
Hence
$$
\PP(\O_1^c)\le N(C/c)^2\d^2+(C/c)/\d\le C\d/c^2=Cc^{m+1}
$$
and hence $\PP(\O_1^c)\le c^m/3$ for all $c\le1/(3C)$.
We restrict to such $c$.

Fix an integer $p\ge2$, to be chosen later, depending on $m$ and $\nu$.
By Proposition \ref{CAPE}, there is a constant $C<\infty$ of the desired dependence such that, 
for $\mu_0=C\left(\sqrt c+\d_0^2\right)$, we have
$$
\big\|\sup_{t\le T_0(\d_0)\wedge T}|\Psi^\cp_t|\big\|_p\le\mu_0,\q
\big\|\sup_{t\le T_0(\d_0)\wedge T}|\hat\Psi_t(\infty)|\big\|_p\le\mu_0.
$$
Set $\l_0=(6c^{-m})^{1/p}$ and consider the event
$$
\O_2=\{|\Psi^\cp_t|\le\l_0\mu_0\text{ and }|\hat\Psi_t(\infty)|\le\l_0\mu_0\text{ for all }t\le T_0(\d_0)\wedge T\}.
$$
Then $\PP(\O_2^c)\le2\l_0^{-p}=c^m/3$.
We choose $p\ge m/\nu$.
Then, since $e^\s\ge1+2c^{1/5-\ve}$ and $\nu\le\ve$, 
there is a constant $C<\infty$ of the desired dependence such that, for $c\le1/C$,
on the event $\O_2$, for all $t\le T_0(\d_0)\wedge T$,
\begin{equation}
\label{RITT}
|\Psi^\cp_t|
\le\l_0\mu_0
\le Cc^{-\nu}(\sqrt c+\d_0^2)
=C\left(c^{1/2-\nu}+c^{1-3\nu}\left(\frac{e^\s}{e^\s-1}\right)^3\right)
\le\d_0.
\end{equation}
We restrict to such $c$.
Set
$$
K=\min\{k\ge1:2^{k}c^{1/2-\ve}\ge1\}.
$$ 
Then $K\le\lfloor\log(1/c)\rfloor+1$. 
For $k=1,\dots,K$, set 
$$
r(k)=1+2^{k}c^{1/2-\ve},\q\rho(k)=\frac{r(k)+1}2.
$$
Then $\rho(k)\ge\rho(1)=1+c^{1/5-\ve}$ for all $k$ and $r(K)\in[2,4]$. 
By Proposition \ref{II}, there is a constant $C<\infty$ of the desired dependence such that, for $k=1,\dots,K$ and all $t\le T$,
$$
\tn D\hat\Psi_t1_{\{t\le T_0(\d_0)\}}\tn_{p,\rho(k)}\le\mu(r(k))
$$
where 
$$
\mu(r)=\frac Cr\left(\sqrt c\left(\frac r{r-1}\right)^{3/2}+c^{1-2\nu}\left(\frac{e^\s}{e^\s-1}\right)^3\left(\frac r{r-1}\right)\right).
$$
Set $\l=\left(3KTc^{-2m-3}\right)^{1/p}$ and consider the event
$$
\O_3=\bigcap_{n=1}^N\bigcap_{k=1}^K\{\|D\hat\Psi_{t(n)}\|_{p,\rho(k)}1_{\{t(n)\le T_0(\d_0)\}}\le\l\mu(r(k))\}.
$$
Then
$$
\PP(\|D\hat\Psi_{t(n)}\|_{p,\rho(k)}1_{\{t(n)\le T_0(\d_0)\}}>\l\mu(r(k)))\le\l^{-p}
$$
so
$$
\PP(\O_3^c)\le KN\l^{-p}\le KT\l^{-p}/\d=c^m/3.
$$ 
Fix $r\ge1+2c^{1/5-\ve}$. 
Then $r(k)\le r<r(k+1)$ for some $k\in\{1,\dots,K\}$, where we set $r(K+1)=\infty$.
Note that $zD\hat\Psi_t(z)$ is a bounded holomorphic function on $\{|z|>\rho(1)\}$.
We use the inequality \eqref{PINF} to see that, on the event $\O_3$, for $n\le N_0$,
\begin{align*}
r\|D\hat\Psi_{t(n)}\|_{\infty,r}
&\le r(k)\|D\hat\Psi_{t(n)}\|_{\infty,r(k)}\\
&\le\left(\frac{r(k)+1}{r(k)-1}\right)^{1/p}r(k)\|D\hat\Psi_{t(n)}\|_{p,\rho(k)}
\le(2c^{-1/2})^{1/p}r(k)\l\mu(k)
\end{align*}
so
$$
\|D\hat\Psi_{t(n)}\|_{\infty,r}
\le(2c^{-1/2})^{1/p}\l\mu(r(k))
\le2(2c^{-1/2})^{1/p}\l\mu(r).
$$
We choose $p\ge(2m+4)/\nu$.
Then there is a constant $C<\infty$ of the desired dependence such that, for $c\le1/C$, on $\O_3$, 
for $n=1,\dots,N_0$ and all $r\ge1+2c^{1/5-\ve}$, we have
$$
\|D\hat\Psi_{t(n)}\|_{\infty,r}
\le\frac Cr\left(c^{1/2-\nu}\left(\frac r{r-1}\right)^{3/2}+c^{1-3\nu}\left(\frac{e^\s}{e^\s-1}\right)^3\left(\frac r{r-1}\right)\right)
$$
and
$$
\|\hat\Psi'_{t(n)}\|_{\infty,e^\s}\le c^{1/2-\nu}\left(\frac{e^\s}{e^\s-1}\right)^{3/2}=\d_0.
$$
We restrict to such $c$.
Set
$$
\O_0=\O_1\cap\O_2\cap\O_3.
$$
Then $\PP(\O_0^c)\le c^m$ and, on the event $\O_0$, for all $t\le T_0(\d_0)\wedge T$ and all $r\ge1+2c^{1/5-\ve}$, 
$$
\|D\hat\Psi_t\|_{\infty,r}
\le\frac Cr\left(c^{1/2-\nu}\left(\frac r{r-1}\right)^{3/2}+c^{1-3\nu}\left(\frac{e^\s}{e^\s-1}\right)^3\right)
$$
and
$$
\|\hat\Psi'_t\|_{\infty,e^\s}\le\d_0.
$$
In conjunction with \eqref{RITT}, this forces $T_0(\d_0)>T$ on $\O_0$ and concludes the proof of our estimate for $D\hat\Psi$.
We deduce the estimate for $\hat\Psi$ by integration, as above.

For the estimate on $\hat\Psi-\hat\Pi$, define
$$
\tilde\mu_0=C(c+\sqrt{c\d_0}+\d_0^2)
$$
where $C$ is the constant in Proposition \ref{CAPE}, and define
\begin{align*}
\tilde\mu(r)
&=\frac Cr\left(\d_0^2\left(\frac r{r-1}\right)+\sqrt{c\d_0}\left(\frac r{r-1}\right)^{3/2}+c\left(\frac r{r-1}\right)^3\right.\\
&\q\q+\left.\bigg(\d_0^3\left(\frac r{r-1}\right)^2+\d_0\sqrt c\left(\frac r{r-1}\right)^{5/2}\bigg)\left(1+\log\left(\frac r{r-1}\right)\right)
\right).
\end{align*}
where $C$ is the constant of Proposition \ref{III}. 
Set $\tilde\O_0=\O_1\cap\tilde\O_2\cap\tilde\O_3$, where
$$
\tilde\O_2=\O_2\cap\{|\hat\Psi_t(\infty)-\hat\Pi_t(\infty)|\le\l_0\tilde\mu_0\text{ for all }t\le T_0(\d_0)\wedge T\}
$$
and
$$
\tilde\O_3=\O_3\cap\bigcap_{n=1}^N\bigcap_{k=1}^K\{\|D(\hat\Psi_{t(n)}-\hat\Pi_{t(n)})\|_{p,\rho(k)}1_{\{t(n)\le T_0(\d_0)\}}\le\l\tilde\mu(r(k))\}.
$$
We follow a similar argument to above to see that $\PP(\tilde\O_0^c)\le2c^m$ and on $\tilde\O_0$ we have $T_0(\d_0)>T$ and
for $t\le T$ 
$$
|\hat\Psi_t(\infty)-\hat\Pi_t(\infty)|\le C\left(
c^{3/4-3\nu/2}\left(\frac{e^\s}{e^\s-1}\right)^{3/4}
+c^{1-3\nu}\left(\frac{e^\s}{e^\s-1}\right)^3\right)
$$
and for $|z|=r\ge1+2c^{1/5-\ve}$,
\begin{align*}
&\|D(\hat\Psi_t-\hat\Pi_t)\|_{\infty,r}
=\frac Cr
\left(
c^{1-3\nu}
\left(\frac{e^\s}{e^\s-1}\right)^3
\left(\frac r{r-1}\right)
+c^{3/4-3\nu/2}
\left(\frac{e^\s}{e^\s-1}\right)^{3/4}
\left(\frac r{r-1}\right)^{3/2}
\right.\\
&\q\q
+\left.
c^{1-\nu}\left(\frac r{r-1}\right)^3
+c^{3/2-4\nu}\left(\frac{e^\s}{e^\s-1}\right)^{9/2}\left(\frac r{r-1}\right)^2
+c^{1-2\nu}\left(\frac{e^\s}{e^\s-1}\right)^{3/2}\left(\frac r{r-1}\right)^{5/2}
\right).
\end{align*}
Finally we can integrate as above to deduce the estimate on $\hat\Psi-\hat\Pi$.
\eb
}\jj
\end{proof}

It is now straightforward to deduce the following high-probability estimates from Proposition \ref{CAPE}
using $L^p$-tail estimates and the fact that $\PP(\O_0)\ge1-c^m$ from Proposition \ref{UDFA}.
The details are left to the reader.
Propositions \ref{UDFA} and \ref{UDFB} together imply Theorem \ref{DISK}.

\begin{proposition}
\label{UDFB}
For all $\a,\eta\in\R$ with $\z=\a+\eta\le1$, all $\ve\in(0,1/2]$ and all $\nu\in(0,\ve/4]$, 
all $m\in\N$ and all $T<t_\z$, there is a constant $C(\a,\eta,\L,\ve,\nu,m,T)<\infty$ with the following property.
In the case $\z<1$, for all $c\le1/C$, for $e^\s\ge1+c^{1/2-\ve}$, with probability exceeding $1-c^m$, for all $t\le T$,
\begin{equation*}
|\Psi^\cp_t(z)|
\le C\bigg(c^{1/2-\nu}+c^{1-4\nu}\left(\frac{e^\s}{e^\s-1}\right)^2\bigg)
\end{equation*}
and
\begin{equation}
\label{CPPQ}
|\Psi^\cp_t(z)-\Pi^\cp_t(z)|
\le C\bigg(c^{3/4-2\nu}\left(\frac{e^\s}{e^\s-1}\right)^{1/2}+c^{1-4\nu}\left(\frac{e^\s}{e^\s-1}\right)^2\bigg).
\end{equation}
Moreover, in the case $\z=1$ with $\ve\in(0,1/5]$, for all $c\le1/C$, for $e^\s\ge1+c^{1/5-\ve}$,
with probability exceeding $1-c^m$, for all $t\le T$,
$$
|\Psi^\cp_t(z)|
\le C\bigg(c^{1/2-\nu}+c^{1-4\nu}\left(\frac{e^\s}{e^\s-1}\right)^3\bigg)
$$
and
\begin{equation}
\label{CPPR}
|\Psi^\cp_t(z)-\Pi^\cp_t(z)|
\le C\bigg(c^{3/4-\nu}\left(\frac{e^\s}{e^\s-1}\right)^{3/4}+c^{1-4\nu}\left(\frac{e^\s}{e^\s-1}\right)^3\bigg).
\end{equation}
\end{proposition}

\begin{proof}[Proof of Theorem \ref{DISC}]
We will write the argument for the case $\z<1$, omitting the modifications needed for $\z=1$, which are left to the reader.
Since $N<n_\a$, we can choose $\d=\d(\a,\eta,N)>0$ and $T<t_\z$ such that $\nu_T=N+\d$.
Choose $\d_0$ and $\O_0$ as in Proposition \ref{UDFA}, with the choice of $T$ just made.
Write $C$ for the constant appearing in Proposition \ref{UDFA} and set
$$
\Delta=C\left(c^{1/2-\nu}+c^{1-4\nu}\left(\frac{e^\s}{e^\s-1}\right)^2\right).
$$
Then, for all $|z|\ge1+c^{1/2-\ve}$ and all $t\le T$, on the event $\O_0$, we have $|\hat\Phi_t(z)-z|\le\Delta$.
Then, by Propositions \ref{DTSA} and \ref{DTSB}, 
choosing $\d_0$ as in Proposition \ref{UDFA} and using an $L^p$-tail estimate for suitably large $p$, 
there is an event $\O_1\sse\O_0$, of probability exceeding $1-2c^m$, on which, for all $t\le T$, both
$|c\cV_t-\nu_t|\le\Delta$ and, provided $c\cV_t\le N$, also
$$
|\cT_t-\tilde\cT_t|\le Cc^{1-4\nu}\left(\frac{e^\s}{e^\s-1}\right)^2.
$$
We can choose $C$ so that, for $c\le1/C$, we have $\Delta\le\d$, so $c\cV_T\ge N+\d-\Delta\ge N$ always on $\O_1$.
Now, for all $n\le N/c$, we have $\cV_t=n$ for some $t\le T$ with $c\cV_t\le N$, 
so on $\O_1$, for all $|z|\ge1+c^{1/2-\ve}$, we have
$$
|\hat\Phi_n^\disc(z)-z|\le\Delta,\q
|\cT^\disc_n-\t^\disc_n|\le Cc^{1-4\nu}\left(\frac{e^\s}{e^\s-1}\right)^2.
$$
\end{proof}

\section{Fluctuation scaling limit for ALE$(\a,\eta)$} 
\label{sec:fluctuations}
Given an ALE$(\a,\eta)$ process $(\Phi_t)_{t\ge0}$, recall that
$$
\hat\Phi_t(z)=\Phi_t(z)/e^{\cT_t},\q \cT_t=\log\Phi_t'(\infty).
$$
The fluctuations in these coordinates are given by
$$
\hat\Psi_t(z)=\hat\Phi_t(z)-z,\q\Psi^\cp_t=\cT_t-\t_t,\q\t_t=\z^{-1}\log(1+\z t).
$$
Recall that we write $\cH$ for the set of holomorphic functions on $\{|z|>1\}$ which are bounded at $\infty$,
and we use on $\cH$ the topology of uniform convergence on $\{|z|\ge r\}$ for all $r>1$.
In this section we prove Theorem \ref{FLUCT} and then, at the end, we deduce Theorem \ref{DISCFLUCT}.

\subsection{Reduction to Poisson integrals}
Our starting point is the interpolation formula \eqref{KID}
$$
\hat\Psi_t(z)=\hat M_t(z)+\hat A_t(z),\q\Psi^\cp_t=M^\cp_t+A^\cp_t.
$$
As a first step, we study the approximations $\hat\Pi_t(z)$ and $\Pi^\cp_t$ to $\hat M_t(z)$ and $M^\cp_t$ which have a simple form
and which prove to be the dominant terms in the considered limit.
Set
$$
H(\th,z)=\frac z{e^{-i\th}z-1}=\sum_{k=0}^\infty e^{i(k+1)\th}z^{-k}.
$$
Recall the multiplier operator $P(\d)$ defined at \eqref{PDEF}.
Then
$$
P(\d)H(\th,z)=\sum_{k=0}^\infty e^{-q(k)\d}e^{i(k+1)\th}z^{-k}.
$$
Recall that $c_t=ce^{-\a\t_t}$ and $\l_t=c^{-1}e^{-\eta\t_t}$, and that we define for $|z|>1$
\begin{align}
\label{YMPS}
\hat\Pi_t(z)
&=\int_{E(t)}e^{-(\t_t-\t_s)}P(\t_t-\t_s)H(\th,z)2c_s1_{\{v\le\l_s\}}\tilde\mu(d\th,dv,ds),\\
\label{YMPT}
\Pi^\cp_t
&=\int_{E(t)}e^{-\z(\t_t-\t_s)}c_s1_{\{v\le\l_s\}}\tilde\mu(d\th,dv,ds).
\end{align}
The following result allows us to deduce the weak limit of the normalized fluctuations 
from that of the Poisson integrals $(\hat\Pi_t,\Pi^\cp_t)_{t\ge0}$.

\begin{proposition}
\label{ONLYM}
For all $\a,\eta\in\R$ with $\z=\a+\eta\le1$, 
$$
c^{-1/2}(\hat\Psi_t-\hat\Pi_t,\Psi^\cp_t-\Pi^\cp_t)\to0
$$ 
in $\cH\times\R$ uniformly on compacts in $[0,t_\z)$, in probability, 
in the limit $c\to0$ and $\s\to0$ considered in Theorem \ref{FLUCT}.
\end{proposition}
\begin{proof}
In Theorem \ref{FLUCT}, for $\z<1$, we restrict to $\s\ge c^{1/4-\ve}$ and take $\d_0=c^{1/2-\nu}e^\s/(e^\s-1)$ with $\nu\le\ve/4$.
On the other hand, for $\z=1$, we restrict to $\s\ge c^{1/6-\ve}$ and take $\d_0=c^{1/2-\nu}(e^\s/(e^\s-1))^{3/2}$.
In both cases, the right-hand sides in 
\eqref{CPPO}, \eqref{CPPP}, \eqref{CPPQ} and \eqref{CPPR} 
are therefore small compared to $\sqrt c$ in the considered limit.
The claim thus follows from Propositions \ref{UDFA} and \ref{UDFB}.
\end{proof}

Since the integral \eqref{YMPS} converges absolutely for all $\o$, 
we can exchange limits to see that
$$
\hat\Pi_t(z)=\sum_{k=0}^\infty\Pi_t(k)z^{-k}
$$
where
\begin{equation*}
\label{MPR}
\Pi_t(k)=2\int_{E(t)}e^{-(1+q(k))(\t_t-\t_s)}e^{i(k+1)\th}c_s1_{\{v\le\l_s\}}\tilde\mu(d\th,dv,ds).
\end{equation*}
Recall that we set $q_0(k)=(1-\z)k$ and define, for all $\z\in(-\infty,1]$ and $t<t_\z$,
\begin{equation*}
\label{MPRA}
\Pi^0_t(k)=2\int_{E(t)}e^{-(1+q_0(k))(\t_t-\t_s)}e^{i(k+1)\th}c_s1_{\{v\le\l_s\}}\tilde\mu(d\th,dv,ds).
\end{equation*}

\begin{proposition}
\label{TIGHT}
For all $\a,\eta\in\R$ with $\a+\eta=\z\le1$, and all $t<t_\z$, there is a constant $C(\a,\eta,t)<\infty$ such that, 
for all $k\ge0$,
$$
\Big\|\sup_{s\le t}|\Pi_s(k)|\Big\|_2\le C\sqrt c,\q
\Big\|\sup_{s\le t}|\Pi^\cp_s|\Big\|_2\le C\sqrt c
$$
and
$$
\Big\|\sup_{s\le t}|\Pi_s(k)-\Pi^0_s(k)|\Big\|_2\le Ck^2\s\sqrt c.
$$
Moreover, $C$ may be chosen so that, for all $h\in[0,1]$ and all stopping times $T\le t-h$,
$$
\|\Pi^0_{T+h}(k)-\Pi^0_T(k)\|_2\le C\sqrt c(\sqrt h+kh),\q
\|\Pi^\cp_{T+h}-\Pi^\cp_T\|_2\le C\sqrt{ch}.
$$
\end{proposition}
\begin{proof}
The estimates for $(\Pi^\cp_t)_{t<t_\z}$ are standard and are left to the reader.
For $(\Pi_t(k))_{t<t_\z}$, we use time-dissection to obtain estimates with good dependence on $k$.
Set $\k=1+q(k)$ and define
$$
M_t(k)
=e^{\k\t_t}\Pi_t(k)
=\int_{E(t)}e^{\k\t_s}e^{i(k+1)\th}2c_s1_{\{v\le\l_s\}}\tilde\mu(d\th,dv,ds).
$$
Set $n=\lceil\k\t_t\rceil$ and $t(n)=t$.
Set $t(i)=i/\k$ for $i=0,1,\dots,n-1$. 
Then $t(i+1)-t(i)\le1/\k$ for all $i$.
We have
$$
\E(|M_t(k)|^2)
=4\int_0^t e^{2\k\t_s}c_s^2\l_sds
\le Cc\int_0^te^{2\k\t_s}\dot\t_sds
\le Cce^{2\k\t_t}/\k
$$
so, by Doob's $L^2$-inequality,
$$
\Big\|\sup_{s\le t}|M_s(k)|\Big\|_2
\le Ce^{\k\t_t}\sqrt{c/\k}.
$$
Now, for $t(i)\le s\le t(i+1)$,
$$
|\Pi_s(k)|\le e^{-\k\t_{t(i)}}|M_s(k)|
$$
so
$$
\Big\|\sup_{t(i)\le s\le t(i+1)}|\Pi_s(k)|\Big\|_2
\le Ce^{-\k\t_{t(i)}}\Big\|\sup_{s\le t(i+1)}|M_s(k)|\Big\|_2
\le C\sqrt{c/\k}
$$
and so
\begin{equation}
\label{PIC}
\Big\|\sup_{s\le t}|\Pi_s(k)|\Big\|_2\le C\sqrt c.
\end{equation}
For the second estimate, set $\k_0=1+(1-\z)k$ and note that
$$
0\le|\k-\k_0|=|\z|k(1-e^{-\s(k+1)})\le|\z|k(k+1)\s.
$$
Restrict for now to the case $\z\ge0$, when $\k\ge\k_0$, and define
$$
M^0_t(k)
=\int_{E(t)}e^{\k_0\t_s}e^{i(k+1)\th}2c_s1_{\{v\le\l_s\}}\tilde\mu(d\th,dv,ds)
$$
and
$$
\tilde M_t(k)
=M_t(k)-M_t^0(k)
=\int_{E(t)}(e^{\k\t_s}-e^{\k_0\t_s})e^{i(k+1)\th}2c_s1_{\{v\le\l_s\}}\tilde\mu(d\th,dv,ds).
$$
Note that
$$
0\le e^{\k\t_s}-e^{\k_0\t_s}\le(\k-\k_0)\t_se^{\k\t_s}
$$
so, by a similar argument,
$$
\Big\|\sup_{s\le t}|e^{-\k\t_s}\tilde M_s(k)|\Big\|_2
\le C(\k-\k_0)\sqrt c.
$$
Now
$$
\Pi_s(k)-\Pi^0_s(k)
=e^{-\k\t_s}\tilde M_s(k)+(e^{-\k\t_s}-e^{-\k_0\t_s})M_s^0(k)
$$
so
$$
|\Pi_s(k)-\Pi^0_s(k)|
\le e^{-\k\t_s}|\tilde M_s(k)|+(\k-\k_0)\t_s|\Pi^0_s(k)|
$$
and so
$$
\Big\|\sup_{s\le t}|\Pi_s(k)-\Pi^0_s(k)|\Big\|_2
\le C(\k-\k_0)\sqrt c\le Ck^2\s\sqrt c.
$$
For $\Pi^0_s(k)$, we used the estimate \eqref{PIC} with $\k$ replaced by $\k_0$, which is the special case $\s=0$.
A similar argument holds in the case $\z<0$, with the roles of $\k$ and $\k_0$ interchanged, which leads to the same estimate.
It remains to show the third estimate, which we will do for general $\s\ge0$.
We have
\begin{align*}
\Pi_{T+h}(k)-\Pi_T(k)
&=e^{-\k\t_{T+h}}M_{T+h}(k)-e^{-\k\t_T}M_T(k)\\
&=e^{-\k\t_{T+h}}\tilde M_h(k)-
(e^{-\k(\t_{T+h}-\t_T)}-1)\Pi_T(k)
\end{align*}
where we redefine
$$
\tilde M_h(k)
=M_{T+h}(k)-M_T(k)
=\int_{E(T+h)\sm E(T)}e^{\k\t_s}e^{i(k+1)\th}2c_s1_{\{v\le\l_s\}}\tilde\mu(d\th,dv,ds).
$$
Now
$$
\E(|\tilde M_h(k)|^2|T)=4\int_T^{T+h}e^{2\k\t_s}c_s^2\l_sds
$$
so
$$
\E(|e^{-\k\t_{T+h}}\tilde M_h(k)|^2)\le Cch.
$$
On the other hand, since $T\le t$,
$$
\|(e^{-\k(\t_{T+h}-\t_T)}-1)\Pi_T(k)\|_2
\le C\k h\Big\|\sup_{s\le t}|\Pi_s(k)|\Big\|_2
\le C(k+1)h\sqrt c.
$$
The claimed estimate follows.
\end{proof}

\subsection{Gaussian limit process}
\label{GLP}
By Proposition \ref{ONLYM}, in order to compute the weak limit of $c^{-1/2}(\hat\Psi_t,\Psi^\cp_t)_{t<t_\z}$,
it suffices to compute the weak limit of $c^{-1/2}(\hat\Pi_t,\Pi^\cp_t)_{t<t_\z}$.
This process is a deterministic linear function of the compensated Poisson random measure $\tilde\mu$.
We are guided to find the weak limit process by replacing $\tilde\mu$ in \eqref{YMPS} and \eqref{YMPT} by a Gaussian white noise 
on $[0,2\pi)\times[0,\infty)\times(0,\infty)$ of the same intensity.
At the same time, we set $\s=0$ in the limit\footnote{%
It is not necessary to pass to the limit $\s\to0$.
Indeed, the best Gaussian approximation for given $\s>0$ would be obtained using $P$ instead of $P_0$.
The limit $c\to0$ with $\s$ fixed then holds uniformly in $\s$, subject to the restrictions stated in Theorem \ref{FLUCT},
and the limit processes for $\s$ fixed converge weakly to the case $\s=0$.
We have stated only the joint limit, since this seems to us of main interest, 
and since the limit fluctuations have in this case a slightly simpler form.}%
, replacing the multiplier operator $P(\d)$ by $P_0(\d)$.
Then, using the scaling properties of white noise, 
we arrive at candidate limit processes $(\hat\G_t(z))_{t<t_\z}$ and $(\G^\cp_t)_{t<t_\z}$ which are defined as follows.
Let $W$ be a Gaussian white noise on $[0,2\pi)\times(0,\infty)$ of intensity $(2\pi)^{-1}d\th dt$.
Define for each $|z|>1$ and $t\in[0,t_\z)$
\begin{align*}
\label{LSBP}
\hat\G_t(z)
&=2\int_0^t\int_0^{2\pi}e^{-(\t_t-\t_s)}P_0(\t_t-\t_s)H(\th,z)e^{-(\a+\eta/2)\t_s}W(d\th,ds),\\
\notag
\G^\cp_t
&=\int_0^t\int_0^{2\pi}e^{-\z(\t_t-\t_s)}e^{-(\a+\eta/2)\t_s}W(d\th,ds)
\end{align*}
where these Gaussian integrals are understood by the usual $L^2$ isometry.
Define for $t\ge0$ and $k\ge0$
$$
B_t(k)=\sqrt 2\int_0^t\int_0^{2\pi}e^{i(k+1)\th}W(d\th,ds),\q
B_t=\int_0^t\int_0^{2\pi}W(d\th,ds).
$$
We can and do choose versions of $(B_t(k))_{t\ge0}$ and $(B_t)_{t\ge0}$ which are continuous in $t$.
Then $(B_t(k))_{t\ge0}$ is a complex Brownian motion for all $k$, 
$(B_t)_{t\ge0}$ is a real Brownian motion, and all these processes are independent.
Note that, almost surely, for all $t<t_\z$,
$$
\G^\cp_t=\int_0^te^{-\z(\t_t-\t_s)}e^{-(\a+\eta/2)\t_s}dB_s.
$$
Define for $t\in[0,t_\z)$ and $k\ge0$
\begin{equation*}
\G_t(k)=\sqrt2\int_0^te^{-(1-\z)k(\t_t-\t_s)}e^{-(\a+\eta/2)\t_s}dB_s(k).
\end{equation*}
The following estimate may be obtained by (a simpler version of) the argument used for Proposition \ref{TIGHT}.
\begin{proposition}
\label{DIB}
For all $\a,\eta\in\R$ with $\a+\eta=\z\le1$, and all $t<t_\z$, 
there is a constant $C(\a,\eta,t)<\infty$ such that, for all $k\ge0$,
$$
\Big\|\sup_{s\le t}|\G_s(k)|\Big\|_2\le C.
$$
\end{proposition}

The following identity holds in $L^2$ for all $|z|>1$ and $t<t_\z$
\begin{equation}
\label{LSBM}
\hat\G_t(z)=\sum_{k=0}^\infty\G_t(k)z^{-k}.
\end{equation}
By Proposition \ref{DIB}, almost surely, the right-hand side in \eqref{LSBM} converges uniformly
on compacts in $[0,t_\z)$, uniformly on $\{|z|\ge r\}$, for all $r>1$.
So we can and do use \eqref{LSBM} to choose a version of $\hat\G_t(z)$ for each $t<t_\z$
and $|z|>1$ such that $(\hat\G_t)_{t<t_\z}$ is a continuous process in $\cH$ and
\eqref{LSBM} holds for all $\o$.

The processes $(\G_t(k))_{t<t_\z}$ and $(\G^\cp_t)_{t<t_\z}$ are also characterized by the following Ornstein--Uhlenbeck-type 
stochastic differential equations
\begin{align*}
d\G_t(k)&=e^{-\a\t_t}\left(\sqrt2e^{-\eta\t_t/2}dB_t(k)-(1+(1-\z)k)\G_t(k)e^{-\eta\t_t}dt\right),\q\G_0(k)=0,\\
d\G^\cp_t&=e^{-\a\t_t}\left(e^{-\eta\t_t/2}dB_t-\z\G^\cp_te^{-\eta\t_t}dt\right),\q\G^\cp_0=0.
\end{align*}
These equations can be put in a simpler form by switching to the time-scale
$$
\nu_t=\int_0^te^{-\eta\t_s}ds
$$
which arises as the limit as $c\to0$ of a time-scale where particles arrive at a constant rate.
Write $\nu\mapsto t(\nu):[0,n_\a)\to[0,t_\z)$ for the inverse map and set
$$
\tilde\G_\nu(k)=\G_{t(\nu)}(k),\q
\tilde\G^\cp_\nu=\G^\cp_{t(\nu)}
$$
and
$$
\tilde\t_\nu=\t_{t(\nu)},\q
\tilde B_\nu(k)=\int_0^{t(\nu)}e^{-\eta\t_s/2}dB_s(k),\q
\tilde B_\nu=\int_0^{t(\nu)}e^{-\eta\t_s/2}dB_s.
$$
Then $e^{-\a\tilde\t_\nu}=(1+\a\nu)^{-1}$.
Also $(\tilde B_\nu(k))_{\nu<n_\a}$ is a complex Brownian motion for all $k$, 
$(\tilde B_\nu)_{\nu<n_\a}$ is a real Brownian motion, and these processes are independent.
Then we have
\begin{align}
\label{GKE}
d\tilde\G_\nu(k)&=(1+\a\nu)^{-1}\left(\sqrt2d\tilde B_\nu(k)-(1+(1-\z)k)\tilde\G_\nu(k)d\nu\right),\q\tilde\G_0(k)=0,\\
\notag
d\tilde\G^\cp_\nu&=(1+\a\nu)^{-1}\left(d\tilde B_\nu-\z\tilde\G^\cp_\nu d\nu\right),\q\tilde\G^\cp_0=0.
\end{align}
We can define a Brownian motion $(\tilde B_\nu)_{\nu<n_\a}$ in $\cH$ by
$$
\tilde B_\nu(z)=\sum_{k=0}^\infty\tilde B_\nu(k)z^{-k}.
$$
Set
$$
\tilde\G_\nu(z)=\sum_{k=0}^\infty\tilde\G_\nu(k)z^{-k}=\hat\G_{t(\nu)}(z).
$$
On summing the equations \eqref{GKE}, we see that $(\tilde\G_\nu)_{\nu<n_\a}$ satisfies
the following stochastic integral equation in $\cH$
$$
\tilde\G_\nu(z)=\int_0^\nu\frac{\sqrt2d\tilde B_s(z)-\tilde\G_s(z)ds+(1-\z)D\tilde\G_s(z)ds}{1+\a s}.
$$

\subsection{Convergence}
Given Proposition \ref{ONLYM}, the following result will complete the proof of Theorem \ref{FLUCT}.

\begin{proposition}
\label{WEC}
For all $\a,\eta\in\R$ with $\a+\eta=\z\le1$ and all $T<t_\z$, we have
$$
c^{-1/2}(\hat\Pi_t,\Pi^\cp_t)_{t\ge0}\to(\hat\G_t,\G^\cp_t)_{t\le T}
$$
weakly in ${{D([0,T],\cH\times\R)}}$ as $c\to0$ and $\s\to0$ as in Theorem \ref{FLUCT}. 
\end{proposition}
\begin{proof}
By Proposition \ref{TIGHT}, 
it will suffice to show the claimed limit with $(\hat\Pi_t)_{t\le T}$ replaced by $(\hat\Pi^0_t)_{t\le T}$.
We first show that
$$
c^{-1/2}((\Pi^0_t(k):k\ge0),\Pi^\cp_t)_{t\le T}\to((\G_t(k):k\ge0),\G^\cp_t)_{t\le T}
$$
in the sense of finite-dimensional distributions.
For all $n\ge1$, all $k_1,\dots,k_n\ge0$ and all $t_1,\dots,t_n\le T$, 
any real-linear function of $c^{-1/2}(\Pi^0_{t_j}(k_j),\Pi^\cp_{t_j}:j=1,\dots,n)$ can be written in the form
$$
F=\int_{E(T)}\tilde f_t(\th)1_{\{v\le\l_t\}}\tilde\mu(d\th,dv,dt)
$$
where
$$
\tilde f_t(\th)=c^{-1/2}f_t(\th)c_t
$$
and $(\th,t)\mapsto f_t(\th):[0,2\pi)\times(0,T]\to\R$ is bounded, measurable and independent of $c$.
Set
$$
\s_t^2=\fint_0^{2\pi}f_t(\th)^2d\th.
$$
The same linear function applied to $(\G_{t_j}(k_j),\G^\cp_{t_j}:j=1,\dots,n)$ gives the random variable
$$
G=\int_0^T\int_0^{2\pi}\tilde f_t(\th)\l_t^{1/2}W(d\th,dt)=\int_0^T\int_0^{2\pi}f_t(\th)e^{-(\a+\eta/2)\t_t}W(d\th,dt).
$$
Then
$$
\E(F^2)
=\E(G^2)
=\int_0^T\fint_0^{2\pi}\tilde f_t(\th)^2\l_td\th dt
=\int_0^Te^{-(2\a+\eta)\t_t}\s_t^2dt
$$
and, using the Campbell--Hardy formula, as $c\to0$,
\begin{align*}
\E(e^{iuF})
&=\exp\left(\int_0^T\fint_0^{2\pi}(e^{iu\tilde f_t(\th)}-1-iu\tilde f_t(\th))\l_td\th dt\right)\\
&\to\exp\left(-\frac{u^2}2\int_0^Te^{-(2\a+\eta)\t_t}\s_t^2dt\right)=
\E(e^{iuG}).
\end{align*}
The claimed convergence of finite-dimensional distributions follows, by convergence of characteristic functions.

Now, Proposition \ref{TIGHT} shows that the processes $(\Pi^0_t(k))_{t\le T}$ and $(\Pi^\cp_t)_{t\le T}$ 
all satisfy Aldous's tightness criterion in $D([0,T],\C)$.
Hence
$$
c^{-1/2}((\Pi^0_t(k):k\ge0),\Pi^\cp_t)_{t\le T}\to((\G_t(k):k\ge0),\G^\cp_t)_{t\le T}
$$
weakly in $D([0,T],\C^{\Z^+}\times\R)$ as $c\to0$.
Hence, for all $K\ge0$, 
$$
c^{-1/2}(p_K(\hat\Pi^0_t),\Pi^\cp_t)_{t\le T}\to(p_K(\hat\G_t),\G^\cp_t)_{t\le T}
$$
weakly in $D([0,T],\cH\times\R)$ as $c\to0$, where, for $f(z)=\sum_{k=0}^\infty a_kz^{-k}$,
$$
p_K(f)(z)=\sum_{k=0}^Ka_kz^{-k}.
$$
For $|z|=r$, we have
$$
|(f-p_K(f))(z)|\le\sum_{k=K+1}^\infty|a_k|r^{-k}.
$$
Hence, it will suffice to show, for $r>1$ and all $\ve>0$, that
$$
\lim_{K\to\infty}\limsup_{c\to0}\PP\left(c^{-1/2}\sup_{t\le T}\sum_{k=K+1}^\infty|\Pi^0_t(k)|r^{-k}>\ve\right)=0.
$$
But, since $\a+\eta=\z\le1$, by Proposition \ref{TIGHT}, there is a constant $C(\a,\eta,T)<\infty$ such that, for all $c>0$ and all $r>1$,
$$
\left\|c^{-1/2}\sup_{t\le T}\sum_{k=K}^\infty|\Pi^0_t(k)|r^{-k}\right\|_2\le\frac{Cr^{-K}}{r-1}.
$$
The desired limit follows.
\end{proof}

\begin{proof}[Proof of Theorem \ref{DISCFLUCT}]
We will argue via the Skorokhod representation theorem.
It will suffice to show the claimed convergence for all sequences $c_k\to0$ and $\s_k\to0$ subject to the 
constraint assumed in Theorem \ref{FLUCT}.
Given $N<n_\a$, choose $\d>0$ and $T<t_\z$ such that $\nu_T=N+\d$, as in the proof of Theorem \ref{DISC}.
By Theorem \ref{FLUCT} and Propositions \ref{DTSA} and \ref{UDFA}, and since $D([0,T],\cH)$ is a complete separable metric space, 
there is a probability space on which are defined a sequence of ALE$(\a,\eta)$ processes $(\Phi_t^{(k)})_{t\ge0}$, 
with common particle family $(P^{(c)}:c\in(0,\infty))$, and a Gaussian process $(\hat\G_t)_{t<t_\z}$ with the following properties:
\begin{enumerate}
\item[(a)]$(\Phi_t^{(k)})_{t\ge0}$ has capacity parameter $c_k$ and regularization parameter $\s_k$,
\item[(b)]$(\hat\G_t)_{t<t_\z}$ has the distribution of the limit Gaussian process in Theorem \ref{FLUCT},
\item[(c)]almost surely, as $k\to\infty$,
$$
\sup_{t\le T}|c\cV^{(k)}_t-\nu_t|\to0
$$
and, for all $r>1$, 
$$
\sup_{t\le T}\sup_{|z|\ge r}\left|c^{-1/2}\hat\Psi^{(k)}_t(z)-\hat\G_t(z)\right|\to0.
$$
\end{enumerate}
Here, $\cV^{(k)}_t$ denotes the number of particles added in $(\Phi_t^{(k)})_{t\ge0}$ by time $t$.
Define for $n\ge0$ and $\nu<n_\a$
$$
J^{(k)}_n=\inf\{t\ge0:\cV^{(k)}_t=n\},\q
t(\nu)=\z^{-1}((1+\a\nu)^{\z/\a}-1).
$$
From (c), we deduce that, almost surely, as $k\to\infty$,
$$
\sup_{n\le N/c}|J^{(k)}_n-t(cn)|\to0
$$
and, for $\nu\in[0,N]$ and $n=\lfloor\nu/c\rfloor$, the following limit holds in $\cH$
$$
c^{-1/2}\hat\Psi^{(k),\disc}_{\nu/c}
=c^{-1/2}\hat\Psi^{(k),\disc}_n
=c^{-1/2}\hat\Psi^{(k)}_{J^{(k)}_n}\to\hat\G_{t(\nu)}.
$$
But $(\hat\G_{t(\nu)})_{\nu<n_\a}$ has the same distribution as $(\hat\G^\disc_\nu)_{\nu<n_\a}$. 
Hence
$$
c^{-1/2}(\hat\Psi^{(k),\disc}_\nu)_{\nu\le N}\to(\hat\G^\disc_\nu)_{\nu\le N}
$$
weakly in $D([0,N],\cH)$. 
\end{proof}

\vfe
\section{Appendix}
\label{APP}
\subsection{Estimates for single-particle maps} 
\label{sec:particle}
Let $P$ be a basic particle and let 
$$
F(z)=e^c\left(z+\sum_{k=0}^\infty a_kz^{-k}\right)
$$
be the associated conformal map $D_0\to D_0\sm P$.
We assume that $F$ extends continuously to $\{|z|\ge1\}$.
Set
\begin{align*}
r_0&=r_0(P)=\sup\{|z|-1:z\in P\},\\
\d&=\d(P)=\inf\{r\ge0:|z-1|\le r\text{ for all }z\in P\}.
\end{align*}
We assume throughout that $\d\le1$.
We use the following well known estimates on the capacity $c$.
There is an absolute constant $C<\infty$ such that 
\begin{equation}
\label{FEST0}
r_0^2/C\le c\le C\d^2.
\end{equation}
The lower bound relies on Beurling's projection theorem and a comparison with the case of a slit particle.
The upper bound follows from a comparison with the case $P_\d=S_\d\cap D_0$,
where $S_\d$ is the closed disk whose boundary intersects the unit circle orthogonally at $e^{\pm i\th_\d}$
with $\th_\d\in[0,\pi]$ is determined by $|e^{i\th_\d}-1|=\d$.
See Pommerenke \cite{MR1217706}.

Write
$$
\log\left(\frac{F(z)}z\right)=u(z)+iv(z)
$$
where we understand the argument to be determined for each $z\in D_0$ so that the left-hand side is holomorphic in $D_0$
and such that $v(z)\to0$ as $z\to\infty$.
Then $u$ and $v$ are bounded and harmonic in $D_0$, with continuous extensions to $\{|z|\le1\}$, and $u(z)\to c$ as $z\to\infty$.
Note also that 
\begin{equation}
\label{USCQ}
0\le u(e^{i\th})\le\log(1+r_0)\le r_0\q\text{ for all $\th$}.
\end{equation}

\begin{lemma}
\label{VB}
Assume that $16\d\le\pi$. 
Then
\begin{equation}
\label{USCO}
u(e^{i\th})=0\q\text{whenever $|\th|\in[16\d,\pi]$}
\end{equation}
and
\begin{equation}
\label{USCP}
|v(e^{i\th})|\le16\d\q\text{ for all $\th$}.
\end{equation}
\end{lemma}
\begin{proof}
Set
$$
p_\d=\PP_\infty(B\text{ hits }S_\d\text{ before leaving }D_0)
$$
where $B$ is a complex Brownian motion.
Consider the conformal map $f$ of $D_0$ to the upper half-plane $\H$ given by 
$$
f(z)=i\frac{z-1}{z+1}.
$$
Set $b=f(e^{-i\th_\d})=\sin\th_\d/(1+\cos\th_\d)$.
Since $\d\le1$, we have 
\begin{equation}
\label{EEA}
\th_\d\le\d\pi/3
\end{equation}
and then $b\le2\pi\d/9$.
By conformal invariance,
$$
p_\d
=\PP_i(B\text{ hits $f(S_\d)$ before leaving $\H$})
=2\int_0^{2b/(1-b^2)}\frac{dx}{\pi(1+x^2)}.
$$
Hence 
\begin{equation}
\label{EEB}
p_\d\le4b/\pi\le8\d/9.
\end{equation}

Now $e^{i\pi}$ is not a limit point of $P$ so $e^{i\pi}=F(e^{i(\pi+\a)})$ for some $\a\in\R$.
Then $u(e^{i(\pi+\a)})=0$ and we can and do choose $\a$ so that $\a+v(e^{i(\pi+\a)})=0$.
Set
$$
\th^+=\sup\{\th\le\pi+\a:u(e^{i\th})>0\},\q
\th^-=\inf\{\th\ge\pi+\a:u(e^{i\th})>0\}-2\pi.
$$
Then $\th^-\le\th^+$.
We will show that $|\th^\pm|\le16\d$, which then implies \eqref{USCO}.
For $\th\in[\th^-,\th^+]$, we have $F(e^{i\th})\in S_\d$ so $|\th+v(e^{i\th})|\le\th_\d$.
Set $P^*=\{F(e^{i\th}):\th\in[\th^-,\th^+]\}$. 
Then $P^*\sse S_\d$ so, by conformal invariance,
$$
\frac{\th^+-\th^-}{2\pi}
=\PP_\infty(B\text{ hits $P^*$ on leaving $D_0\sm P$})
\le p_\d.
$$
On the other hand, for $\th,\th'\in[\th^+,\th^-+2\pi]$ with $\th\le\th'$, by conformal invariance,
$$
\frac{\th'-\th}{2\pi}
=\PP_\infty\big(B\text{ hits $\big[e^{i(\th+v(e^{i\th}))},e^{i(\th'+v(e^{i\th'}))}\big]$ on leaving $D_0\sm P$}\big)
\le\frac{\th'+v(e^{i\th'})-\th-v(e^{i\th})}{2\pi}
$$
so $v$ is non-decreasing on $[\th^+,\th^-+2\pi]$, and so
$$
\a+v(e^{i\th^+})
\le\a+v(e^{i(\pi+\a)})=0
\le\a+v(e^{i\th^-}).
$$
Hence
\begin{equation}
\label{EEC}
\th^+-\a
\le2\pi p_\d+\th^--\a
\le2\pi p_\d+\th_\d-v(e^{i\th^-})-\a
\le2\pi p_\d+\th_\d
\end{equation}
and similarly 
\begin{equation}
\label{EED}
\th^--\a\ge-2\pi p_\d-\th_\d.
\end{equation}
So we obtain, for all $\th\in[\th^-,\th^+]$,
\begin{equation}
\label{EEE}
|\a+v(e^{i\th})|\le2\th_\d+2\pi p_\d.
\end{equation}
Since $v$ is continuous and is non-decreasing on the complementary interval, this inequality then holds for all $\th$.
Now $v$ is bounded and harmonic in $D_0$ with limit $0$ at $\infty$,
so
$$
\int_0^{2\pi}v(e^{i\th})d\th=0.
$$
Hence
$$
|\a|
=\left|\fint_0^{2\pi}(\a+v(e^{i\th}))d\th\right|
\le2\th_\d+2\pi p_\d.
$$
On combining this with \eqref{EEC}, \eqref{EED} and \eqref{EEE}, we see that
$$
|\th^\pm|\le4\th_\d+4\pi p_\d,\q
|v(e^{i\th})|\le4\th_\d+4\pi p_\d\q\text{for all $\th$}.
$$
But $4\th_\d+4\pi p_\d\le44\pi\d/9\le16\d$ by \eqref{EEA} and \eqref{EEB}, so we have shown the claimed inequalities.
\end{proof}

\begin{proposition}
\label{PEST}
There is an absolute constant $C<\infty$ with the following properties.
In the case where $\d=\d(P)\le1/C$, for all $|z|>1$,
\begin{equation}
\label{FEST1}
\left|\log\left(\frac{F(z)}z\right)-c\right|\le\frac{C\d}{|z|}
\end{equation}
and, for all $|z|>1$ with $|z-1|\ge C\d$,
\begin{equation}
\label{FEST3}
\left|\log\left(\frac{F(z)}z\right)-c-\frac{2c}{z-1}\right|\le\frac{C\d c|z|}{|z-1|^2}
\end{equation}
and 
\begin{equation}
\label{FEST4}
|a_0-2c|\le C\d c
\end{equation}
and 
\begin{equation}
\label{FEST5}
\left|\log\left(\frac{F(z)}z\right)-c-\frac{a_0}{z-1}\right|\le\frac{C\d c}{|z-1|^2}.
\end{equation}
\end{proposition}
\begin{proof}
Since $z\log(F(z)/z)$ is bounded and holomorphic in $\{|z|>1\}$, 
\eqref{FEST1} follows from \eqref{USCQ} and \eqref{USCP} by the maximum principle.
The inequality \eqref{FEST4} follows from \eqref{FEST3} on letting $z\to\infty$, since $z(\log(F(z)/z)-c)\to a_0$.
Moreover, since $(z-1)^2(\log(F(z)/z)-c)-a_0z$ is bounded and holomorphic on $\{|z|>1\}$, 
\eqref{FEST5} follows from \eqref{FEST3} by the maximum principle, at the cost of replacing $C$ by $6C$, say.
We will show \eqref{FEST3} holds whenever $|z-1|\ge 3a$, where $a=16\d$.

Since $u$ is bounded and harmonic with $u(z)\to c$ as $z\to\infty$, we have
$$
\fint_0^{2\pi}u(e^{i\th})d\th=c
$$
and, for all $|z|>1$,
$$
u(z)
=\fint_0^{2\pi}u(e^{i\th})\re\left(\frac{z+e^{i\th}}{z-e^{i\th}}\right)d\th
=c+\fint_0^{2\pi}u(e^{i\th})\re\left(\frac{2e^{i\th}}{z-e^{i\th}}\right)d\th.
$$
Let $\a\in(-\pi,\pi]$ and $\rho>0$ be defined by 
$$
\fint_0^{2\pi}u(e^{i\th})e^{i\th}d\th=c\rho e^{i\a}.
$$
We use \eqref{USCO} to see that $|\a|\le a$ and $\rho\in[\cos a,1)$.
Now
$$
u(z)-c-\re\left(\frac{2\rho ce^{i\a}}{z-e^{i\a}}\right)
=\fint_0^{2\pi}u(e^{i\th})\re\left(\frac{2e^{i\th}}{z-e^{i\th}}-\frac{2e^{i\th}}{z-e^{i\a}}\right)d\th.
$$
For $|z-1|\ge2a$ and any $\th$ such that $u(e^{i\th})>0$, we have
$$
|z-e^{i\a}|\ge|z-1|/2,\q 
|z-e^{i\th}|\ge|z-1|/2,\q
|e^{i\th}-e^{i\a}|\le2a.
$$
Hence, for $|z|>1$ with $|z-1|\ge2a$,
$$
\left|u(z)-c-\re\left(\frac{2\rho ce^{i\a}}{z-e^{i\a}}\right)\right|
\le2\fint_0^{2\pi}u(e^{i\th})\frac{|e^{i\th}-e^{i\a}|}{|z-e^{i\th}||z-e^{i\a}|}d\th
\le\frac{16ac}{|z-1|^2}
$$
and
$$
\left|\frac2{z-1}-\frac{2\rho e^{i\a}}{z-e^{i\a}}\right|
\le\frac{2(1-\rho+|\rho e^{i\a}-1||z|)}{|z-1||z-e^{i\a}|}
\le\frac{12a|z|}{|z-1|^2}
$$
and hence
\begin{equation}
\label{UEST}
\left|u(z)-c-\re\left(\frac{2c}{z-1}\right)\right|
\le\frac{Cac|z|}{|z-1|^2}.
\end{equation}
We can extend $F$ to a holomorphic function in $\{|z-1|>a\}$ by setting $F(\bar z^{-1})=\widebar{F(z)}^{-1}$.
Then $u$ and $v$ also extend and it is straightforward to check that the estimate \eqref{UEST} remains valid 
for all $|z-1|\ge2a$.
Since $v(z)\to0$ as $z\to\infty$, a standard argument allows us to deduce from \eqref{UEST} that, for $|z-1|\ge3a$,
\begin{equation*}
\label{VEST}
\left|v(z)-\im\left(\frac{2c}{z-1}\right)\right|
\le\frac{Cac|z|}{|z-1|^2}
\end{equation*}
and hence 
$$
\left|\log\left(\frac{F(z)}z\right)-c-\frac{2c}{z-1}\right|
\le\frac{Cac|z|}{|z-1|^2}.
$$
\end{proof}

We sometimes use exponentiated versions of the inequalities just proved, 
which are straightforward to deduce and are noted here for easy reference.
There is an absolute constant $C<\infty$ with the following properties.
Suppose that $\d\le1/C$.
Then, for all $|z|>1$,
\begin{equation*}
\label{FEST6}
|e^{-c}F(z)-z|\le C\d
\end{equation*}
and, in the case $|z-1|\ge C\d$, 
\begin{equation}
\label{FEST7}
\left|e^{-c}F(z)-z-\frac{2cz}{z-1}\right|\le\frac{C\d c|z|^2}{|z-1|^2}
\end{equation}
and
\begin{equation*}
\label{FEST8}
\left|e^{-c}F(z)-z-\frac{a_0z}{z-1}\right|\le\frac{C\d c|z|}{|z-1|^2}.
\end{equation*}

\begin{proposition}
\label{PESTC}
There is an absolute constant $C<\infty$ with the following properties.
Let $P_1,P_2$ be basic particles with $P_1\sse P_2$.
For $i=1,2$, write $F_i$ for the associated conformal map $D_0\to D_0\sm P_i$ 
and write $c_i$ for the capacity of $P_i$.
Set $\d_i=\d(P_i)$ and $a_{0,i}=a_0(P_i)$ and set
$$
\ve_i(z)=\log\left(\frac{F_i(z)}z\right)-c_i-\frac{2c_i}{z-1},\q
\ve_{0,i}(z)=\log\left(\frac{F_i(z)}z\right)-c_i-\frac{a_{0,i}}{z-1}.
$$
Assume that $\d_2\le1/C$.
Then
\begin{equation}
\label{FEST9}
|a_{0,2}-a_{0,1}-2(c_2-c_1)|\le C\d_2(c_2-c_1)
\end{equation}
and, for all $|z|>1$ with $|z-1|\ge C\d_2$, 
\begin{equation}
\label{FEST10}
\left|\ve_1(z)-\ve_2(z)\right|\le\frac{C\d_2(c_2-c_1)|z|}{|z-1|^2}
\end{equation}
and
\begin{equation}
\label{FEST11}
\left|\ve_{0,1}(z)-\ve_{0,2}(z)\right|\le\frac{C\d_2(c_2-c_1)}{|z-1|^2}.
\end{equation}
\end{proposition}
\begin{proof}
The inequalities \eqref{FEST10} and \eqref{FEST11} follow from \eqref{FEST9} by the same argument used to deduce
\eqref{FEST4} and \eqref{FEST5} from \eqref{FEST3}.
Set $\tilde P=F_1^{-1}(P_2\sm P_1)$.
Write $\tilde F$ for the associated conformal map $D_0\to D_0\sm\tilde P$ and write $\tilde c$ for the capacity of $\tilde P$.
Then 
$$
F_2=F_1\circ\tilde F,\q c_2=c_1+\tilde c.
$$
Note that, for $z\in\tilde P$, we have $F_1(z)\in P_2$, so $|F_1(z)-1|\le\d_2$.
But $|e^{-c_1}F_1(z)-z|\le C\d_1$ for all $|z|>1$ and $c_1\le C\d_1^2$.
Hence $|z-1|\le C\d_2$ for all $z\in\tilde P$ and so 
$$
\tilde\d=\d(\tilde P)\le C\d_2.
$$
Hence, for $C$ sufficiently large and $\d_2\le1/C$, for all $|z|>1$ with $|z-1|\ge C\d_2$,
\begin{equation}
\label{TILDEF}
\left|\log\left(\frac{\tilde F(z)}z\right)-\tilde c-\frac{2\tilde c}{z-1}\right|\le\frac{C\d_2\tilde c|z|}{|z-1|^2}
\end{equation}
and in particular
\begin{equation}
\label{TILDEF2}
\left|\log\left(\frac{\tilde F(z)}z\right)\right|\le\frac{C\tilde c|z|}{|z-1|}.
\end{equation}
Set $z_t=z\exp(t\log(\tilde F(z)/z))$ and $f(t)=\log(F_1(z_t)/F_1(z))$.
Then
$$
\log\left(\frac{F_2(z)}{F_1(z)}\right)
=f(1)-f(0)
=\int_0^1\dot f(t)dt
=\log\left(\frac{\tilde F(z)}z\right)\int_0^1F_1'(z_t)\left(\frac{F_1(z_t)}{z_t}\right)^{-1}dt
$$
so
\begin{align*}
\ve_2(z)-\ve_1(z)
&=\log\left(\frac{F_2(z)}{F_1(z)}\right)
-\tilde c-\frac{2\tilde c}{z-1}\\
&=\log\left(\frac{\tilde F(z)}z\right)\int_0^1F_1'(z_t)\left(\frac{F_1(z_t)}{z_t}\right)^{-1}dt
-\tilde c-\frac{2\tilde c}{z-1}\\
&=\log\left(\frac{\tilde F(z)}z\right)-\tilde c-\frac{2\tilde c}{z-1}
+\log\left(\frac{\tilde F(z)}z\right)\int_0^1\left(F_1'(z_t)\left(\frac{F_1(z_t)}{z_t}\right)^{-1}-1\right)dt.
\end{align*}
Now $|\log(\tilde F(z)/z)|\le C\d_2$, so $|z_t-z|\le C\d_2$ for all $t$.
Hence, for $C$ sufficiently large and $|z-1|\ge C\d_2$, we have $|z_t-1|\ge C_0\d_1$ for all $t$, 
where $C_0$ is the constant from Proposition \ref{PEST}.
Then 
$$
|e^{-c_1}F_1(z_t)-z_t|\le C\d_1,\q
|e^{-c_1}F_1'(z_t)-1|\le\frac{C\d_1}{|z-1|}
$$
where we used Cauchy's integral formula for the second inequality, adjusting the value of $C$ if necessary.
On combining these estimates with \eqref{TILDEF} and \eqref{TILDEF2}, we see that
$$
|\ve_2(z)-\ve_1(z)|\le\frac{C\d_2(c_2-c_1)|z|}{|z-1|^2}
$$
as claimed.
\end{proof}

The following is a straightforward consequence of \eqref{FEST0} and Propositions \ref{PEST} and \ref{PESTC}.

\begin{proposition}
\label{CNC}
Let $(P^{(c)}:c\in(0,1])$ be a family of basic particles and suppose that the associated conformal maps $F_c$ are given by
$$
F_c(z)=e^c\left(z+\sum_{k=0}^\infty a_k(c)z^{-k}\right).
$$
Fix $\L\in[1,\infty)$ and assume that $\d(c)\le\L r_0(c)$ for all $c$. 
Then there is a constant $C(\L)<\infty$ such that, for all $c\le1/C$,
$$
|a_0(c)-2c|\le Cc^{3/2}
$$
and, for all $|z|>1$,
$$
\left|\log\left(\frac{F_c(z)}z\right)-c\right|
\le\frac{Cc}{|z-1|}
$$
and
$$
\left|\log\left(\frac{F_c(z)}z\right)-c-\frac{a_0(c)}{z-1}\right|\le\frac{Cc^{3/2}}{|z-1|^2}
$$
and
$$
\left|e^{-c}F_c(z)-z-\frac{a_0(c)z}{z-1}\right|\le\frac{Cc^{3/2}|z|}{|z-1|^2}.
$$
Moreover, if $(P^{(c)}:c\in(0,1])$ is nested, then $C$ may be chosen so that, for all $c_1,c_2\in(0,c]$,
$$
|(a_0(c_1)-2c_1)-(a_0(c_2)-2c_2)|\le Cc^{1/2}|c_1-c_2|
$$
and, for all $|z-1|\ge C\sqrt c$,
$$
\left|\left(\log\left(\frac{F_{c_1}(z)}z\right)-c_1\right)-\left(\log\left(\frac{F_{c_2}(z)}z\right)-c_2\right)\right|
\le\frac{C|c_1-c_2|}{|z-1|}
$$
and
$$
\left|
\left(\log\left(\frac{F_{c_1}(z)}z\right)-c_1-\frac{a_0(c_1)}{z-1}\right)
-\left(\log\left(\frac{F_{c_2}(z)}z\right)-c_2-\frac{a_0(c_2)}{z-1}\right)
\right|
\le\frac{C\sqrt c|c_1-c_2|}{|z-1|^2}
$$
and
$$
\left|\left(e^{-c_1}F_{c_1}(z)-z-\frac{a_0(c_1)z}{z-1}\right)-\left(e^{-c_2}F_{c_2}(z)-z-\frac{a_0(c_2)z}{z-1}\right)\right|
\le\frac{C\sqrt c|c_1-c_2||z|}{|z-1|^2}.
$$
\end{proposition}

For our final particle estimates, we use the following integral representation for the family of particle maps
\begin{equation*}
F_c(z)=z+\int_0^cDF_t(z)\int_0^{2\pi}\frac{z+e^{i\th}}{z-e^{i\th}}\mu_t(d\th)dt
\end{equation*}
for some measurable family of probability measures $(\mu_t:t\in(0,\infty))$, 
with $\mu_t$ supported on $\{\th:|e^{i\th}-1|\le \d(t)\}$ for all $t$.
This follows from our requirements that the particles $P^{(c)}$ have capacity $c$, 
are contained in $\{|z-1|\le\d(c)\}$ and are nested, by the Loewner--Kufarev theory.
Our condition \eqref{CONCENTRATED} and the inequality \eqref{FEST0} 
then give a constant $C(\L)<\infty$ such that 
\begin{equation}
\label{MUSP}
\supp\mu_t\sse\{\th:|e^{i\th}-1|\le C\sqrt t\}.
\end{equation}
Define holomorphic functions $L_t$ and $Q_t$ on $\{|z|>1\}$ by
\begin{equation}
\label{QTH}
L_t(z)=\int_0^{2\pi}l_t(\th,z)\mu_t(d\th),\q
Q_t(z)=\int_0^{2\pi}q_t(\th,z)\mu_t(d\th)
\end{equation}
where
$$
l_t(\th,z)=\left(D\log\left(\frac{F_t(z)}z\right)+1\right)\frac{z+e^{i\th}}{z-e^{i\th}},\q
q_t(\th,z)=DF_t(z)\frac{z+e^{i\th}}{z-e^{i\th}}-e^tz-\frac{2e^te^{i\th}z}{z-1}.
$$
Note that $l_t(\th,z)\to1$ and $q_t(\th,z)\to0$ as $z\to\infty$, uniformly in $\th$.
It is then straightforward to show the integral representations
$$
\log\left(\frac{F_c(z)}z\right)=\int_0^cL_t(z)dt,\q
e^c\left(e^{-c}F_c(z)-z-\frac{a_0(c)z}{z-1}\right)=\int_0^cQ_t(z)dt.
$$

\begin{proposition}
\label{LQP}
There is a constant $C(\L)<\infty$ with the following property.
For all $t\le1/C$ and all $|z|>1$,
$$
|L_t(z)|\le\frac{C|z|}{|z-1|}+\frac{C\sqrt t|z|}{|z-1|^2},\q
|Q_t(z)|\le\frac{C\sqrt t|z|}{|z-1|^2}.
$$
\end{proposition}
\begin{proof}
We give the details for the second estimate, leaving the first which is similar but simpler to the reader.
We split $q_t(\th,z)=g_t(\th,z)+h_t(\th,z)$, where
$$
g_t(\th,z)=(DF_t(z)-e^tz)\frac{z+e^{i\th}}{z-e^{i\th}},\q
h_t(\th,z)=e^tz\left(\frac{z+e^{i\th}}{z-e^{i\th}}-1-\frac{2e^{i\th}}{z-1}\right).
$$
Now
$$
\frac{z+e^{i\th}}{z-e^{i\th}}-1-\frac{2e^{i\th}}{z-1}=\frac{2e^{i\th}(e^{i\th}-1)}{(z-e^{i\th})(z-1)}
$$
so, on the support of $\mu_t$, we have, for $|z-1|\ge2C\sqrt t$,
$$
|h_t(\th,z)|\le\frac{2C e^t\sqrt t|z|}{|z-1|^2}
$$
where $C$ is the constant in \eqref{MUSP}.
On the other hand, we showed above that, for all $|z|>1$,
$$
|F_t(z)-e^tz|\le C\sqrt t
$$
and $F_t$ extends by reflection to a holomorphic function on $\{|z-1|>C\sqrt t\}$ satisfying the same inequality.
Hence, by Cauchy's integral formula, for $|z-1|\ge2C\sqrt t$,
$$
|DF_t(z)-e^tz|\le\frac{C\sqrt t|z|}{|z-1|}
$$
and so, for $\th$ in the support of $\mu_t$, 
$$
|g_t(\th,z)|\le\frac{C\sqrt t|z|}{|z-1|^2}.
$$
We have shown that, for all $|z-1|\ge C\sqrt t$,
\begin{equation*}
|Q_t(z)|\le\frac{C\sqrt t|z|}{|z-1|^2}.
\end{equation*}
\end{proof}

\subsection{Operator inequalities}
\label{sec:operators}
Recall that, for a measurable function $f$ on $\{|z|>1\}$, for $p\in[1,\infty)$ and $r>1$, we set
\begin{equation*} \label{prnorm}
\|f\|_{p,r}=\left(\fint_0^{2\pi}|f(re^{i\th})|^pd\th\right)^{1/p},\q
\|f\|_{\infty,r}=\sup_{\th\in[0,2\pi)}|f(re^{i\th})|.
\end{equation*}
Suppose that $f$ is holomorphic and is bounded at $\infty$.
It is standard that, for $\rho\in(1,r)$,
\begin{equation}\label{PINF}
\|f\|_{p,r}\le\|f\|_{p,\rho},\q 
\|f\|_{\infty,r}\le\left(\frac\rho{r-\rho}\right)^{1/p}\|f\|_{p,\rho}.
\end{equation}
Moreover, there is an absolute constant $C<\infty$ such that
\begin{equation}
\label{DEST}
\|Df\|_{p,r}\le\frac{C\rho}{r-\rho}\|f\|_{p,\rho}
\end{equation}
where $Df(z)=zf'(z)$. 
The function $f$ has a Laurent expansion
$$
f(z)=\sum_{k=0}^\infty f_kz^{-k}.
$$
Let $M$ be an operator which acts as multiplication by $m_k$ on the the $k$th Laurent coefficient.
Thus
$$
Mf(z)=\sum_{k=0}^\infty m_k f_kz^{-k}.
$$
Assume that there exists a finite constant $A>0$ such that, for all $k\ge0$, 
$$
|m_k|\le A
$$
and, for all integers $K\ge0$, 
$$
\sum_{k=2^K}^{2^{K+1}-1}|m_{k+1}-m_k|\le A.
$$
Then, by the Marcinkiewicz multiplier theorem \cite[Vol. II, Theorem 4.14]{MR933759}, for all $p\in(1,\infty)$, 
there is a constant $C=C(p)<\infty$ such that, for all $r>1$,
\begin{equation}
\label{MEST}
\|Mf\|_{p,r}\le CA\|f\|_{p,r}.
\end{equation}
We will use also the following estimate.

\begin{proposition}
\label{LFG}
Let $f$ and $g$ be holomorphic in $\{|z|>1\}$ and bounded at $\infty$.
Set 
$$
f_\th(z)=f(e^{-i\th}z).
$$
Let $M$ be a multiplier operator and set
$$
h(z)=\left(\fint_0^{2\pi}|M(f_\th.g)|^2(z)d\th\right)^{1/2}.
$$
Then, for all $r,\rho>1$, we have
\begin{equation*}
\|h\|_{p,r}\le\|M\|_{p,\rho\to r}\|g\|_{p,\rho}\|f\|_{2,\rho}
\end{equation*}
where
$$
\|M\|_{p,\rho\to r}=\sup\{\|Mf\|_{p,r}:\|f\|_{p,\rho}\le1\}.
$$
\end{proposition}
\begin{proof}
We can write
$$
f(z)=\sum_{k=0}^\infty f_kz^{-k},\q
g(z)=\sum_{k=0}^\infty g_kz^{-k},\q
Mf(z)=\sum_{k=0}^\infty m_k f_kz^{-k}.
$$
Then
$$
M(f_\th.g)(z)=\sum_{k=0}^\infty\sum_{j=0}^\infty m_{j+k}f_kg_je^{i\th k}z^{-(k+j)}
$$
so
$$
h(z)^2=\sum_{k=0}^\infty|f_k|^2|M(\tau_kg)(z)|^2
$$
where $\tau_kg(z)=z^{-k}g(z)$.
Hence
\begin{align*}
\|h\|_{p,r}
\|h^2\|_{p/2,r}
&\le\sum_{k=0}^\infty|f_k|^2\|M(\tau_kg)\|^2_{p,r}
\le\sum_{k=0}^\infty|f_k|^2\|M\|^2_{p,\rho\ra r}\|\tau_kg\|^2_{p,\rho}\\
&=\sum_{k=0}^\infty|f_k|^2\rho^{-2k}\|M\|^2_{p,\rho\ra r}\|g\|^2_{p,\rho}
=\|M\|^2_{p,\rho\ra r}\|f\|^2_{2,\rho}\|g\|^2_{p,\rho}.
\end{align*}
\end{proof}

\subsection{Push-out estimates}
Recall that we write $(\t_t)_{t<t_\z}$ for the capacity of the disk solution to the LK$(\z)$ equation of initial capacity $\t_0=0$,
and that $\dot\t_t=e^{-\z\t_t}$.
We have found several times an integral such as
$$
\int_0^t\left(\frac{r_s}{r_s-1}\right)^pds
$$
where, for $r>1$ and $p\ge1$, we set
$$
r_s=
\begin{cases}
e^{(1-\z)(\t_t-\t_s)}r,&\text{ if $\z\ge0$},\\
e^{\t_t-\t_s}r,&\text{ if $\z<0$}.
\end{cases}
$$
We record here an estimate on such integrals which improves on the obvious maximum bound $t(r/(r-1))^p$
by using the fact that $r_s$ pushes out as $s$ decreases from $t$, away from the singularity at $r=1$.

\begin{lemma}
\label{PUSHOUT}
For all $\z\in[0,1)$, all $r>1$ and all $p\ge2$, we have
$$
\int_0^t\left(\frac{r_s}{r_s-1}\right)^pds
\le\frac{2(1+t)}{1-\z}\left(\frac r{r-1}\right)^{p-1}
$$
and
$$
\int_0^t\left(\frac{r_s}{r_s-1}\right)ds
\le\frac{2(1+t)}{1-\z}\left(1+\log\left(\frac r{r-1}\right)\right).
$$
Moreover, the inequalities remain valid for $\z<0$ and $t<t_\z$ if the denominator $1-\z$ is removed from the right-hand sides.
\end{lemma}
\begin{proof}
The left-hand sides are increasing in $\z$, so the claim for $\z<0$ follows from the case $\z=0$. 
For $p>2$, we have
$$
\int_0^t\left(\frac{r_s}{r_s-1}\right)^pds
\le\left(\frac r{r-1}\right)^{p-2}\int_0^t\left(\frac{r_s}{r_s-1}\right)^2ds
$$
so, for the first inequality, it will suffice to consider the case $p=2$.
For $r\ge2$ it suffices to use the obvious maximum bound.
Suppose then that $r=r_t\in(1,2)$.
Choose $t_0\in[0,t]$ so that $r_{t_0}=2$ if possible, setting $t_0=0$ otherwise.
Then $r_{t_0}\le2$ and
$$
\int_0^{t_0}\left(\frac{r_s}{r_s-1}\right)^2ds\le2t\left(\frac r{r-1}\right).
$$
Note that
$$
\frac d{ds}\left(\frac1{r_s-1}\right)=(1-\z)e^{-\z\t_s}\frac{r_s}{(r_s-1)^2}
$$
so
$$
(1-\z)\int_{t_0}^t\left(\frac{r_s}{r_s-1}\right)^2ds
\le2e^{\z\t_t}\int_{t_0}^t(1-\z)e^{-\z\t_s}\frac{r_s}{(r_s-1)^2}ds
=2(1+\z t)\left(\frac1{r-1}-\frac1{r_{t_0}-1}\right).
$$
The first claimed inequality follows. 
A similar argument leads to the inequalities
$$
\int_0^{t_0}\left(\frac{r_s}{r_s-1}\right)ds\le2t,\q
(1-\z)\int_{t_0}^t\left(\frac{r_s}{r_s-1}\right)ds
\le(1+\z t)\log\left(\frac{r_{t_0}-1}{r-1}\right)
$$
from which the second claimed inequality follows.
\end{proof}

\bibliography{p}

\end{document}